\documentclass{amsart}
\usepackage[utf8]{inputenc}
\usepackage[asymmetric,vmargin=3cm,hmargin=3cm,marginpar=2cm]{geometry}
\usepackage{amsmath,amssymb,amsfonts,amsthm,amscd,graphicx,psfrag,epsfig,fge}
\usepackage{xcolor}
\usepackage[]{hyperref}
\usepackage{mathtools}
\usepackage{thmtools}
\usepackage{tikz-cd}

\usepackage{microtype}

\usepackage{scalerel}
\usepackage{stackengine,wasysym}

\newcommand\reallywidetilde[1]{\ThisStyle{%
  \setbox0=\hbox{$\SavedStyle#1$}%
  \stackengine{-0.1\LMpt}{$\SavedStyle#1$}{%
    \stretchto{\scaleto{\SavedStyle\mkern.2mu\AC}{.5150\wd0}}{.6\ht0}%
  }{O}{c}{F}{T}{S}%
}}

\newcommand\reallywidetildelow[1]{\ThisStyle{%
  \setbox0=\hbox{$\SavedStyle#1$}%
  \stackengine{-0.3\LMpt}{$\SavedStyle#1$}{%
    \stretchto{\scaleto{\SavedStyle\mkern.2mu\AC}{.5150\wd0}}{.4\ht0}%
  }{O}{c}{F}{T}{S}%
}}

\usepackage[T2A,T1]{fontenc}
\newcommand{\Dee}{\mbox{\usefont{T2A}{\rmdefault}{m}{n}\CYRD}}

\let\geq\geqslant
\let\leq\leqslant

\newcommand{\sm}{\smallsetminus}

\newcommand{\setsep}{:}
\newcommand{\smid}{\mkern1mu|\mkern1mu}

\usepackage{tikz}
\usepackage{pgfplots}
\pgfplotsset{compat=1.15}
\usepackage{mathrsfs}
\usetikzlibrary{arrows}
\definecolor{g1}{rgb}{0.4,0.4,0.4}
\definecolor{g2}{rgb}{0.55,0.55,0.55}
\definecolor{g3}{rgb}{0.7,0.7,0.7}

\def\vdotsB{\smash{\vdots}\smash[b]{\strut}}
\def\circB{\smash{\raisebox{-1pt}{\scriptsize$\circ$}}}
\newcommand{\Xupcirc}{\overset{\circB}{\Xup}}

\DeclareMathOperator{\gr}{gr}
\DeclareMathOperator{\Stab}{Stab}
\DeclareMathOperator{\res}{res}
\DeclareMathOperator{\Spec}{Spec}
\DeclareMathOperator{\Gal}{Gal}
\DeclareMathOperator{\adj}{adj}
\DeclareMathOperator{\Bup}{B}
\DeclareMathOperator{\St}{St}
\DeclareMathOperator{\STup}{ST}
\DeclareMathOperator{\Tor}{Tor}
\DeclareMathOperator{\Hom}{Hom}
\DeclareMathOperator{\Coker}{Coker}
\newcommand{\op}{\mathrm{op}}
\newcommand{\toprm}{\mathrm{top}}
\newcommand{\prim}{\textup{prim}}

\newcommand{\StH}{\St^2}
\newcommand{\StL}{\St^\infty}
\newcommand{\StB}{\St^B}

\DeclareMathOperator{\Lup}{L}
\DeclareMathOperator{\Sup}{S}
\DeclareMathOperator{\Mup}{M}
\DeclareMathOperator{\Xup}{X}
\newcommand{\Qc}{\mathcal{Q}}
\newcommand{\Cc}{\mathcal{C}}
\newcommand{\Pc}{\mathcal{P}}
\newcommand{\Sc}{\mathcal{S}}
\newcommand{\Tc}{\mathcal{T}}
\newcommand{\Ic}{\mathcal{I}}
\newcommand{\Lc}{\mathcal{L}}
\newcommand{\Hc}{\mathcal{H}}
\newcommand{\Dc}{\mathcal{D}}
\newcommand{\Fc}{\mathcal{F}}
\newcommand{\Gc}{\mathcal{G}}
\newcommand{\CupL}{\operatorname{C}^\Lc}

\DeclareMathOperator{\Sh}{Sh}

\DeclareMathOperator{\I}{I}
\DeclareMathOperator{\Li}{Li}
\DeclareMathOperator{\Cor}{Cor}
\DeclareMathOperator{\Log}{Log}

\DeclareMathOperator{\Id}{Id}
\DeclareMathOperator{\GL}{GL}
\DeclareMathOperator{\Ker}{Ker}

\DeclareMathOperator{\Ind}{Ind}

\DeclareMathOperator{\Sp}{Sp}

\newcommand{\dd}{\mathrm{d}}
\newcommand{\Tup}{\mathrm{T}}
\newcommand{\Pup}{\mathrm{P}}

\newcommand{\Lie}{\mathrm{Lie}}

\newcommand{\M}{\mathcal{M}}
\newcommand{\Q}{\mathbb{Q}}
\newcommand{\Z}{\mathbb{Z}}
\newcommand{\R}{\mathbb{R}}
\newcommand{\N}{\mathbb{N}}
\newcommand{\PP}{\mathbb{P}}
\newcommand{\LL}{\mathbb{L}}
\newcommand{\Sbb}{\mathbb{S}}
\newcommand{\TT}{\mathbb{T}}
\newcommand{\Cbb}{\mathbb{C}}

\newcommand{\Sfr}{\mathfrak{S}}

\newcommand{\VB}{\textup{VB}}
\newcommand{\Mod}[1]{\: \mathrm{mod} \: #1}

\newcommand{\fup}{\mathrm{f}}

\DeclareMathOperator{\Nm}{Nm}
\DeclareMathOperator{\sgn}{sgn}
\newcommand{\lra}{\longrightarrow}
\newcommand{\Kbar}{\,\overline{\!K}}

\DeclareMathOperator{\Kmap}{K}

\newcommand{\llp}{\mathopen{(\mkern-4mu(}}
\newcommand{\rrp}{\mathclose{)\mkern-4mu)}}

\theoremstyle{plain}
\newtheorem{theorem}{Theorem}
\theoremstyle{definition}
\newtheorem{remark}[theorem]{Remark}
\newtheorem{definition}[theorem]{Definition}
\theoremstyle{plain}
\newtheorem{proposition}[theorem]{Proposition}
\newtheorem{conjecture}[theorem]{Conjecture}
\newtheorem{corollary}[theorem]{Corollary}
\newtheorem{example}[theorem]{Example}
\newtheorem{lemma}[theorem]{Lemma}

\usepackage[inline]{enumitem}
\setlist[enumerate,1]{label = {\rm(\roman*)}, leftmargin=*}

\makeatletter
\renewcommand{\sectionautorefname}{\S{}\@gobble}
\makeatother
\def\equationautorefname~#1\null{(#1)\null}

\newcommand{\indsp}{\mkern0.5mu}

\newcommand{\llangle}{\langle\mkern-4mu\langle}
\newcommand{\rrangle}{\rangle\mkern-4mu\rangle}
\usepackage{stmaryrd}

\title[Multiple polylogarithms and the Steinberg module]{Multiple polylogarithms and the Steinberg module}


\author[Charlton]{Steven Charlton}
\address{Department of Mathematics and Computer Science, Division of Mathematics, University of Cologne,
Weyertal 86-90, 50931 Cologne, Germany}
\email{steven.charlton@uni-koeln.de}

\author[Radchenko]{Danylo Radchenko}
\address{Laboratoire Paul Painlev\'e, Universit\'e de Lille, F-59655 Villeneuve d'Ascq,	France}
\email{danradchenko@gmail.com}

\author[Rudenko]{Daniil Rudenko}
\address{Department of Mathematics, University of Chicago, 5801 S Ellis Ave, 60637 Chicago, IL, USA}
\email{rudenkodaniil@gmail.com}
\date{18 February 2026}

\begin{document}

\begin{abstract}
We establish a connection between multiple polylogarithms on a torus and 
the Steinberg module of $\Q$, and show that multiple polylogarithms of depth $d$ and weight $n$ can be expressed via a single function $\Li_{n-d+1,1,\dots,1}(x_1,x_2,\dots,x_d)$. Using this connection, we give a simple proof of the Bykovski\u{\i} theorem, explain the duality between multiple polylogarithms and iterated integrals, and provide a polylogarithmic interpretation of the conjectures of Rognes and Church--Farb--Putman. 
\end{abstract}

\maketitle

{
	\setcounter{tocdepth}{1}
	\tableofcontents
	\setcounter{tocdepth}{2}
}

\section{Introduction}\label{SectionIntroduction}
\subsection{Multiple polylogarithms and the Steinberg module}\label{SectionMPandSteinbergModule}
Multiple polylogarithms are multivalued complex analytic functions depending on positive integer parameters $n_1,\dots,n_d\in \N$. In the polydisc $|x_1|,|x_2|,\dots, |x_d| <1$ multiple polylogarithms are defined by the power series 
\begin{equation}\label{FormulaPolylogarithm}
\Li_{n_1,n_2,\dots, n_d}(x_1,x_2,\dots,x_d)=\sum_{0<m_1<m_2<\dots<m_d}\frac{x_1^{m_1} x_2^{m_2}\cdots x_d^{m_d}}{m_1^{n_1}m_2^{n_2}\cdots m_d^{n_d}}.
\end{equation}
The number $n=n_1+\dots+n_d$ is called the weight of the multiple polylogarithm, and the number $d$ is called its depth. In the form of the power series~\eqref{FormulaPolylogarithm}, multiple polylogarithms were introduced by Goncharov \cite{GoncharovMSRI,Gon98} in order to study \emph{multiple Dirichlet $L$-values}, which are given by $\Q$-linear combinations of multiple polylogarithms at roots of unity.

One of the most interesting characteristics of polylogarithms concerns depth reductions: a phenomenon whereby certain sums of multiple polylogarithms can be expressed through functions of lower depth. This phenomenon plays a key role in Goncharov's approach to Zagier's conjecture, see Dupont's \cite{Cle21} for an overview.  In \cite[Theorem 1.1]{Rud20} it is shown that any multiple polylogarithm of weight $n \geq  2$ can be expressed as a linear
combination of multiple polylogarithms of depth at most $\lfloor n/2 \rfloor$ and products of polylogarithms of lower weight. However there are still exponentially many different functions of weight $n$ and depth $\lfloor n/2\rfloor$. In~\cite[Conjecture~1]{GCRR22} it was conjectured that all polylogarithms of weight $n$ and depth $d$ are expressible essentially in terms of a single function $\Li_{n-d+1,1,\dots,1}$ of depth $d$. Our main result is the following theorem, proving this conjecture.

\begin{theorem}\label{TheoremMain}
A function $\Li_{n_1,n_2,\dots, n_d}(x_1,x_2,\dots,x_d)$ of weight $n$ and depth $d$ can be expressed as a polynomial with rational coefficients in multiple polylogarithms $\Li_{m-k+1, 1,\dots,1}$ of depth $k\leq d$ and weight $m\leq n$ whose  arguments are Laurent monomials in $\sqrt[N]{x_1}, \dots, \sqrt[N]{x_d}$ for sufficiently large $N$.
\end{theorem}
Here is an example of this type of identity in weight four and depth two: for $|x_1|,|x_2| < 1$, we have
   \begin{align*}
      \Li_{2,2}(x_1,x_2)= {} & 
       {-}4 \Li_{3,1}\Bigl({-}\frac{\sqrt{x_1}}{\sqrt{x_2}},x_2\Bigr)-4 \Li_{3,1}\Bigl(\frac{\sqrt{x_1}}{\sqrt{x_2}},x_2\Bigr)
      +4 \Li_{3,1}\Bigl({-}\frac{\sqrt{x_2}}{\sqrt{x_1}},x_1\Bigr)+4 \Li_{3,1}\Bigl(\frac{\sqrt{x_2}}{\sqrt{x_1}},x_1\Bigr)
      \\
      &
     {} + \Li_{3,1}(x_1,x_2)-\Li_{3,1}(x_2,x_1)
       -\Li_{3,1}\Bigl(\frac{x_2}{x_1},x_1\Bigr)-\frac{1}{2} \Li_4(x_1 x_2)+\Li_1(x_1) \Li_3(x_2).
   \end{align*}
Note that because the right hand side is Galois invariant, it does not depend on the choice of branches for the square roots $\sqrt{x_1}, \sqrt{x_2}$.

The word ``function'' in the statement of the theorem requires a comment. The multivaluedness of multiple polylogarithms makes identities hard to work with. There are several standard ways to circumvent this issue; we discuss them in \S \ref{SecPolylogsAsFunctionsAndMotives}. In the paper, we chose to work with so-called {\it formal Hopf algebra of multiple polylogarithms} $\mathcal{H}^\textup{f}$ introduced in \cite{CMRR24}. In \S\ref{SectionLin11} we give a precise statement of Theorem \ref{TheoremMain} in this setting. From polylogarithmic identities in $\mathcal{H}^\textup{f}$ one can obtain numerically verifiable identities using single-valued versions of the multiple polylogarithms \cite{CDG21}.

Theorem \ref{TheoremMain} is a consequence of a more general structural result, which describes a certain space of multiple polylogarithms on a torus via the Steinberg module. We will work with an algebraic torus  $\Tup_d=\bigl(\overline{\Q}^{\times}\bigr)^d$ over the field $\overline{\Q}$. Though our results would hold for an algebraic torus over $\mathbb{C}$, we prefer to work over $\overline{\Q}$ to be able to deduce consequences for motivic multiple polylogarithms, see the discussion in \S \ref{SecPolylogsAsFunctionsAndMotives}.

We will use the following notation related to algebraic tori \cite{Spr1998}. The ring of regular function on the torus is
\[
\overline{\Q}[\Tup_d]\cong\overline{\Q}[x_1^{\vphantom{-1}},x_1^{-1},\dots,x_d^{\vphantom{-1}},x_d^{-1}].
\]
A homomorphism of algebraic groups $\chi\colon\Tup_d\to\Tup_1$ is called a (rational) character of $\Tup_d$. We denote by $\Xup(\Tup_d)$ the character lattice of $\Tup_d$, i.e., the set of all rational characters of $\Tup_d$; it has a natural structure of an abelian group, isomorphic to $\Z^d$. The elements of $X(\Tup_d)$ are precisely the monomials $x_1^{a_1}\cdots x_d^{a_d}$ for some exponent vector $(a_1,\dots,a_d)\in\Z^d$. A matrix $A\in\Mup_d(\Z)$ with $\det(A)\ne0$ defines an isogeny $p_A\colon\Tup_d\to\Tup_d$ given by
    \[p_A(x_1,\dots,x_d)=(x_1^{a_{11}}x_2^{a_{21}}\cdots x_{d}^{a_{d1}},\dots,x_1^{a_{1d}}x_2^{a_{2d}}\cdots x_{d}^{a_{dd}}).\]
The degree of the isogeny $p_A$ is equal to $|\det(A)|$, so, in particular, for $A\in\GL_d(\Z)$ the morphism $p_A$ is an automorphism of $\Tup_d$.

 Next we give an informal description of a  space $\LL_{n}(\Tup_d)$ 
 whose elements are certain linear combinations of polylogarithms whose arguments are Laurent monomials in $\sqrt[N]{x_i}$. In this description, we will work with polylogarithms (\ref{FormulaPolylogarithm}) as if they were single-valued functions on a torus $\Tup_d$. The detailed definition of $\LL_{n}(\Tup_d)$ is spelled out in \S\ref{SectionModuleLn}.

Consider a $\Q$-algebra $A$ of functions on $\Tup_d$ generated by pushforwards of polylogarithms (\ref{FormulaPolylogarithm}) along isogenies $p\colon \Tup_d\lra \Tup_d.$ Elements 
\[\log(-x_i)=\Li_1(x_i^{-1})-\Li_1(x_i)
\]
lie in this algebra and generate an ideal $I$ in $A$. The space $\LL_{n}(\Tup_d)$  is a $\Q$-subspace of $A/I$ spanned by the following  pushforwards of polylogarithms:
\begin{equation} \label{FormulaPolylogarithmsRoots}
N^{n-d-1}\sum_{\substack{y_1^{N}=x_1 \\ \hspace{\widthof{$y_1^{N}$}} \vdotsB \hspace{\widthof{$x_d$}} \\ y_d^{N}=x_d}}
   \Li_{n_1,\dots,n_d}\left(\prod_{i=1}^d y_i^{a_{i1}},\dots,\prod_{i=1}^d y_i^{a_{id}} \right), \quad   (a_{ij})\in \Mup_{d}(\Z), \ N=\lvert\det(a_{ij})\rvert>0.
\end{equation}
Next, let $\gr_d^{\Dc}\LL_{n}(\Tup_d)$ be the quotient of $\LL_{n}(\Tup_d)$ by a similarly defined subspace generated by polylogarithms of depth lower than $d.$  In \S\ref{SectionModuleLn} and Definition \ref{DefLnDepth} we will see that $\gr_d^{\Dc}\LL_{n}(\Tup_d)$ is a $\GL_d(\Q)$-module generated by functions $\Li_{n_1,\dots, n_d}(x_1,\dots,x_d)$, $n_1+\dots+n_d=n$.

We now recall the notion of the Steinberg module $\St(V)$ of a $d$-dimensional vector space $V$ over a field $F$.  Let $\Tc_{V}$ be the Tits building of a vector space $V,$ i.e., the geometric realization of the poset of nonzero proper subspaces of $V$ ordered by inclusion. The Solomon--Tits theorem~\cite{Sol69} states that $\Tc_{V}$ is homotopy equivalent to a bouquet of $(d-2)$-spheres. The top reduced homology $\widetilde{H}_{d-2}(\Tc_V,\Q)$ is called the Steinberg module $\St(V)$ and is generated by elements $[v_1,\dots,v_d]$ for $v_i\in V$. The elements $[v_1,\dots,v_d]$ are skew-symmetric,  vanish for linearly dependent $v_1,\dots,v_d$, and satisfy the relation
\[
\sum_{i=0}^{d} (-1)^i[v_0,\dots,\widehat{v_i},\dots, v_{d}]=0
\] 
for any vectors $v_0,\dots,v_{d}\in V.$ We will sometimes use notation $\St_d(F)$ for $\St(F^d)$. The Steinberg module $\St(V)$ is equipped with a natural action of $\GL(V)$. Denote by $\Sbb^nV$ the $n$-th symmetric power of~$V$, which is also a $\GL(V)$-module. 

Our main result is the following theorem, which implies Theorem \ref{TheoremMain} by an elementary ``cut and paste'' argument. 

\begin{theorem}\label{TheoremMain2} Consider a torus $\Tup_d$ and put $V=\Xup(\Tup_d)\otimes\Q.$ Let $e_i\coloneqq x_i$, $i=1,\dots,d$, be the basis of the character lattice $\Xup(\Tup_d)$. For $n\geq 0$ there exists a unique isomorphism of $\GL(V)$-modules 
\begin{equation}
 \label{maptheoremmain2}
\STup\colon \gr_d^{\Dc}\LL_{n}(\Tup_d) \stackrel{\sim}{\lra} \St(V)\otimes\St(V)\otimes\Sbb^{n-d}V
\end{equation}
sending 
$\Li_{n_1,n_2,\dots, n_d}(x_1,x_2,\dots,x_d)$ to 
\[
 [e_d, e_d+e_{d-1},\dots, e_d+\dots+e_1] \otimes [e_d,e_{d-1},\dots,e_1] \otimes \prod_{i=1}^d \frac{e_i^{n_i-1}}{(n_i-1)!}.
\]
\end{theorem}

The appearance of the tensor square of the Steinberg module is explained by the Koszul duality between $\St(V)$ and $\St(V)\otimes \St(V)$ discovered by Miller--Nagpal--Patzt \cite{MNP18} and Miller--Patzt--Wilson \cite{MPW23}.
The surjectivity of the map $\STup$ follows from an elementary yet intriguing statement about $\St(V)\otimes\St(V)$, which is true for an arbitrary field, see Theorem~\ref{thm:Stn_li_basis}.  As a consequence, we obtain an isomorphism
\[
\bigl( \St_d(F)\otimes \St_d(F)\bigr)_{\GL_d(F)}\cong \Q \quad \text{for} \: d\geq 1.
\]
 We learned that this result was obtained prior to our work, see Galatius--Kupers--Randal-Williams \cite[Theorem C]{GKRW20}. The proof of the injectivity of the map $\STup$ is the technical heart of the paper and occupies \S\S \ref{SectionProofoftheoremMain}-\ref{SectionLemmaMain}. 

Here is a conjectural description of the inverse map $\STup^{-1}$. 
A \emph{lattice Aomoto polylogarithm} is a multivalued analytic function
\[
\Li_{n_1,\dots,n_d}(v_1,\dots,v_d;w_1,\dots,w_d;x) 
= \frac{\sgn(\det(v_1,\dots,v_d))}{\det(w_1,\dots,w_d)} \!\!
\sum_{\nu\in C\cap\Z^d} \frac{x^{\nu}} {\prod_{j=1}^{d} \bigl(w^{j}(\nu)\bigr)^{n_j}},
\] 
where $x=(x_1,\dots,x_d)\in \Tup_d$, $x^\nu=\prod_{j=1}^d x_j^{\nu_j}$, $C=\R_{>0}v_1+\dots+\R_{>0}v_d$, and $w^1,\dots,w^d$ is the basis of linear functionals dual to $w_1,\dots,w_d$. Clearly, for $w_i=e_i$ and $v_i=e_i+e_{i+1}+\dots+e_d$ we recover the multiple polylogarithm (\ref{FormulaPolylogarithm}).  This function is a toric analogue of the usual Aomoto polylogarithm (\cite{Aom82,BGSV90}).  We expect that the function $\Li_{n_1,\dots,n_d}(v_1,\dots,v_d;w_1,\dots,w_d;x)$ can be expressed as a $\Q$-linear combination of multiple polylogarithms (\ref{FormulaPolylogarithmsRoots}).  Such a presentation via multiple polylogarithms should allow one to construct a lift of the lattice Aomoto polylogarithm to $\LL_{n}(\Tup_d)$.  We conjecture that its projection to $\gr^{\Dc}_d\LL_{n}(\Tup_d)$  coincides with 
\[
\STup^{-1}\bigg([v_1,\dots,v_d]\otimes[w_1,\dots,w_d]\otimes \frac{w_1^{n_1-1}}{(n_1-1)!}
\cdots \frac{w_d^{n_d-1}}{(n_d-1)!}\bigg) .
\]
In \S\ref{sec:polylogdistributions} we will give evidence for this by constructing an analogue of the map $\STup^{-1}$ valued in a certain space of distribution on the real torus $\R^d/\Z^d$.

One of the most intriguing properties of polylogarithms is a certain self-duality studied by Goncharov,  which underlies his notion of a dihedral Lie coalgebra, see \cite[\S 3]{Gon98}. For instance, multiple polylogarithms satisfy two families of relations: shuffle relations and quasi-shuffle relations. The quasi-shuffle relations easily follow from the power series presentation of polylogarithms, while the shuffle relations are evident from the presentation of multiple polylogarithms as iterated integrals. The appearance of the tensor square of the Steinberg module in Theorem~\ref{TheoremMain2} explains this duality: in \S\S\ref{sec:steinbergLiI}--\ref{SectionDoubleShuffle} we define a natural involution $D$ on $\St(V)\otimes \St(V)$ which sends multiple polylogarithms to iterated integrals and  exchanges shuffle and quasi-shuffle relations, see Remark \ref{RemarkDualityGeneratingFunctions}.

In \S\S\ref{SteinbergAndMilnor}--\ref{InjectivitySteinbergMilnor}, we prove that the Steinberg module $\St_d(\Q)$ is isomorphic to a certain subquotient of the Milnor $K$-group of the function field of the torus $\Tup_d$. Under the differential symbol, elements of this $K$-group are mapped to linear combinations of wedge products of differential forms
\[
\dd\log(1-\zeta \! \sqrt[N]{x_1}^{\,a_1}\!\cdots \! \sqrt[N]{x_d}^{\,a_d}).
\]
Multiple polylogarithms (\ref{FormulaPolylogarithmsRoots}) can be expressed as iterated integrals of such forms, which gives another perspective on Theorem \ref{TheoremMain2}. The connection between the Steinberg module and Milnor $K$-theory can be used to transfer ideas between these fields. As an example of this viewpoint, we give a short proof of the theorems of Ash--Rudolph \cite[Theorem 4.1]{AR79}  and Bykovski\u{\i} \cite{Byk} in \S \ref{SectionARByk}. This proof is almost a word-by-word translation of the classical result of Milnor \cite[Theorem 2.1]{Mil69}. We expect that the results of Br\"uck--Miller--Patzt--Sroka--Wilson \cite{BMPSW24} can be derived from the results of \cite{Rud21} in a similar way.

\subsection{The construction of the truncated  symbol map \texorpdfstring{$\STup$}{ST}}\label{IntroductionMultiplePolylogarithms}

Our next goal is to explain the construction of the truncated symbol map \eqref{maptheoremmain2}. We start by recalling some key facts about polylogarithms in the context of the Goncharov Program (see Dupont \cite{Cle21} for a detailed overview). 

For a field $F$ there exists a graded connected commutative Hopf algebra $\Hc^{\fup}(F)$ of formal multiple polylogarithms \cite{CMRR24}, see \S \ref{SectionMultiplePolylogs}.  As a $\Q$-vector space, $\Hc^{\fup}_n(F)$ is generated by multiple polylogarithms $\Li^{\fup}_{n_1,n_2,\dots, n_d}(x_1,x_2,\dots,x_d)$ for $x_1,\dots,x_d\in F$, and $n=n_1+\dots+n_d$ (see~\cite[Corollary 30]{CMRR24}). The Hopf algebra $\Hc^{\fup}(F)$ is expected to be isomorphic to the conjectural Hopf algebra $\Hc^{\M}(F)$ of mixed Tate motives over $F$. In \S\ref{SecPolylogsAsFunctionsAndMotives} we discuss the connection between formal polylogarithms and polylogarithms viewed as multivalued functions. In weight one,  $\Hc^{\fup}_1(F)$ is isomorphic to $\smash[b]{F^{\times}_{\Q}}$ via the map sending an element $a\in F^{\times}_{\Q}$ to $\log^{\fup}(a)=-\Li^{\fup}_1(1-a)$.

The coproduct $\Delta \Li^{\fup}_{n_1,n_2,\dots, n_d}(x_1,x_2,\dots,x_d)$ is given by an explicit yet intricate combinatorial formula, which comes from analyzing mixed Hodge structures associated to  multiple polylogarithms \cite[\S 6]{Gon01}. Here is an example of a computation of a (reduced) coproduct:
\begin{equation}\label{FormulaCoproductLi21}
\begin{split}
\Delta'(\Li^{\fup}_{2,1}(x_1,x_2)) = {}
 &\Li^{\fup}_{1,1}(x_1,x_2)\otimes \log^{\fup}(x_1)
+\Li^{\fup}_{2}(x_1x_2)\otimes \Li^{\fup}_{1}(x_2)+\Li^{\fup}_{1}(x_2)\otimes \Li^{\fup}_2(x_1)\\
& {} +\Li^{\fup}_{1}(x_1x_2)\otimes \Big(\Li^{\fup}_{2}(x_1^{-1})-\Li^{\fup}_{2}(x_2)+\Li_1^{\fup}(x_2)\log^{\fup}(x_1x_2)\Big).
\end{split}
\end{equation}

As an algebra, $\Hc^{\fup}(F)$ is filtered by depth: we denote by  $\Dc_k\Hc^{\fup}(F)$ the subspace of $\Hc^{\fup}(F)$ spanned by polylogarithms of depth at most $k$. Crucially, the computation in \eqref{FormulaCoproductLi21} shows that the depth filtration does {\bf not} make  $\Hc^{\fup}(F)$  a filtered coalgebra. There is a way to fix it: consider the quotient  of $\Hc^{\fup}(F)$ by the ideal generated by $\Hc_1^{\fup}(F)\cong F^{\times}_{\Q}$, then the depth filtration makes this quotient a filtered coalgebra. In \cite[Conjecture 7.6]{Gon01} Goncharov conjectured 
that the Hopf algebra $\Hc^{\M}(F)\,/\,(\Hc_1^{\M}(F)\Hc^{\M}(F))$ is cofree as a conilpotent coalgebra and the depth filtration on $\Hc^{\M}(F)\,/\,(\Hc_1^{\M}(F)\Hc^{\M}(F))$ coincides with the coradical filtration; the same result is also conjectured to hold for $\Hc^{\fup}(F)$, see \cite[Conjecture 44]{CMRR24}.

The motivic Hopf algebra is constructed in the case of a number field $F$, see Levine \cite{Lev93} and Deligne--Goncharov \cite{DG05}.  In a series of papers on higher cyclotomy \cite{GoncharovMSRI, Gon01B, Gon19B}, Goncharov observed that the sub-Hopf algebra of $\Hc^{\M}(\overline{\Q})$ generated by motivic multiple polylogarithms at roots of unity
\begin{equation} \label{FormulaPolylogarithmsAtRootsofunity}
\Li_{n_1,n_2,\dots, n_d}^{\M}(\zeta_1,\zeta_2,\dots,\zeta_d),\quad \zeta_i\in \mu_{\infty}
\end{equation}
is filtered as a coalgebra. For example, since $\log^{\M}(\zeta)=0$ in $\Hc^{\M}_1(\overline{\Q})$ for any root of unity $\zeta$, the element 
\begin{equation}\label{FormulaCoproductLi21roots of unity}
\Delta' (\Li_{2,1}^{\M}(\zeta_1,\zeta_2)) = \Li_{2}^{\M}(\zeta_1\zeta_2)\otimes \Li_{1}^{\M}(\zeta_2)+\Li_{1}^{\M}(\zeta_2)\otimes \Li_2^{\M}(\zeta_1)
+\Li_{1}^{\M}(\zeta_1\zeta_2)\otimes \Big(\Li_{2}^{\M}(\zeta_1^{-1})-\Li_{2}^{\M}(\zeta_2)\Big)
\end{equation}
lies in $\Dc_1\Hc^{\M}(\overline{\Q})\otimes \Dc_1\Hc^{\M}(\overline{\Q})$ by (\ref{FormulaCoproductLi21}).

 In \S \ref{SectionHopfAlgebraPolylogarithmsOnTorus} we define the Hopf algebra $\Hc(\Tup_d)$ of polylogarithms on a torus $\Tup_d$, whose elements are products of polylogarithms of the form
\[
\Li_{n_1,\dots,n_k}^{\fup}\biggl(\zeta_1\prod_{i=1}^d (\!\sqrt[N]{x_i}\,)^{a_{i1}},\dots,\zeta_k\prod_{i=1}^d (\!\sqrt[N]{x_i}\,)^{a_{ik}}\biggr) ,
\]
where $\zeta_i$ are $N$-th roots of unity and the exponent vectors $(a_{1j},\dots,a_{d\indsp  j})$ are linearly independent. Let $\overline{\Hc}(\Tup_d)$ be the quotient of \( \Hc(\Tup_d) \) by the ideal generated by elements  $\log^{\fup}(x_1),\dots,\log^{\fup}(x_d)$; we denote the coproduct on \( \overline{\Hc}(\Tup_d) \) by $\overline{\Delta}$ and call it the \emph{truncated coproduct}. A computation similar to  Goncharov's for (\ref{FormulaPolylogarithmsAtRootsofunity}) shows that $\overline{\Hc}(\Tup_d)$ is filtered by depth as a Hopf algebra, see \S \ref{SectionCoproductOfPolylogs}. 
 
The next key idea is to look at the iterated truncated coproduct 
\[
\overline{\Delta}^{[n]}\colon \overline{\Hc}(\Tup_d) \lra \left(\overline{\Hc}(\Tup_d)\right)^{\otimes (n+1)} ,
\]
defined as the composition 
$\overline{\Delta}^{[n]}=(\overline{\Delta}'\otimes\overbrace{\Id\otimes \dots \otimes\Id}^{\smash{n-1}})\circ \dots \circ ( \overline{\Delta}'\otimes\Id ) \circ \overline{\Delta}'.$
Since $\overline{\Hc}(\Tup_d)$ is filtered as a coalgebra, the iterated truncated coproduct induces a map
\begin{equation}\label{FormulaIteratedTruncatedCoproduct}
\overline{\Delta}^{[d-1]}\colon \gr^{\Dc}_d\overline{\Hc}(\Tup_d) \lra \left(\gr^{\Dc}_1\overline{\Hc}(\Tup_d)\right)^{\otimes d}.
\end{equation}
Elements of $\left(\gr^{\Dc}_1\overline{\Hc}(\Tup_d)\right)^{\otimes d}$ are $\Q$-linear combinations of terms
\begin{equation}\label{FormulaTensorProductOfClassicaslPolylogs}
\Li^{\fup}_{n_1}\biggl(\zeta_1\prod_{i=1}^d (\!\sqrt[N]{x_i}\,)^{a_{i1}}\biggr) \otimes \dots\otimes \Li^{\fup}_{n_d}\biggl(\zeta_d\prod_{i=1}^d (\!\sqrt[N]{x_i}\,)^{a_{id}}\biggr).
\end{equation}
In \S\ref{SectionCoproductOfPolylogs}, we prove that  for every such term appearing in the image of $\overline{\Delta}^{[d-1]}$, the vectors $(a_{11},\dots,a_{d\indsp  1}),\dots,(a_{1d},\dots,a_{dd})$ are in general position. 

For $V$ a $d$-dimensional $F$-vector space, let $\Bup_d\St(V)$ be the free $\Q$-vector space on decompositions of $V$ into a direct sum of lines ${L_1\oplus \dots\oplus L_d}$; the notation comes from the connection with the bar complex discussed below. If vectors $v_1,\dots,v_d$ are in general position, we denote the basis element of $\Bup_d\St(V)$ corresponding to the decomposition $V=\langle v_1\rangle\oplus \dots\oplus \langle v_d\rangle$ by $[v_1|\cdots|v_d]$ or by $[L_1|\cdots|L_d]$, where $L_i=\langle v_i\rangle$. If the lines~$L_i$ are not in general position, we set $[L_1|\cdots|L_d]=0$.

For $V=\Q^d$ (cf. Theorem \ref{TheoremMain2}), consider the map 
\begin{equation*}
\STup^d\colon \left(\gr^{\Dc}_1\overline{\Hc}(\Tup_d)\right)^{\otimes d}
\lra \Bup_d\St(V)\otimes \Sbb^{\bullet}V
\end{equation*}
sending (\ref{FormulaTensorProductOfClassicaslPolylogs}) to
\[
\lvert\det(v_1,\dots,v_d)\rvert \cdot [v_1|\cdots|v_d]\otimes \frac{v_1^{n_1-1}}{(n_1-1)!}
\cdots \frac{v_d^{n_d-1}}{(n_d-1)!}
\]
where $v_i=\frac{1}{N}(a_{i1},\dots,a_{id})$, $N=\lvert\det(v_1,\dots,v_d)\rvert$. The composition $\STup\coloneqq\STup^d\circ \overline{\Delta}^{[d-1]}$ is called the \emph{truncated symbol map}
\begin{equation}\label{FormulaTruncatedSymbol}
\STup\colon \gr^{\Dc}_d\overline{\Hc}(\Tup_d) \lra \Bup_d\St(V)\otimes \Sbb^{\bullet}V.
\end{equation}
For instance, \eqref{FormulaCoproductLi21} implies that
\begin{equation*}
\begin{split}
\STup(\Li^{\fup}_{2,1}(x_1,x_2)) 
={}&
[e_1{+}e_2 \smid e_2]\otimes (e_1{+}e_2)+[e_2 \smid e_1]\otimes e_1 + [e_1{+}e_2 \smid {-}e_1]\otimes ({-}e_1) -[e_1{+}e_2 \smid e_2]\otimes e_2\\
={}&([e_1{+}e_2 \smid e_2]+[e_2 \smid e_1]-[e_1{+}e_2 \smid {-}e_1])\otimes e_1.
\end{split}
\end{equation*}
(Since $\log^{\fup}(x_i)$, $i=1,\dots,d$, vanish in $\overline{\Hc}(\Tup_d)$, the terms $\Li^{\fup}_{1,1}(x_1,x_2)\otimes \log^{\fup}(x_1)$ and $\Li^{\fup}_{1}(x_1x_2)\otimes \Li_1^{\fup}(x_2)\log^{\fup}(x_1x_2)$ drop out from~\eqref{FormulaCoproductLi21}, and the map $\STup$ is well-defined.)
 Our next goal is to show that the image of the truncated symbol map lies in a subspace of $\Bup_d\St(V)\otimes \Sbb^{\bullet}V$ which can be identified with $\St(V)\otimes \St(V)\otimes \Sbb^{\bullet}V$. For this, we review the Koszul duality between $\St(V)$ and $\St(V)\otimes \St(V)$ discovered by Miller--Nagpal--Patzt \cite{MNP18} and Miller--Patzt--Wilson \cite{MPW23}. 

A $\VB$-module is a functor $\Fc$ from the category whose objects are finite-dimensional vector spaces over a given field $F$ and whose morphisms are linear isomorphisms to the category of $\Q$-vector spaces. The functor $\St$ sending a finite-dimensional vector space $V$ to $\St(V)$ is a $\VB$-module. For $\VB$-modules $\Fc$ and $\Gc$ we define their Day convolution by the formula
\[
(\Fc\otimes_{\VB} \Gc)(V)=\bigoplus_{V=V_1\oplus V_2} \Fc(V_1)\otimes \Gc(V_2).
\]
Here the sum is taken over all decompositions $V=V_1\oplus V_2$ of $V$ into a direct sum of its subspaces $V_1$ and $V_2$.
For a decomposition $V=V_1\oplus V_2$ we have a product map $\St(V_1)\otimes \St(V_2)\lra  \St(V)$
sending $[v_1,\dots,v_{d_1}]\otimes[v_{d_1+1},\dots,v_{d_1+d_2}]$ to $[v_1,\dots,v_{d_1+d_2}].$ These maps can be combined into a morphism ${\St} \otimes_{\VB} {\St}\lra \St$ which makes $\St$ a monoid object in the category of $\VB$-modules. 

Consider the (normalized) bar complex $\Bup_\bullet\St$ where
\[
\Bup_m\St(V)=\bigoplus_{\substack{V=V_1\oplus\dots\oplus V_m \\ V_i\neq 0}}\St(V_1)\otimes \dots \otimes \St(V_m)
\]
and denote by $\partial$ the usual bar differential.
The key result of Miller--Nagpal--Patzt \cite{MNP18} and Miller--Patzt--Wilson \cite{MPW23} is that
\begin{equation}\label{FormulaKoszulity}
H_i(\Bup_\bullet\St(V))=\begin{cases}
0 & \text{if } i\neq d,\\
\St(V)\otimes\St(V)& \text{if } i=d,
\end{cases}
\end{equation}
which can be interpreted as the Koszulity property of the Steinberg module viewed as a monoid object in the category of $\VB$-modules. We give a direct proof of (\ref{FormulaKoszulity}) in \S\ref{SectionKoszulitySteinberg}. The $\VB$-module $\StH$ sending $V$ to $\St(V)\otimes\St(V)$ is referred to as the Koszul dual of the Steinberg module. It has the structure of a commutative Hopf algebra in the $\VB$-sense, which we study in \S \ref{SectionProductCoproductInSteinberg} answering \cite[Question 1.6]{MPW23}. 

 The image of (\ref{FormulaTruncatedSymbol}) lies in the kernel of the map $\partial\otimes \Id$; we have ${(\partial_d\otimes \Id)\circ\STup=0}$, see \S\ref{SectionTruncatedSymbolMap}.
 By (\ref{FormulaKoszulity}), the image of $\STup$ lies in $\St(V)\otimes \St(V)\otimes \Sbb(V)$. So, we obtain a map 
\begin{equation} \label{FormulaIntroductionTruncatedSymbolMap}
\STup\colon  \gr^{\Dc}_d\overline{\Hc}(\Tup_d) \lra \St(V)\otimes \St(V)\otimes \Sbb^{\bullet}V,
\end{equation}
which we also call the truncated symbol map. A direct computation shows that 
\[
\STup(\Li_{n_1,\dots, n_d}^{\fup}(x_1,\dots,x_d))= [e_d, e_d+e_{d-1},\dots, e_d+\dots+e_1] \otimes [e_d,e_{d-1},\dots,e_1] \otimes \prod_{i=1}^d \frac{e_i^{n_i-1}}{(n_i-1)!};
\]
see Proposition~\ref{PropositionTruncatedSymbolMap}.

For $v_1,\dots,v_d\in V$ in general position, consider an element
\[
\Lup[v_1,\dots,v_d]=[v_d, v_d+v_{d-1},\dots, v_d+\dots+v_1] \otimes [v_d,v_{d-1},\dots,v_1] \in \St(V)\otimes \St(V),
\]
which we call \emph{the Steinberg polylogarithm.} In \S \ref{SectionCoxeterPairs} we prove that  Steinberg polylogarithms generate ${\St(V)\otimes \St(V)}$, which implies that the truncated symbol map is surjective. The map (\ref{FormulaIntroductionTruncatedSymbolMap}) is not injective, but becomes an isomorphism when restricted to the subspace $\gr^{\Dc}_d\LL_{n}(\Tup_d)$ of $\gr^{\Dc}_d\overline{\Hc}(\Tup_d)$ discussed in \S \ref{SectionMPandSteinbergModule}. The injectivity of the resulting map is the most technically challenging result of the paper and is proven in
\S\S \ref{SectionProofoftheoremMain}--\ref{SectionLemmaMain}.

\subsection{The conjectures of Rognes and Church--Farb--Putman}
The $\GL_d(\Q)$-module $\LL_{n}(\Tup_d)$ can be viewed as a $\GL_d(\Z)$-module. Denote by $\LL_{n}^{\Z}(\Tup_d)$ its $\GL_d(\Z)$-submodule generated by the functions $\Li_{n_1,\dots,n_k}(x_1,\dots,x_k)$ for $k\leq d$ and let $\gr_d^{\Dc}\LL_{n}^{\Z}(\Tup_d)$ be its quotient by a similarly defined subspace generated by polylogarithms of depth lower
than $d$.   As a $\Q$-vector space,  $\gr_d^{\Dc}\LL_{d}^{\Z}(\Tup_d)$ is spanned by elements 
\begin{equation}\label{FormulaPolylogsOverZ}
\Li_{1,\dots,1}\biggl(\prod_{i=1}^d x_i^{a_{i1}},\dots, \prod_{i=1}^d x_i^{a_{id}}\biggr)
\end{equation}
for $(a_{ij})\in \GL_{d}(\Z)$.

The truncated symbol of elements (\ref{FormulaPolylogsOverZ}) satisfies extra integrality condition: it lies in the subspace of 
$\Bup_d\St(\Q^d)$ spanned by elements $[v_1|\cdots|v_d]$ for $v_i \in \Z^d$ with $\det(v_1,\dots,v_d)=\pm 1$. We conjecture that this condition is also sufficient. To formulate this conjecture more precisely, consider the bar complex  $\Bup_\bullet^{\Z}\St(\Z^d)$  of the Steinberg module viewed as $\GL_d(\Z)$-module. Explicitly, we have 
\[
\Bup_m^{\Z}\St(\Z^d)=\bigoplus_{\substack{\Z^d=L_1\oplus\dots\oplus L_m \\ L_i\neq 0}}\St(L_1)\otimes \dots \otimes \St(L_m).
\]

\begin{conjecture}\label{ConjectureSteinbergZ} The truncated symbol map  $\STup$ induces an isomorphism 
\[
\STup\colon \gr_d^{\Dc}\LL_{d}^{\Z}(\Tup_d) \lra \Ker\left(\Bup^{\Z}_d\St(\Z^d)\stackrel{\partial_d}{\lra}  \Bup_{d-1}^{\Z}\St(\Z^d)\right).
\] 
\end{conjecture}

Conjecture  \ref{ConjectureSteinbergZ} is closely related to the Rognes Connectivity Conjecture. 
In \cite{Rog92}, Rognes introduced the so-called common basis complex $\textup{CB}_d(R)$ for a principal ideal domain $R$. This is a simplicial complex whose vertices are the proper nonzero summands of $R^d$ and ${W_0,\dots,W_p}$ forms a simplex if there is a
basis for $R^n$ such that each $W_i$ is the span of a subset of that basis. Rognes conjectured that if $R$ is a Euclidean ring or a local ring, the reduced homology of  $\textup{CB}_d(R)$ is concentrated in degree $2d-3$. Rognes named the conjectural single
non-vanishing reduced homology group \emph{the stable Steinberg module} $\StL_d(R)$.

For fields, the Rognes Connectivity Conjecture was proven in \cite[Theorem 1.2]{MPW23}; it is not hard to show that in this case $\StL_d(F)$ is isomorphic to the quotient of the module $\St_d(F)\otimes \St_d(F)$ by products in the $\VB$-sense. For $R=\Z$, Rognes conjecture is open. Assuming this case is true, Conjecture~\ref{ConjectureSteinbergZ} would imply that the rationalization of the stable Steinberg module is isomorphic to the quotient of $\gr_d^{\Dc}\LL_{d}^{\Z}(\Tup_d)$ by products. In particular, since 
$\gr_d^{\Dc}\LL_{d}^{\Z}(\Tup_d)$ is generated as a $\GL_d(\Z)$-module  by one element $\Li_{1,\dots,1}(x_1,\dots,x_d)$, we would have that $\StL_d(\Z)\otimes_{\Z} \Q$ is a cyclic $\GL_d(\Z)$-module.
Then by considering any nontrivial shuffle relation (see Proposition~\ref{Proposition_double_shuffle}) and taking coinvariants we would get that
\begin{equation}\label{FormulaCoinvariantsStZ}
 \bigl(\StL_d(\Z)\otimes_{\Z} \Q \bigr)_{\GL_d(\Z)}=0\quad  \text{for}\quad d\geq 2.
\end{equation}
The results of Miller--Nagpal--Patzt \cite{MNP18} mean that the Rognes Connectivity Conjecture for $R=\Z$ together with \eqref{FormulaCoinvariantsStZ} would imply the conjecture of Church--Farb--Putman \cite[Conjecture 2]{CFP14}.

\subsection{Polylogarithms as functions and as motives}\label{SecPolylogsAsFunctionsAndMotives}
There are several common perspectives on polylogarithms. The purpose of this section is to discuss some of them, providing intuitions, examples, and a guide to the literature. We discuss the following four perspectives:
\begin{enumerate}
\item Complex-analytic: polylogarithms are multivalued functions of complex arguments. 
\item Hodge-theoretic: polylogarithms are certain framed mixed Hodge-Tate structures of geometric origin. 
\item Formal:  polylogarithms are elements of a Hopf algebra 
$\mathcal{H}^\fup(F)$ introduced in \cite{CMRR24}.
\item Motivic: polylogarithms are certain framed mixed Tate motives. 
\end{enumerate}
Later in the paper we will take the formal perspective on polylogarithms; we provide the necessary technical details in \S \ref{SectionMultiplePolylogs}.

We defined multiple polylogarithms as power series \eqref{FormulaPolylogarithm} converging in a polydisc. These functions also admit an integral presentation. To state it, we recall Chen's definition of an iterated integral \cite{Che77}; we restrict here to the case of $1$-forms. For a smooth complex manifold $M$ consider a collection of $1$-forms $\omega_1,\ldots,\omega_m$ on $M$ and a piecewise smooth path $\gamma\colon [0,1]\lra M.$ Denote by $\gamma^*\omega_i=f_i(t)\dd  t$
the pull-back of the form $\omega_i$ to the segment $[0,1].$ The iterated integral of $\omega_1,\ldots,\omega_n$ along $\gamma$ is defined by 
\[
\int_\gamma\omega_1\circ\dots\circ\omega_n\\
=\int_{0\leq t_1\leq t_2\leq \dots\leq t_n\leq 1}f_1(t_1)\dd t_1 \wedge \dots \wedge f_n(t_n) \dd t_n.
\]

By term-wise integration of the geometric series (see \cite[Theorem 2.1]{Gon01}), one can show that multiple polylogarithms can be presented as iterated integrals:
\begin{equation} \label{FormulaPolylogarithmsIteratedIntegrals}
\begin{split}
	&\Li_{n_1,n_2,\dots, n_k}(a_1,a_2,\dots,a_k)\\
	={}&(-1)^k \int_0^1 \underbrace{\frac{dt}{t-(a_1\dots a_k)^{-1}} \circ \frac{\dd t}{t}\circ\dots \circ\frac{\dd t}{t}}_{n_1}\circ \dots \circ \underbrace{\frac{\dd t}{t-a_k^{-1}} \circ \frac{\dd t}{t}\circ\dots \circ\frac{\dd t}{t}}_{n_k}.
\end{split}
\end{equation}
Formula (\ref{FormulaPolylogarithmsIteratedIntegrals}) extends  the definition of multiple polylogarithms outside of the polydisc by analytic continuation. The resulting function is multivalued, and its monodromy has been analyzed \cite[\S 5]{Woj97}. The key observation is that the monodromy can always be expressed in terms of multiple polulogarithms of lower weight multiplied by a certain power of $2 \pi i$. 

The rich structure of the monodromy of multiple polylogarithms is greatly clarified by the connection with mixed Hodge--Tate structure. We will explain this connection for the example of the dilogarithm, and give references for the general case. A mixed Hodge structure is a triple, consisting of a finite-dimensional $\Q$-vector space $H$, an increasing \emph{weight filtration} $W_\bullet$ on $H$ and a decreasing \emph{Hodge filtration} $F^\bullet$ on the complexification $H_\mathbb{C}$. These filtrations should satisfy the following compatibility properties: $W_{2n}=W_{2n+1}$ and 
\begin{equation}\label{eqnMHSdecomposition}
W_{2n}H_{\mathbb{C}}=W_{2n-2}H_{\mathbb{C}} \oplus \left ( W_{2n}H_{\mathbb{C}} \cap F^nH_{\mathbb{C}} \right)
\end{equation}
for all  $n\in \Z$. Mixed Hodge--Tate structures form a tensor category; the simple objects are \emph{pure Tate structures} $\Q(-n)$, for which $H$ is one-dimensional and the filtrations are characterized by the following properties: 
\[
W_{2n}H=H,\quad W_{2n-1}H=0,\quad F^n H_{\mathbb{C}} =H_{\mathbb{C}},\quad F^{n-1}H_{\mathbb{C}}=0.
\]

We give a construction of the mixed Hodge-Tate structure associated to the dilogarithm. For a point $a\in \mathbb{C}\sm \{0,1\}$ choose  branches of the functions $\Li_2(a), \Li_1(a)$ and $\log(a)$. Consider the matrix 
\[
L(a)
=
\begin{pmatrix}
1 & 0 & 0\\[6pt]
-\Li_1 (a) & 2\pi i & 0\\[8pt]
-\Li_2(a) & (2\pi i)\,\log a & (2\pi i)^2
\end{pmatrix}
\]
and denote its columns by $c_1(a), c_2(a), c_3(a)$. Let $H(a)$ be the $\Q$-vector space spanned by the columns of the matrix.  We define the weight filtration as follows (we specify only the terms $W_n$ for which the corresponding associated graded pieces are nontrivial):
\[
W_0=\langle c_3 (a) \rangle,\quad  W_2=\langle c_2(a), c_3 (a) \rangle,\quad W_4=\langle c_1(a),c_2(a), c_3 (a) \rangle .
\]
Next, we define the Hodge filtration $F^k$ to be the subspace of $H_{\mathbb{C}}(a)$ consisting of vectors for which the last $k$ coordinates vanishing. For example, $F^2 H_{\mathbb{C}}(a)$ is spanned by the vector 
\[
(2\pi i)^2 c_1(a)+ (2\pi i)\Li_1(a) c_2(a) + (\Li_2(a)-\log(a)\Li_1(a))c_3(a). 
\]
The condition \eqref{eqnMHSdecomposition} can be easily checked, so we indeed obtain a mixed Hodge-Tate structure $H$. The key observation is that the vector space $H(a)$, and both filtrations, are independent of the choice of the branches of the functions $\Li_2(a), \Li_1(a)$ and $\log(a)$.

The fact that the category of mixed Hodge-Tate structures is Tannakian allows to view $H(a)$ as a representation of a certain pro-algebraic group. The entries of the matrix can be upgraded to elements of a certain Hopf algebra of ``matrix coefficients'', called the Hopf algebra of framed mixed Hodge--Tate structures (see \cite{BMS87} and \cite{BGSV90}). The element corresponding to the  entry $\Li_2(a)$ is called Hodge dilogarithm $\Li^{\mathcal{H}}_2(a)$ (we omitted the minus sign). One should think of the Hodge dilogarithm as the ``dilogarithm modulo multiples of $(2\pi i)$''.

In a similar way, one can construct mixed Hodge--Tate structures corresponding to multiple polylogarithms and define their Hodge versions. Here are three key observations which make Hodge versions of multiple polylogarithms useful:
\begin{enumerate}
\item These elements do not depend on the choice of branches. Also, many functional equations for multiple polylogarithms look cleaner. For instance, Euler's reflection formula
\[
\Li_2(a)+\Li_2(1-a)=\frac{\pi^2}{2}-\log(a)\log(1-a) \text{ for } a\in (0,1)
\]
becomes
\[
\Li_2^{\mathcal{H}}(a)+\Li_2^{\mathcal{H}}(1-a)=-\log^{\mathcal{H}}(a)\log^{\mathcal{H}}(1-a) \text{ for } a\in \mathbb{C}\sm \{0,1\}.
\]
\item There exists a \emph{real period map} $p\colon \mathcal{H}\to \R$ which associates to multiple polylogarithms their single-valued versions. For instance, the real period of Hodge dilogarithm is given by \emph{Bloch-Wigner dilogarithm} 
\[
p(\Li_2^{\mathcal{H}}(a))=\Im(\Li_2(a))+\arg(1-a)\log|a|.
\]
Remarkably, the Bloch--Wigner dilogarithm is a single-valued continuous function on $\mathbb{C}\cup \{\infty\}$. The  real period map allows one to associate identities between continuous functions to any identities which hold in $\mathcal{H}$. 
\item The coproduct of Hodge multiple polylogarithms was computed by Goncharov and is given by an explicit formula \cite[\S6]{Gon01}. For example, we have
\[
\Delta \Li_2^{\mathcal{H}}(a)= \Li_2^{\mathcal{H}}(a)\otimes 1+  \Li_1^{\mathcal{H}}(a) \otimes \log^{\mathcal{H}}(a)+ 1\otimes  \Li_2^{\mathcal{H}}(a).
\]
The coproduct is not visible on the level of multivalued functions, but is the central tool for studying multiple polylogarithms. 
\end{enumerate}

Finally, we discuss the motivic perspective on polylogarithms. Beilinson and Deligne \cite{BD94} conjectured that for every field $F$ there exists an abelian category of mixed Tate motives. Assuming this category exists, it would be
equivalent to the category of graded comodules over a certain graded connected commutative Hopf algebra $\mathcal{H}^{\M}(F)$, which is called the Hopf algebra of framed mixed Tate motives.  Then, for any  $a_1,\dots,a_k\in F$ and $n_1,\dots,n_k\in \mathbb{N}$, $n=n_1+\dots+n_k$ there would exist elements 
\[
\Li_{n_1,n_2,\dots, n_k}^{\M}(a_1,a_2,\dots,a_k)\in \mathcal{H}_n^{\M}(F)
\]
called \emph{motivic multiple polylogarithms}. Goncharov~\cite[Conjectures 1.9 and 7.4]{Gon01} conjectured that motivic multiple polylogarithms span $\mathcal{H}_n^{\M}(F)$ as a $\Q$-vector space and gave a candidate description of all relations in weight three. For $n\geq 1$, denote by $H^{i}(\mathcal{H}^{\M}(F),\Q)_n$  the $i$-th cohomology group of the $n$-th graded component of the cobar complex of the Hopf algebra $\mathcal{H}^{\M}(F)$, given by
\[
0 \lra \mathcal{H}^{\M}_{n}(F)\lra \bigoplus_{\substack{n_1+n_2=n \\ n_1,n_2\geq 1}}\mathcal{H}^{\M}_{n_1}(F)\otimes \mathcal{H}^{\M}_{n_2}(F) \lra \cdots 
\,, \]
with \( \mathcal{H}^{\M}_n(F) \) placed in cohomological degree 1. The properties of mixed Tate motives would imply that
\begin{equation}\label{FormulaPolylogsKtheory}
H^{i}(\mathcal{H}^{\M}(F),\Q)_n \cong \gr_{\gamma}^{n} K_{2n-i}(F)_{\Q}\,,
\end{equation}
where $\gamma$ denotes the $\gamma$-filtration on the algebraic $K$-groups  and for any abelian group $A$, we denote by $A_\Q$ its rationalization $\Q\otimes_{\Z}A$. 
The existence of the category of mixed Tate motives is known only for some classes of fields, but in particular it is known for number fields \cite{Lev93}, \cite{DG05}. For an embedding of $F$ in $\mathbb{C}$ there is a realization morphism from $\mathcal{H}_n^{\M}(F)$ to the Hopf algebra of framed mixed Hodge--Tate structures. 

Since the motivic perspective on polylogarithms is only conjectural, we cannot use this approach. Instead, we work with the Hopf algebra of formal multiple polylogarithms $\mathcal{H}^\fup(F)$ which was introduced in \cite{CMRR24}. It is expected to be isomorphic to $\mathcal{H}^\mathcal{M}(F)$ when the latter exists. For number fields, there exist a {\it motivic realization}: a homomorphism 
\[
\mathcal{H}^\fup(F)\lra \mathcal{H}^\mathcal{M}(F).
\]
Also, for any embedding $F\hookrightarrow \mathbb{C}$ there exists a {\it Hodge realization}: a Hopf algebra morphism from $\mathcal{H}^\fup(F)$ to the Hopf algebra of framed mixed Hodge--Tate structures.  Note: in the remainder of the paper the superscript ${}^\fup$ will be ommited for notational simplicity, including in \S\ref{SectionMultiplePolylogs} where formal multiple polylogarithms are defined.

\subsection{Acknowledgements}

We would like to thank Benson Farb, Herbert Gangl, Alexander Goncharov, Alexander Kupers, and Ismael Sierra for many helpful discussions and suggestions. We are particularly grateful to Jeremy Miller and Peter Patzt for explaining the proof of Lemma \ref{LemmaFormulaForS}. We would also like to thank the anonymous referee for numerous comments and suggestions that helped to improve the clarity and exposition.

The second author acknowledges funding by the European Union (ERC, FourIntExP, 101078782). The third author is supported in part by the NSF grant DMS-2502729.

\section{Generalities on Hopf algebras and \texorpdfstring{$\VB$}{VB}-modules}
In this section, we review some of the background material used in the rest of the paper. In \S\S \ref{SectionRecapOfHopfAlgebras}-\ref{SectionRecapOfLieCoalgebras} we review commutative Hopf algebras and Lie coalgebras, see, for instance, \cite{MM65,Mic80, Mon93, Car07,LV12}. In \S \ref{SectionMultiplePolylogs} we discuss the Hopf algebra of multiple polylogarithms, following \cite{CMRR24}. Finally, in \S\S \ref{SectionVBmodules}-\ref{SectionVBmodulesKoszulDuality} we discuss $\VB$-modules and Koszul duality in the $\VB$-setting, following \cite[\S 4.1]{MPW23}. 

\subsection{Recap of Hopf algebras}\label{SectionRecapOfHopfAlgebras}
Let $H=\bigoplus_{n\geq 0} H_{n}$ be a graded Hopf algebra over $\Q$ with product~$m$, coproduct $\Delta$, and antipode $S$. We assume that $H$ is connected ($H_0=\Q$) and commutative. Denote by $H_{+}$ the augmentation ideal $\bigoplus_{n\geq 1} H_{n}$.

The reduced coproduct $\Delta'\colon H_{+}\lra H_{+}\otimes H_{+}$ is defined as $\Delta'=\Delta-\Id\otimes 1- 1\otimes \Id.$  The kernel of the reduced coproduct on $H_+$ is denoted $P(H)$ and is called the space of primitive elements in $H$. For $n\geq 0$ we define an iterated coproduct $\Delta^{[n]}\colon H_{+}\lra H_{+}^{\otimes (n+1)}$ as the composition
\[
\Delta^{[n]}=(\Delta'\otimes{}\underbrace{\Id\otimes \dots \otimes\Id}_{n-1})\circ \dots \circ(\Delta'\otimes\Id\otimes\Id )\circ ( \Delta'\otimes\Id ) \circ \Delta'.
\]
We have  $\Delta^{[0]}=\Id,$ $\Delta^{[1]}=\Delta'.$ 

An important example of a graded Hopf algebra is the shuffle algebra $\Sh(V)$ for which $\Sh_n(V)=V^{\otimes n}$ with shuffle product and deconcatenation coproduct. The map 
\[
\Sc\colon H\lra \Sh(H_1)
\]
sending $a\in H_n$ to $\Delta^{[n-1]}(a)\in \Sh_n(H_1)\cong H_1^{\otimes n}$ is a morphism of Hopf algebras, see \cite[Lemma 2.4]{Gon13}; we call it  \emph{the symbol}.

The coradical filtartion $N_kH_+$ on $H_+$ is defined inductively, see \cite[\S 5.2]{Mon93}. We have $N_0H_+=0$ and 
\[
N_k H_{+}=\{a\in H_+ \setsep \Delta'(a)\in (N_{k-1}H_+)  \otimes H_{+}\}
\]
for $k\geq 1$. It is known that $N_i(H_+)= \Ker(\Delta^{[i]})$, so, in particular, $N_1H_{+}=P(H)$. 
We define the coradical filtration on $H$ as a pullback of  the coradical filtration on $H_+$ with respect to the natural projection $H\twoheadrightarrow H_+$. Notice that $N_0H=H_0=\Q$.

The coradical filtration is compatible with the graded Hopf algebra structure. In particular, 
we have $N_iN_j\subseteq N_{i+j}$, $S(N_i)\subseteq N_i$ and $\Delta'(N_k)\subseteq \sum_{k=i+j} N_i\otimes N_j.$ Also, an element $a$ lies in $N_k$ if and only if all its graded components lie in $N_k$.

We will need the following lemma.

\begin{lemma} \label{LemmaHopfAlgebras1} We have
\[
(N_k H_{+}) \cap (H_{+})^2=\sum_{i+j=k} (N_{i} H_{+})\cdot(N_j H_{+}).
\]
\end{lemma}
\begin{proof}

Consider an element $a$ in  $\sum_{i+j=k} (N_{i} H_{+})\cdot(N_j H_{+})$. This element lies in $N_k H_{+}$ because the coradical filtration is compatible with the algebra structure on $H$. Since $N_0 H_{+}=0$, the element $a$  also lies in $(H_{+})^2$. 

Next, consider an element $a$ in $(N_k H_{+}) \cap (H_{+})^2$. We may assume that $a$ is homogeneous of degree $n\geq 1$. Let $Y\colon H\lra H$ be the grading operator sending $a\in H_n$ to $na$. Consider the \emph{Dynkin operator} $\Dee \colon  H\lra  H$ given by the formula $\Dee=m\circ (S\otimes Y) \circ \Delta$. We have $\Ker(\Dee)= (H_{+})^2$, see  \cite[\S 2.2]{CDG21}. Since $a$  lies in  $(H_{+})^2$, we have
\[
0=\Dee(a)=m\circ (S\otimes Y) \circ\Delta(a)=na+m\circ (S\otimes Y) \circ\Delta'(a).
\]
Since the coradical filtration is compatible with the coalgebra structure, we have 
\[
m\circ (S\otimes Y) \circ\Delta'(a)\in  \sum_{i+j=k} (N_{i} H_{+})\cdot(N_j H_{+}).
\]
and so $a$ lies in $\sum_{i+j=k} (N_{i} H_{+})\cdot(N_j H_{+})$.
\end{proof}

We define the \emph{cobar complex} of $H$
\[
    0\lra \Omega^1H \stackrel{\dd}{\lra} \Omega^2H\stackrel{\dd}{\lra} \Omega^3H\stackrel{\dd}{\lra} \cdots
\]
where $\Omega^m H=\bigl(H_+\bigr)^{\otimes m}$ and the differential $\dd\colon \Omega^mH\lra \Omega^{m+1}H$ given by the formula
\[
\dd=\sum_{i=1}^{m}(-1)^{i-1} \underbrace{\Id\otimes \dots \otimes \Id}_{i-1} {} \otimes \Delta' \otimes {} \underbrace{\Id\otimes \dots \otimes \Id}_{m-i}.
\]

\subsection{Recap of Lie coalgebras}\label{SectionRecapOfLieCoalgebras}
Recall the definition of a Lie coalgebra. For a vector space $V$, consider maps $\tau \colon V \otimes V \lra V \otimes V$ sending $a\otimes b$ to $b\otimes a$ and $\eta \colon  V \otimes V \otimes V\lra V \otimes V\otimes V$ sending $a\otimes b \otimes c$ to $b\otimes c \otimes a$. A~Lie coalgebra is a vector space $\Lc$ together with a map $\delta \colon\Lc\lra \Lc\otimes \Lc$ such that $\tau \circ \delta = -\delta$ and  $(1+\eta +\eta^2)\circ (1\otimes \delta)\circ\delta=0$ (coJacobi identity). We denote the induced map $\Lc\lra \Lambda^2\Lc$ by $\delta$ as well. The cobracket can be extended to a graded derivation of the exterior algebra  $\Lambda^\bullet \Lc$. The corresponding complex is called \emph{Chevalley-Eilenberg complex} of $\Lc$:
\[
0\lra \Lc \stackrel{\delta}{\lra} \Lambda^2 \Lc \stackrel{\delta}{\lra} \Lambda^3 \Lc \stackrel{\delta}{\lra} \cdots .
\]
 The cohomology groups $H^i(\Lc,\Q)$ of the Lie coalgebra $\Lc$  coincide with the cohomology of its Chevalley-Eilenberg complex.  In particular,  $H^1(\Lc,\Q)=\Ker(\delta)$. 

Let $H$ be a graded connected commutative Hopf algebra. The vector space $H_+/(H_+)^2\cong {H_+\otimes_{H}} \Q$ is called the space of \emph{indecomposable elements} of $H$ and is denoted $Q(H)$. The space  $Q(H)$ carries a Lie coalgebra structure with the cobracket $\delta$ obtained from the coproduct $\Delta$ by anti-symmetrization. The following lemma is well known, see \cite{MM65}.

\begin{lemma}\label{LemmaMilnorMoore}
The natural map $P(H)\lra Q(H)$ is injective and its image equals to $\Ker(\delta)$.
\end{lemma}

Lie coalgebra $Q(\Sh(V))$
is called a cofree conilpotent Lie coalgebra on $V$ 
and is denoted $\Lie^c(V)$. Recall that the coLie cooperad $Lie^c$ is the graded dual of the Lie operad $Lie$. For $n\geq 0$ we have an isomorphism 
\[
\Lie^c_n(V)\cong Lie^c_{n} \otimes_{\Sfr_{n}} V^{\otimes n}.
\]

The structure maps 
$
\delta^{[n]}\colon \Lc \lra \Lie^c_{n+1}(\Lc)$ are called \emph{iterated cobrackets}. The element $\delta^{[n]}(a)\in \Lie^c_{n+1}(\Lc)$ is the projection of an element
\[
 (\underbrace{1\otimes \dots \otimes 1}_{n-1} {} \otimes \delta)\circ\dots\circ(1\otimes \delta)\circ\delta(a)\in \Lc^{\otimes n+1}
\]
to the  Lie coalgebra of indecomposable elements of the shuffle algebra.

Lie coalgebra $\Lc$ carries a  \emph{coradical filtration} defined by the formula $N_i\Lc=\Ker(\delta^{[i]})$ for $i\geq 0$. By definition, $N_0\Lc=0$ and $N_1\Lc=\Ker(\delta)$. We will need the following lemma.
\begin{lemma} We have the following:
\[
N_2 \Lc=\{a\in \Lc \setsep \delta^{[2]}(a)=0\}=\{a\in \Lc \setsep \delta(a)\in\Lambda^2\Ker(\delta) \}.
\]
\end{lemma}
\begin{proof} If $\delta(a)$ lies in $\Lambda^2\Ker(\delta)\subseteq \Lc\otimes \Lc$ (where we embed $x\wedge y$ as $x\otimes y - y\otimes x$), then  $(1\otimes \delta)\circ \delta(a)$ vanishes in $\Lc^{\otimes 3}$, and so $\delta^{[2]}(a)=0$. 

Next, assume that $\delta^{[2]}(a)=0$. By the definition of a Lie coalgebra, the image of the map $\delta^{[2]}=(1\otimes \delta)\circ \delta$ lies in the subspace 
\[
W=\Ker(1+\eta+\eta^2) \cap \Ker(1\otimes 1\otimes 1+1\otimes \tau)\subseteq \Lc^{\otimes 3}.
\]
The composition 
\[
W \hookrightarrow \Lc^{\otimes 3}  \twoheadrightarrow \Lie_3(\Lc)
\]
is an isomorphism. So, the vanishing of $\delta^{[2]}(a)$ in $\Lie_3(\Lc)$ implies the vanishing of  $(1\otimes \delta)\delta (a)$ in $\Lc^{\otimes 3}$. It follows that $\delta(a)\in \Lc\otimes \Ker(\delta)$. Since $\tau(\delta(a))=-\delta(a)$, we also have $\delta(a)\in \Ker(\delta)\otimes \Lc$. Since 
\[
\bigl( \Lc\otimes \Ker(\delta)\bigr) \cap \bigl(\Ker(\delta) \otimes \Lc\bigr)\cap \Lambda^2\Lc=\Lambda^2 \Ker(\delta),
\]
we see that  $\delta(a)$ lies in $\Lambda^2 \Ker(\delta)$.
\end{proof}

\subsection{\VB-modules} 
\label{SectionVBmodules}
\begin{definition} Let $F$ be a field. Let $\VB_F$ be the groupoid whose objects are finite-dimensional vector spaces over $F$ and morphisms are isomorphisms. A $\VB_F$-module is a covariant functor from $\VB_F$ to $\Q$-vector spaces.   
\end{definition}

For simplicity we will omit $F$ from the notation, and simply write $\VB$-modules.  $\VB$-modules form an abelian category. For a $\VB$-module $M$ we put $M_d=M(F^d);$ notice that $M_d$ is a $\GL_d(F)$-module.  The symmetric monoidal structure $\otimes$ on the category $\VB$ 
induces a symmetric monoidal structure on the category of $\VB$-modules known as \emph{Day convolution} given by the formula
\begin{equation}\label{FormulaVBproduct}
(M\otimes_{\VB} N)(V)=\bigoplus_{V_1\oplus V_2=V} M(V_1)\otimes N(V_2). 
\end{equation}
Here the sum is taken over all subspaces $V_1,V_2\subseteq V$ with internal direct sum $V_1\oplus V_2=V$. Equivalently,
\[
(M\otimes_{\VB} N)_d=\bigoplus_{d=d_1+d_2}\Ind_{\GL_{d_1}\times \GL_{d_2}}^{\GL_{d}}(M_{d_1}\otimes N_{d_2}).
\]
We call  this symmetric monoidal structure the $\VB$-tensor product.  The symmetric group $\Sfr_m$ acts on $\VB$-tensor powers $M^{\otimes m}$, so we  define symmetric and wedge powers:
\begin{align*}
&\Sbb^m M (V) = M^{\otimes m}(V)\otimes_{\Sfr_m} \textup{triv}_m,\\ 
&\Lambda^m M (V) = M^{\otimes m}(V)\otimes_{\Sfr_m} \textup{sgn}_m,
\end{align*}
where  $\textup{triv}_m$ is the trivial representation of $\Sfr_m$ and $\textup{sgn}_m$ is the alternating representation of $\Sfr_m.$

\begin{example} The unit object in the category of $\VB$-modules is the functor
\[
\Q_{F}(V)=
\begin{cases}
\Q & \text{if $V=0$},\\
0 & \text{if $V\neq 0$}.
\end{cases}
\]
\end{example}

An augmented $\VB$-monoid is a monoid object $A$ in the category of $\VB$-modules together with a surjective map of monoids $A\lra\Q_{F}$ whose kernel $A_+$ is concentrated in positive degrees; we call this map \emph{augmentation}. If $M$ is a right $A$-module and $N$ is a left $A$-module, we define
$M\otimes_A N$ as the coequalizer of two natural maps $M\otimes A \otimes N \lra M\otimes N.$ Let $\Tor_\bullet^A(M,N)$ be the associated derived functor. It can be computed using two-sided bar construction: $\Tor_\bullet^A(M,N)$ is the homology of the complex
\[
\Bup_\bullet(M,A,N)= M\otimes_{A} A_{+}^{\otimes m}\otimes_{A}N.
\]
We will be interested only in the special case $M=N=\Q_{F}$ and denote $\Bup_m(\Q_{F},A,\Q_{F})$ simply by  $\Bup_m (A).$  In this case, we have 
\[
\Bup_m (A)(V)=A_+^{\otimes m}(V)=\bigoplus_{V=V_1\oplus\dots\oplus V_m}A(V_1)\otimes\dots\otimes A(V_m),
\]
where the sum goes over all decompositions of $V$ into a direct sum of its nonzero subspaces. The space $\Bup_m (A)$ is spanned by elements 
\[
[a_1 | \cdots |a_m], \: a_i \in A(V_i).
\]
The homological degree of the above element is defined to be $m$ and the internal degree is $\sum_{i=1}^m \dim(V_i)$. 

The differential is 
    \begin{equation} \label{BarDifferential}
    \partial_m[a_1|\cdots|a_m]=\sum_{j=1}^{m-1}(-1)^{j-1}[a_1|\cdots|a_{j-1}|\,a_{j}a_{j+1}\,|a_{j+2}|\cdots|a_m].
    \end{equation}
    
\begin{remark}\label{RemarkVBZ} In \S\ref{SectionSteinbergCorrelators},  we will need a related notion of a $\VB_{\Z}$-module. A $\VB_{\Z}$-module $M$ is a functor from a groupoid of free abelian groups of finite rank (lattices) to $\Q$-vector spaces. The Day convolution is defined similarly to (\ref{FormulaVBproduct}): 
\[
(M\otimes_{\VB_{\Z}} N)(L)=\bigoplus_{L_1\oplus L_2=L} M(L_1)\otimes N(L_2). 
\]
\end{remark}

\subsection{Koszul duality for \VB-modules}\label{SectionVBmodulesKoszulDuality}
\begin{definition} An augmented $\VB$-monoid $A$ is Koszul if $\Tor_i^{A}(\Q_{F},\Q_{F})$ is supported only in internal degree $i$ for each $i\geq 0.$ Equivalently, $A$ is Koszul if for each $V$ the homology of the complex $\Bup_\bullet A(V)$ is supported in homological degree $\dim(V).$
\end{definition}

Assume that an augmented $\VB$-monoid $A$ is Koszul. Its Koszul dual $\VB$-module $A^{\Hc}$ is defined as follows:
\[
A^{\Hc}(V)=\Tor_{\dim(V)}^A(\Q_{F},\Q_{F}).
\]
More explicitly, for a vector space $V$ of dimension $d$ we have the following description of $A^{\Hc}(V):$
\[
\Ker\Bigg(\bigoplus_{\substack{V=L_1\oplus \dots \oplus L_{d} \\  \dim(L_i)=1}} \!\!\! A(L_1)\otimes \dots \otimes A(L_d)\stackrel{\partial}{\lra}  \!\!\! \bigoplus_{\substack{V=L_1\oplus\dots\oplus W \oplus \dots \oplus L_{d-2} \\ \dim(W)=2, \dim(L_i)=1}} \!\!\! A(L_1)\otimes \dots \otimes A(W)\otimes \dots \otimes A(L_{d-2})\Bigg).
\]

Assume that the  $\VB$-monoid $A$ is graded-commutative in the sense that for $a_1\in A(V_1)$ and $a_2\in A(V_2)$ we have that $a_1 a_2 =(-1)^{\dim(V_1)\dim(V_2)} a_2 a_1.$ Then the $\VB$-module $A^{\Hc}$ is a commutative Hopf algebra in $\VB$-sense. The  product 
 \[
 A^{\Hc}(V_1)\otimes A^{\Hc}(V_2)\stackrel{m}{\lra}  A^{\Hc}(V_1\oplus V_2)
 \]
 is induced by the the shuffle product
 \begin{equation} \label{BarShuffleProduct}
[a_1|\cdots|a_{d_1}]\cdot [a_{d_1+1}|\cdots|a_{d_1+d_2}]=\sum_{ \sigma \in \Sigma_{d_1,d_2}}[a_{\sigma (1)}|\cdots|a_{\sigma (d_1+d_2)}]
 \end{equation}
 where  $\Sigma_{d_1,d_2}\subseteq \Sfr_{d_1+d_2}$ is the set of $(d_1,d_2)$-shuffles. The coproduct 
 \[
 A^{\Hc}(V)\stackrel{\Delta}{\lra} \bigoplus_{V=V_1\oplus V_2}  A^{\Hc}(V_1)\otimes A^{\Hc}(V_2)
 \]
 is induced by the deconcatenation coproduct 
 \begin{equation} \label{BarDeconcatenationCoproduct}
\Delta [a_1|\cdots|a_{d}]=\sum_{d=d_1+d_2}[a_1|\cdots|a_{d_1}]\otimes[a_{d_1+1}|\cdots|a_{d_1+d_2}].
 \end{equation}

Next, define the cobar complex for $A^{\Hc}:$
\[
\Omega^mA^{\Hc}=\bigl(A^{\Hc}_+\bigr)^{\otimes m}
\]
with the differential $\dd\colon \Omega^mA^{\Hc}\lra \Omega^{m+1}A^{\Hc}$ given by the formula
\[
\dd=\sum_{i=1}^{m}(-1)^{i-1} \underbrace{\Id\otimes \dots \otimes \Id}_{i-1} \otimes \Delta' \otimes \underbrace{\Id\otimes \dots \otimes \Id}_{m-i}.
\]

\begin{proposition}[Bar-Cobar duality] \label{PropositionBarCobar} Assume that $A$ is a graded-commutative Koszul $\VB$-monoid.  Then 
\[
H^i\bigl ( (\Omega^{\bullet}A^{\Hc})(V)\bigr)=
\begin{cases}
0 & \text{ \ for \ } i\neq \dim(V),\\
A(V) & \text{ \ for \ } i=\dim(V).
\end{cases}
\]
\end{proposition}
\begin{proof}
The proof is similar to the corresponding result in the non-$\VB$ setting, see \cite[\S 3]{LV12}.
\end{proof}

We define the $\VB$-Lie coalgebra of indecomposables  $A^{\Lc}:$
\[
A^{\Lc}(V)=\Coker\Bigg(\bigoplus_{\substack{V=V_1\oplus V_2\\ \dim(V_i)\geq 1}}  A^{\Hc}(V_1)\otimes A^{\Hc}(V_2)\stackrel{m}{\lra}  A^{\Hc}(V)\Bigg).
\]
We have the Chevalley-Eilenberg complex $\Lambda^\bullet A^{\Lc}$:
\[
0\lra A^{\Lc}\stackrel{\delta}{\lra} \Lambda^2 A^{\Lc} \stackrel{\delta}{\lra} \Lambda^3 A^{\Lc}\stackrel{\delta}{\lra} \cdots.
\]

\begin{proposition}[Bar-Cobar duality, Lie coalgebra version]\label{PropositionBarCobarCoLie} Assume that $A$ is a graded-commutative Koszul $\VB$-monoid.  Then 
\[
H^i\bigl ((\Lambda^\bullet A^{\Lc})(V)\bigr)=
\begin{cases}
0 & \text{ \ for \ } i\neq \dim(V),\\
A(V) & \text{ \ for \ } i=\dim(V).
\end{cases}
\]
\end{proposition}
\begin{proof}
The proof is similar to the corresponding result in the non-$\VB$ setting.
\end{proof}

Observe that the functor  sending $V$ to $\Bup_{d}(A) (V)$ for $d=\dim(V)$ is a commutative cofree VB-Hopf algebra with shuffle product. The Koszulity of $A$ implies that there is an injective  morphism  
\[
A^{\Hc}(V) \lra \Bup_{d}(A) (V),
\]
compatible with $\VB$-monoid structures. After passing to indecomposables with respect to the shuffle product, we obtain an injective map
\begin{equation} \label{VBLiesymbolmap}
A^{\Lc}(V) \lra \Bup_d(A) (V)\otimes_{\Sfr_d}Lie^c_d.
\end{equation}

\begin{definition}\label{DefinitionTensorS} Let $A$ be a $\VB$-module. Define a new $\VB$-module $A\otimes \Sbb:$
\[
(A\otimes \Sbb)(V)=A(V)\otimes  \Sbb^{\bullet}_{\Q}(V),
\]
where $\Sbb^{\bullet}_{\Q}(V)=\Sbb^{\bullet}(V\otimes_{\Z}\Q)$ is the symmetric algebra of the $\Q$-vector space $V\otimes_{\Z} \Q$.
\end{definition}
It is easy to see that if $A$ is an augmented $\VB$-monoid, then $A\otimes {\Sbb}$ is also an augmented $\VB$-monoid with the product
\[
\bigl(A(V_1)\otimes \Sbb^{\bullet}_{\Q}(V_1)\bigr)\otimes \bigl(A(V_2)\otimes \Sbb^{\bullet}_{\Q}(V_2)\bigr) \to \bigl(A(V_1)\otimes A(V_2)\bigr)\otimes \bigl(\Sbb^{\bullet}_{\Q}(V_1)\otimes \Sbb^{\bullet}_{\Q}(V_2)\bigr) \to A(V_1\oplus V_2)\otimes \Sbb^{\bullet}_{\Q}(V_1\oplus V_2).
\]
Here the map $\Sbb^{\bullet}_{\Q}(V_1)\otimes \Sbb^{\bullet}_{\Q}(V_2)\to \Sbb^{\bullet}_{\Q}(V_1\oplus V_2)$ is obtained by an embedding of $\Sbb^{\bullet}_{\Q}(V_1)$ and $\Sbb^{\bullet}_{\Q}(V_2)$ into $\Sbb^{\bullet}_{\Q}(V_1\oplus V_2)$ followed by the multiplication in the symmetric algebra.
\begin{lemma}\label{LemmaTensorS} Assume that $A$ is a graded-commutative Koszul $\VB_{F}$-monoid. Then $A\otimes \Sbb$ is also Koszul and we have
$(A\otimes \Sbb)^{\Hc}=A^{\Hc}\otimes \Sbb$ and $(A\otimes \Sbb)^{\Lc}=A^{\Lc}\otimes \Sbb.$
\end{lemma}
\begin{proof}
Notice that for every decomposition $V=V_1\oplus\dots\oplus V_k$ we have a canonical isomorphism
\begin{equation}\label{EqDecomposition}
\Sbb^{\bullet}_{\Q}(V)\cong \Sbb^{\bullet}_{\Q}(V_1)\otimes\dots\otimes\Sbb^{\bullet}_{\Q}(V_k).
\end{equation}
From here the statement follows.
\end{proof}
We will only use this for $F=\Q$, in which case $\Sbb^{\bullet}_{\Q}(V)$ is just $\Sbb^{\bullet}(V)$, the symmetric algebra of $V$.

\section{The Steinberg module and its Koszul dual}

\subsection{Definition and basic properties of the Steinberg module}
\label{SectionSteinbergDefinition}

 Let $V$ be a vector space over a field $F$ of dimension $d$. The Tits building $\Tc_{V}$ of the vector space $V$ is a $(d-2)$-dimensional simplicial complex whose $p$-simplices are flags of proper nonzero subspaces of $V$ of length $p+1$. Solomon--Tits theorem \cite{Sol69} implies that the Tits building $\Tc_{V}$ is homotopy equivalent to a wedge of $(d-2)$-dimensional spheres. The Steinberg module is its top-dimensional reduced homology group:
\[
\St(V)=\widetilde{H}_{d-2}(\Tc_{V},\Q).
\]
By convention, the reduced homology of an empty set is the base field, so $\St(V)\cong \Q$ if $\dim(V)=1$, and we separately define $\St(V):= \Q$ for $V=0$. Notice that our definition is slightly different from the usual one, because we work with rational coefficients instead of integer coefficients. Below we review various well-known facts about the Steinberg module, which follow from the proof of the Solomon--Tits theorem from~\cite{Sol69}.

As a $\Q$-vector space, the Steinberg module $\St(V)$ is generated by \emph{apartments}, defined as follows. Observe that the space of chains $C_{d-2}(\Tc_{V})$ is a free $\Q$-vector space on the set of complete flags of nonzero proper subspaces in $V.$  Consider a basis $v_1,\dots,v_d$ of $V$. 
A chain 
    \begin{equation}\label{FormulaFromSteinbergToChains}
    [v_1,\dots,v_d]=\sum_{\sigma\in\Sfr_d}(-1)^{\sigma}\bigl(0 \subsetneq\langle v_{\sigma(1)}\rangle \subsetneq \langle v_{\sigma(1)},v_{\sigma(2)}\rangle\subsetneq \dots \subsetneq  \langle v_{\sigma(1)},v_{\sigma(2)},\dots,v_{\sigma(d-1)}\rangle\subsetneq V \bigr)
    \end{equation}
is closed and defines an element  $[v_1,\dots,v_d]\in \widetilde{H}_{d-2}(\Tc_V,\Q)$ which is called an apartment.  Apartments $[v_1,\dots,v_d]$ generate the Steinberg module $\St(V)$ and satisfy the following properties:
\begin{enumerate}
\item\label{strel1} $[v_{\sigma(1)},v_{\sigma(2)}, \dots,v_{\sigma(d)}]=(-1)^{\sigma}[v_1,v_2, \dots,v_d]$ for $\sigma \in \Sfr_d,$
\item \label{strel2} $[\lambda v_1,v_2, \dots,v_d]=[v_1,v_2, \dots,v_d]$  for $\lambda \in F^{\times},$
\item \label{strel3} $\sum_{i=0}^{d} (-1)^i[v_0,\dots,\widehat{v_i},\dots, v_{d}]=0$  for any vectors $v_0,\dots,v_{d}\in V.$
\end{enumerate}
Moreover, any linear relation between  apartments follows from relations \ref{strel1}--\ref{strel3}, see  \cite[\S 3]{LS76}. For linearly dependent vectors $v_1,\dots,v_d$ we put $[v_1,\dots,v_d]=0$; it is not hard to see that relations \ref{strel1}--\ref{strel3} hold even if we allow such ``degenerate'' apartments.

Note that by \ref{strel2} the element $[v_1,\dots,v_d]$ only depends on the images of $v_i$ in the projective space $\PP(V)$, so for $P_i=\langle v_i\rangle$ we will sometimes use the notation
\[[P_1,\dots,P_d] \coloneqq  [v_1,\dots,v_d]\]
and picture an apartment as a simplex with vertices $P_1,\dots,P_d\in \PP(V)$.

We will need the following lemma.

\begin{lemma} \label{LemmaSteinbergBasis}
Let $0=\Fc_0\subsetneq\Fc_1\subsetneq\Fc_2\subsetneq\dots\subsetneq \Fc_d=V$ be a complete flag in $V$. Then the set of apartments $[P_1,\dots,P_d]$ with $\langle P_1,\dots,P_i\rangle = \Fc_i$, $i=1,\dots,d$, is a basis of $\St(V)$ as a $\Q$-vector space. Moreover, for any $[w_1,\dots,w_d]\in\St(V)$ we have 
    \begin{equation} \label{EquationSteinbergFlag}
    [w_1,\dots,w_d] = \sum_{\tau\in\Sfr_d}(-1)^{\tau} 
    [\Fc_1 \cap \Gc_{d}^{\tau},\Fc_2\cap \Gc_{d-1}^{\tau},\dots,\Fc_{d}\cap \Gc_{1}^{\tau}],
    \end{equation}
where $\Gc_{j}^{\tau} = \langle w_{\tau(d-j+1)},\dots,w_{\tau(d)}\rangle$, and we set $[\Fc_1\cap \Gc_{d}^{\tau},\Fc_2\cap \Gc_{d-1}^{\tau},\dots,\Fc_{d}\cap \Gc_{1}^{\tau}]=0$ if the flags $\Fc_{\bullet}$ and $\Gc_{\bullet}^{\tau}$ are not in general position.
\end{lemma}
\begin{proof}
The first part of the statement is well-known and follows from the proof of the Solomon--Tits theorem. It remains to prove (\ref{EquationSteinbergFlag}), which we will do by induction on $d$. For $d=1$ the claim is trivial. Next, let $P_i=\langle w_i\rangle$, then from~\ref{strel3} it follows that
    \[[w_1,\dots,w_d] = \sum_{i=1}^{d}(-1)^{i-1}[\Fc_1,P_1,\dots,\widehat{P_i},\dots,P_d].\]
Denote $W_i=\langle P_1,\dots,\widehat{P_i},\dots,P_d\rangle$, and let $\Sfr_{d,i}$ be the stabilizer of $i$ in $\Sfr_d$. If $\Fc_1\not\subseteq W_i$, then $\Fc_{j}\cap W_i$, $j\geq2$ forms a complete flag in $W_i$, so by inductive assumption in $\St(W_i)$ we compute
    \[[P_1,\dots,\widehat{P_{i}},\dots,P_d] = \sum_{\tau'\in\Sfr_{d,i}}(-1)^{\tau'}
    [\Fc_2 \cap \Gc_{d-1}^{\tau},\Fc_3\cap \Gc_{d-2}^{\tau},\dots,\Fc_{d}\cap \Gc_{1}^{\tau'}],\]
where $\tau=\tau'(1\;2\;\cdots\;i)^{-1}$, and $(1\;2\;\cdots\;i)$ denotes the cyclic permutation of $1,\dots,i$. This implies
    \[(-1)^{i-1}[\Fc_1,P_1,\dots,\widehat{P_{i}},\dots,P_d] = \sum_{\tau'\in\Sfr_{d,i}}(-1)^{\tau}
    [\Fc_1, \Fc_2 \cap \Gc_{d-1}^{\tau},\Fc_3\cap \Gc_{d-2}^{\tau},\dots,\Fc_{d}\cap \Gc_{1}^{\tau'}],\]
which is now true also if $\Fc_1\subseteq W_i$, since both sides vanish in this case.
Note that $\Sfr_d$ is the disjoint union of $\Sfr_{d,i}\cdot (1\;2\;\cdots\;i)^{-1}$, $i=1,\dots,d$, since $\Sfr_{d,i}\cdot (1\;2\;\cdots\;i)^{-1}$ is the set of all $\tau\in\Sfr_d$ with $\tau(1)=i$. Therefore, summing up the above identity for $i=1,\dots,d$ gives~\eqref{EquationSteinbergFlag}.
\end{proof}

For any decomposition $V=V_1\oplus V_2$ we have a map
\[
m\colon \St(V_1)\otimes \St(V_2)\lra \St(V_1\oplus V_2)
\]
sending $[v_1,\dots,v_{d_1}]\otimes[v_{d_1+1},\dots,v_{d_1+d_2}]$ to $[v_1,\dots,v_{d_1+d_2}]$; this map will play a key role in \S \ref{SectionKoszulitySteinberg}. We will also denote by $m$ the multiplication map from $\St(V_1)\otimes\dots\otimes\St(V_k)$ to $\St(V_1\oplus\dots\oplus V_k)$.

We will need the following well-known result.

\begin{proposition}\label{PropositionSteinbergSubspace} Consider a vector space $V$ and a subspace $W\subseteq V$. The map 
\begin{equation} \label{FormulaInPropositionSteibergFlags}
 \bigoplus_{V=W\oplus W'} \St(W)\otimes \St(W') \stackrel{\cong}{\lra} \St(V)
\end{equation}
sending $a_1\otimes a_2\in  \St(W)\otimes \St(W')$ to $m(a_1, a_2)\in\St(V)$ is an isomorphism of $\Stab(W)$-modules, where $\Stab(W)$ is the stabilizer of $W$ inside $\GL(V)$.
\end{proposition}
\begin{proof}
Assume that $\dim(W)=k$. Choose a complete flag $0=\Fc_0\subsetneq \Fc_{1}\subsetneq\dots\subsetneq\Fc_{d-1}\subsetneq \Fc_{d}=V$ 
with $\Fc_k=W$. By Lemma~\ref{LemmaSteinbergBasis}, apartments $[P_1,\dots,P_d]$ such that $\langle P_1,\dots,P_i\rangle=\Fc_i$ form a basis of $\St(V).$  The flag $\Fc$ induces complete flags on $W$ and on $W'$, so, by  Lemma~\ref{LemmaSteinbergBasis}, elements $[P_1,\dots,P_k]\otimes [P_{k+1},\dots,P_d]$ with  $\langle P_1,\dots,P_i\rangle=\Fc_i$ for $i\in \{1,\dots,d\}$ and $\langle P_{k+1},\dots,P_d\rangle=W'$ form a basis of $\St(W)\otimes \St(W')$. So, the map \eqref{FormulaInPropositionSteibergFlags} sends a basis of  $\bigoplus_{V=W\oplus W'} \St(W)\otimes \St(W')$ to a basis of $\St(V)$ and thus is an isomorphism.
\end{proof}
By induction, we obtain the following corollary (which is a special case of Reeder's lemma, see~\cite[Proposition 1.1]{Ree91}):

\begin{corollary} \label{CorollarySteinbergFlag} Consider a vector space $V$ and a partial flag $0=\Fc_{0}\subsetneq \Fc_{1}\subsetneq \dots \subsetneq \Fc_{k-1} \subsetneq \Fc_{k}=V$. Then the natural map
\[
\bigoplus_{\substack{V=V_1\oplus\dots\oplus V_k:\\V_1\oplus \dots\oplus V_i=\Fc_{i},\,i=1,\dots,k}} \!\!\!\! \St(V_1)\otimes \dots\otimes \St(V_k) \stackrel{\cong}{\lra} \St(V) 
\]
is an isomorphism of $\Stab(\Fc)$-modules, where $\Stab(\Fc)$ is the stabilizer of the partial flag $\Fc$ inside $\GL(V)$.
\end{corollary}

\subsection{Koszulity of the Steinberg module} \label{SectionKoszulitySteinberg}

The Steinberg module $V\mapsto \St(V)$ is a $\VB$-module and it has a natural augmentation ${\St\lra \Q_F}$ which is an isomorphism in degree~$0$ and the zero map otherwise. Moreover, it is a $\VB$-monoid with a graded-commutative multiplication
\[
m\colon \St(V_1)\otimes \St(V_2)\lra \St(V_1\oplus V_2)
\]
sending $[v_1,\dots,v_{d_1}]\otimes[v_{d_1+1},\dots,v_{d_1+d_2}]$ to $[v_1,\dots,v_{d_1+d_2}],$ which we discussed in \S \ref{SectionSteinbergDefinition}. We will sometimes denote the product $m([v_1,\dots,v_{d_1}], [v_{d_1+1},\dots,v_{d_1+d_2}])$ by $[v_1,\dots,v_{d_1}]\cdot [v_{d_1+1},\dots,v_{d_1+d_2}]$

\begin{theorem}[\cite{MNP18,MPW23}] \label{TheoremSteinbergKoszulity} The Steinberg module is Koszul as an augmented $\VB$-monoid. Moreover, for a vector space $V$ we have the following isomorphism of $\GL(V)$-modules:
\[
\St^{\Hc}(V)\cong \St(V)\otimes \St(V).
\]
\end{theorem}
We will denote $\StH(V)=\St(V)\otimes \St(V)$ and $\StL(V)=\St^{\Lc}(V)$, the corresponding VB-Lie coalgebra of indecomposables. We give a short algebraic proof of Theorem~\ref{TheoremSteinbergKoszulity} that will also allow us to identify $\StH(V)$ as a subspace of $\Bup_d\St(V)$ and also to give explicit formulas for the product and coproduct in $\StH(V)$. 

\begin{proof}
The main idea is that the complex $\Bup_{\bullet}\St(V)$ is isomorphic to the complex $C_{\bullet-2}(\Tc_V)\otimes\St(V)$, where $C_{\bullet}(\Tc_V)$ is augmented with $C_{-1}(\Tc_V)\coloneqq \Q$. After that, Theorem \ref{TheoremSteinbergKoszulity} follows immediately from the Solomon--Tits theorem.

Consider a map $\beta_k\colon\Bup_k\St(V)\rightarrow C_{k-2}(\Tc_V)\otimes\St(V)$ sending $a_1\otimes \dots \otimes a_k\in \St(V_1)\otimes \dots \otimes \St(V_k)$ to 
\[(0\subsetneq V_1\subsetneq V_1\oplus V_2\subsetneq \dots \subsetneq V_1\oplus\dots\oplus V_{k}=V)\otimes m(a_1,\dots, a_k),
\]
for any decomposition $V=V_1\oplus \dots \oplus V_k$. By~\eqref{BarDifferential} the maps $\beta_k$ commute with the differentials, so it suffices to show that each $\beta_k$ is an isomorphism. Since $C_{k-2}(\Tc_V)$ is a free vector space on partial flags, it is enough to show that for any partial flag $0=\Fc_0\subsetneq \Fc_{1}\subsetneq\dots\subsetneq\Fc_{k-1}\subsetneq \Fc_{k}=V$, the multiplication $m$ induces an isomorphism
\[
\bigoplus_{\substack{V=V_1\oplus\dots\oplus V_k: \\V_1\oplus \dots\oplus V_i=\Fc_{i},\,i=1\dots,k}}\St(V_1)\otimes \dots\otimes \St(V_k) \stackrel{\cong}{\lra} \St(V).
\]
This follows from Corollary \ref{CorollarySteinbergFlag}. 
\end{proof}

Let $s\colon \St(V)\otimes \St(V) \to \StH(V) \subseteq \Bup_d\St(V)$ be the isomorphism constructed in the proof of Theorem~\ref{TheoremSteinbergKoszulity}. We have the following explicit formula for $s$. 

\begin{lemma} \label{LemmaFormulaForS} Let $v_1,\dots,v_d$ and $w_1,\dots,w_d$ be a pair of bases of $V.$ 
For any $\sigma,\tau\in \Sfr_d$ and $0\leq i,j\leq d$ define the subspaces
\[
\Fc^\sigma_i =\bigl\langle v_{\sigma(1)},  \dots, v_{\sigma(i)}\bigr\rangle , \quad \Gc^\tau_j =\bigl\langle 
w_{\tau(d-j+1)},  \dots, w_{\tau(d)}\bigr\rangle.
\]
Then we have
\begin{equation}\label{FormulaSteinbergSteinberg}
s([v_1,\dots,v_d]\otimes[w_1,\dots,w_d])=\sum_{\sigma,\tau \in \Sfr_d}(-1)^{\sigma}(-1)^{\tau}[\,\Fc_1^{\sigma}{\cap}\Gc_{d}^{\tau}\,|\,\Fc_2^{\sigma}{\cap}\Gc_{d-1}^{\tau}\,|\cdots|\,\Fc_d^{\sigma}{\cap} \Gc_{1}^{\tau}\,]\in\Bup_d\St(V),
\end{equation}
where we set $[\,\Fc_1^{\sigma}{\cap} \Gc_{d}^{\tau}\,|\,\Fc_2^{\sigma}{\cap}\Gc_{d-1}^{\tau}\,|\cdots|\,\Fc_d^{\sigma}{\cap} \Gc_{1}^{\tau}\,]=0$ if the flags $\Fc_\bullet^{\sigma}$ and $\Gc_\bullet^{\tau}$ are not in general position. 
\end{lemma}
\begin{proof}
This follows by combining the definition of $\beta_d$ with equations~\eqref{FormulaFromSteinbergToChains} and~\eqref{EquationSteinbergFlag}.
\end{proof}

We will often denote by $[v_1,\dots,v_d]\otimes [w_1,\dots,w_d]$ both the element of $\St(V)\otimes \St(V)$ and its image under~$s$ in $\Bup_d\St(V)$.

\subsection{The product and coproduct in \texorpdfstring{$\StH$}{St\textasciicircum{}2}}
\label{SectionProductCoproductInSteinberg}

\begin{proposition} \label{LemmaFormulaProduct} Consider a decomposition $V=V_1\oplus V_2$. The product of elements $a_1\otimes b_1 \in \StH(V_1)$ and  $a_2\otimes b_2 \in \StH(V_2)$ is the element
 \[
 m(a_1\otimes b_1, a_2\otimes b_2)=m(a_1, a_2)\otimes m(b_1, b_2)\in \StH(V).
 \]
\end{proposition}
\begin{proof}
Assume that $\dim(V_1)=d_1$ and $\dim(V_2)=d_2$.  We have a map 
\[
\mu\colon C_{d_1-2}(V_1)\otimes C_{d_2-2}(V_2)\to C_{d_1+d_2-2}(V_1\oplus V_2)
\]
given by the formula
    \begin{align*} 
    (0\subsetneq\langle v_1\rangle\subsetneq\;&\langle v_1,v_2\rangle \subsetneq\dots\subsetneq \langle v_1,\dots,v_{d_1}\rangle)\otimes (0\subsetneq\langle v_{d_1+1}\rangle\subsetneq\langle v_{d_1+1},v_{d_1+2}\rangle\subsetneq\dots\subsetneq \langle v_{d_1+1},\dots,v_{d_1+d_2}\rangle)\\
    \mapsto&\sum_{\sigma\in\Sigma_{d_1,d_2}}(-1)^{\sigma}(
    0\subsetneq\langle v_{\sigma(1)}\rangle\subsetneq\langle v_{\sigma(1)},v_{\sigma(2)}\rangle\subsetneq\dots\subsetneq \langle v_{\sigma(1)},\dots,v_{\sigma(d_1+d_2)}\rangle),
    \end{align*}
where $\Sigma_{d_1,d_2}$ is the set of $(d_1,d_2)$-shuffles.
The definition of shuffle permutations implies that $\Sigma_{d_1,d_2}$ is the set of left coset representatives of the subgroup $\Sfr_{d_1}\times \Sfr_{d_2}\subseteq\Sfr_{d}$, where $\Sfr_{d_1}\times \Sfr_{d_2}$ is the set of permutations leaving the sets $\{1,\dots,d_1\}$ and $\{d_1+1,\dots,d\}$ invariant.
From this we see that after skew-symmetrizing
the above formula 
over all permutations in $\Sfr_{d_1}\times \Sfr_{d_2}$ and using~\eqref{FormulaFromSteinbergToChains} we obtain
    \[\mu([v_1,\dots,v_{d_1}]\otimes[v_{d_1+1},\dots,v_{d}]) = [v_1,\dots,v_{d}],\]
so on $\St(V)$ the map $\mu$ induces $m$. In the proof of Theorem \ref{TheoremSteinbergKoszulity} we constructed an isomorphism between  $\Bup_{d}\St(V)$ and  $C_{d-2}(\Tc_V)\otimes \St(V)$. Under this isomorphism, the shuffle product on $\Bup_{d}\St(V)$  corresponds to the map 
$\mu \otimes m$. Therefore, on $\St(V)\otimes\St(V)$ the multiplication in $\Bup_d\St(V)$ induces $m\otimes m$, proving the claim.
\end{proof}

The formula for the coproduct is more involved. We introduce the following notation: for a subset $I=\{i_1,\dots,i_k\}\subseteq \{1,\dots,d\}$ and $\{1,\dots,d\}\sm I = \{i_1',\dots,i_{d-k}'\}$, with $i_1<\dots<i_k$ and $i_1'<\dots<i_{d-k}'$ we denote
\[\sigma_I\coloneqq\left(\begin{array}{cccccc}
1 &\cdots & k & k+1&  \cdots & d \\
i_1 & \cdots & i_{k}& i_1' &  \cdots & i_{d-k}' \end{array}\right),\qquad
\tau_I\coloneqq\left(\begin{array}{cccccc}
1 &\cdots & d-k & d-k+1&  \cdots & d \\
i_1' & \cdots & i_{d-k}'& i_1 &  \cdots & i_{k}' \end{array}\right).
\]
Note that $(-1)^{\sigma_I}(-1)^{\tau_I}=(-1)^{|I|(d-|I|)}$.
\begin{proposition}\label{LemmaFormulaCoProduct} Let $A=[v_1,\dots,v_d]$, $B=[w_1,\dots,w_d]$ be a pair of apartments. For subsets $I=\{i_1,\dots,i_k\}$ and $J=\{j_1,\dots,j_{d-k}\}$ in $\{1,\dots,d\}$ we put $A_I=\langle v_{i_1},\dots,v_{i_k} \rangle$ and $B_J=\langle w_{j_1},\dots,w_{j_{d-k}} \rangle.$ The coproduct is given by the following formula:
\begin{equation} \label{FormulaCoproductLA} 
\begin{split}
&\Delta(A\otimes B)=\sum_{I,J}(-1)^{\sigma_I}(-1)^{\tau_J}(A,B)_{I,J} \otimes (A,B)^{I,J}\in \bigoplus_{I,J}\StH(A_I) \otimes \StH(B_J), \\
\end{split} 
\end{equation}
where the sum goes over all pairs of subsets $I$ and $J$ with $V=A_I\oplus B_J$. Here $\bar{I}\coloneqq \{1,\dots,d\}\sm I=\{i'_1,\dots,i_{d-k}'\},$ $\bar{J}\coloneqq \{1,\dots,d\}\sm J=\{j'_1,\dots, j_{k}'\},$ and
\begin{align*}
&(A,B)_{I,J}=[v_{i_1},\dots, v_{i_k}]\otimes [A_I\cap \langle w_{j_1'},B_J\rangle ,\dots, A_I\cap \langle w_{j_{k}'} , B_J\rangle ]\in \StH(A_I),\\
&(A,B)^{I,J}=[B_J\cap \langle v_{i_1'}, A_I  \rangle ,\dots, B_J\cap \langle v_{i_{d-k}'} , A_I\rangle ]\otimes [w_{j_1},\dots, w_{j_{d-k}}] \in \StH(B_J).
\end{align*}
\end{proposition}

Before proving this let us reinterpret~\eqref{FormulaCoproductLA}. Let $C_{\toprm}(\Tc_V)$ denote $C_{\dim(V)-2}(\Tc_V)$, which we recall is the free vector space on complete flags in $V$ (for $V=0$ we set $C_{\toprm}(\Tc_V)=\Q$). Then there is a coproduct map $\Delta_{\Tc}\colon C_{\toprm}(\Tc_V)\to \bigoplus_{W\subseteq V} C_{\toprm}(\Tc_W)\otimes C_{\toprm}(\Tc_{V/W})$, defined on complete flags by
    \begin{equation} \label{FormulaFlagCoproduct}
    \Delta_{\Tc}(\Fc_0\subsetneq\dots\subsetneq \Fc_d)
    = \sum_{j=0}^{d}
    (\Fc_{0}\subsetneq\dots\subsetneq \Fc_j)\otimes
    (\Fc_{j}/\Fc_{j}\subsetneq\dots\subsetneq \Fc_{d}/\Fc_{j}).
    \end{equation}
By~\eqref{FormulaFromSteinbergToChains} this induces a well-defined coproduct $\Delta_{\St}\colon \St(V)\to \bigoplus_{W\subseteq V} \St(W)\otimes \St(V/W)$ given by
    \begin{equation}\label{FormulaStCoproduct}
    \Delta_{\St}([v_1,\dots,v_d]) = \sum_{I\subseteq \{1,\dots,d\}}(-1)^{\sigma_I}[v_{i_1},\dots,v_{i_k}]\otimes [v_{i_1'},\dots,v_{i_{d-k}'}],
    \end{equation}
Let~$\Delta_{\overline{\St}}$ be a map defined by the same formula as~\eqref{FormulaStCoproduct} but with sign $(-1)^{\tau_I}$ instead (the two maps differ by sign $(-1)^{|I|(d-|I|)}$).
Note that $[A_I\cap \langle w_{j_1'},B_J\rangle ,\dots, A_I\cap \langle w_{j_{k}'} , B_J\rangle]$ from the definition of $(A,B)_{I,J}$ is simply $\pi([w_{j_1'},\dots,w_{j_k'}])$, where $\pi\colon V/B_{J}\to A_{I}$ is a natural isomorphism, and similarly for $(A,B)^{I,J}$. Therefore, equation~\eqref{FormulaCoproductLA} can be rewritten as
    \[\Delta(A\otimes B) = \Psi(\Delta_{
    \vphantom{\overline{\St}}
    \St}(A)\otimes\Delta_{\overline{\St}}(B)),\]
where, for any VB-modules $F,G$ we define
    \[\Psi=\Psi_{F,G}\colon\bigoplus_{V_1,V_2\subseteq V}F(V_1)\otimes F(V/V_1)\otimes G(V_2)\otimes G(V/V_2) \lra \bigoplus_{V=V_1\oplus V_2}F(V_1)\otimes G(V_1)\otimes F(V_2)\otimes G(V_2)\]
by setting
    \[\Psi(A\otimes A'\otimes B\otimes B') = (A\otimes \pi_{V_1}(B'))\otimes (\pi_{V_2}(A')\otimes B)\qquad \mbox{if}\;\;V=V_1\oplus V_2\]
and $0$ otherwise, where $\pi_{V_1}\colon V/V_{2}\to V_1$ and $\pi_{V_2}\colon V/V_{1}\to V_2$ are natural isomorphisms (cf.~\cite{JS}, where a similar construction is used to define a braiding for external tensor products).

\begin{proof}[Proof of Proposition~\ref{LemmaFormulaCoProduct}]
By the above discussion, it is enough to show that the deconcatenation coproduct $\Delta$ on $\Bup_d\St(V)\cong C_{\toprm}(\Tc_V)\otimes \St(V)$ is equal to $\widetilde{\Delta}=\Psi\circ (\Delta_{\Tc}\otimes\Delta_{\overline{\St}})$. It is enough to show this on the basis elements $[v_1|\dots|v_d]$. Let $\Fc_i=\langle v_1,\dots,v_i\rangle$, then the coproduct $\Delta_{\Tc}(\Fc_{\bullet})$ is given by~\eqref{FormulaFlagCoproduct}, and $\Delta_{\overline{\St}}([v_1,\dots,v_d])$ is given by~\eqref{FormulaStCoproduct} with $\tau_{I}$ in place of $\sigma_{I}$. To apply $\Psi$ note that $V=\Fc_j\oplus \langle v_{i_1},\dots,v_{i_k}\rangle$ if and only if $I=\{j+1,\dots,d\}$, in which case $(-1)^{\tau_{I}}=1$. Therefore,
    \begin{multline*}
    \widetilde{\Delta}\big((\Fc_0\subsetneq\dots\subsetneq \Fc_d)\otimes [v_1,\dots,v_d]\big) = \smash[b]{\sum_{j=0}^{d}}
    \big((0\subsetneq\langle v_1\rangle\subsetneq\dots\subsetneq \langle v_1,\dots,v_j\rangle)\otimes [v_1,\dots,v_j]\big)\\
    \otimes \big((0\subsetneq \langle v_{j+1}\rangle\subsetneq\dots\subsetneq\langle v_{j+1},\dots,v_d\rangle)\otimes [v_{j+1},\dots,v_d]\big),
    \end{multline*}
which under isomorphism $\beta_d$ maps to the deconcatenation coproduct~\eqref{BarDeconcatenationCoproduct}.
\end{proof}

\subsection{Coxeter pairs}\label{SectionCoxeterPairs}
Let $V$ be a vector space over $F$ of dimension $d$ and denote by $\PP^{d-1}=\PP(V)$ the projectivization of $V$. Lines in $V$ are in bijection with points of $\PP(V)$; we refer to them as \emph{points}.  Planes in $V$ are in bijection with lines of $\PP(V)$; we refer to them as \emph{projective lines}.
Finally, hyperplanes in $V$ are in bijection with hyperplanes in $\PP(V)$; we refer to them as \emph{hyperplanes}. We say that vectors $v_1,\dots,v_m \in V$ are in general position if vectors $v_{i}, i\in I$ are linearly independent for every subset $I\subseteq\{1,\dots, m\}$ of size at most $d$.  We say that points $P_1,\dots,P_m$ are in general position  if the vectors in $V$ generating the corresponding lines are in general position . 

A \emph{projective simplex} is an ordered set of $d$ points $\Pc=(P_1,\dots,P_d)$ in $\PP(V).$ We say that $\Pc$ is non-degenerate if the points $P_1,\dots,P_d$ are in general position. Let $[\Pc]=[P_1,\dots,P_d]\in \St(V)$ be the corresponding apartment. A pair of simplexes $(\Pc, \Qc)$ defines an element $[\Pc]\otimes [\Qc]$ in $\StH(V)$.

\begin{definition}\label{DefinitionSpecialPairs} A pair of simplexes $(\Pc, \Qc)$ is called a \emph{Coxeter pair} if both simplexes $\Pc$ and $\Qc$ are non-degenerate, $P_1=Q_1$, and $Q_i$ lies on the projective line $\langle P_{i-1},P_i\rangle$ for $i=2,\dots,d$. We will also use the term  \emph{Coxeter pair} for the corresponding element $[\Pc]\otimes [\Qc]\in \StH(V)$. 
\end{definition}

The main result of this section is the following theorem.

\begin{theorem} \label{thm:Stn_li_basis}
	Let $V$ be a vector space over $F$ of dimension $d;$ let $H\subseteq V$ be a hyperplane. Then Coxeter pairs $[\Pc]\otimes [\Qc]$ with $P_1,\dots,P_{d-1}\subseteq H$ generate $\StH(V)$ as a $\Q$-vector space.
\end{theorem}
\begin{proof}
Choose any complete flag $\Fc_{\bullet}$ with $\Fc_{d-1}=H$. By Lemma  \ref{LemmaSteinbergBasis}, apartments $[P_1,\dots,P_d]$ with $P_i\in\Fc_i$ form a basis $\St(V)$. In particular,  apartments $[P_1,\dots,P_d]$ with $P_1,\dots,P_{d-1}\subseteq H$ span $\St(V)$. So, the statement of the theorem follows from the following claim:
any element $[\Pc]\otimes [\Qc]$ with $P_1,\dots,P_{d-1}\subseteq H$ can be written as a linear combination of Coxeter pairs $[\Pc']\otimes [\Qc']$ satisfying $P_1',\dots,P_{d-1}',Q_1',\dots,Q_{d-1}'\subseteq H$ and $P_{d}'=P^{\vphantom{'}}_{d}$. 

We prove the claim by induction on $d$. For $d=0,1$ there is nothing to prove, and for $d=2$ the hyperplane $H$ is a point and the statement follows from the identity 
\[
[P_1,P_2]\otimes [Q_1,Q_2]=[P_1,P_2]\otimes [H,Q_2]-[P_1,P_2]\otimes [H,Q_1].
\] 
To prove the induction step, we first use Lemma \ref{LemmaSteinbergBasis} to reduce to the case when $Q_1,\dots,Q_{d-1}\subseteq H$. Let $R$ be the point of intersection of $H$ with $\langle Q_d, P_d\rangle$ (in case $Q_d=P_d$ take $R$ to be any point of $H$). Then the following identity holds in $\St(V)$:
	\[[P_1,\dots,P_d] = \sum_{i=1}^{d-1}(-1)^{d-i-1}
    [P_1,\dots,\widehat{P_{i}},\dots,P_{d-1},R,P_d].\]
By the induction hypothesis applied to the space $H$  and a hyperplane $H'\coloneqq \langle P_1,\dots,\widehat{P_i},\dots,P_{d-1}\rangle$, any element $[P_1,\dots,\widehat{P_i},\dots,P_{d-1},R]\otimes[Q_1,\dots,Q_{d-1}]$ can be written as a sum of Coxeter pairs $[P_1',\dots,P_{d-1}']\otimes [Q_1',\dots,Q_{d-1}']$ with $P_{d-1}'=R$. By the choice of $R$, the point $Q_d$ lies on the projective line $\langle P'_{d-1}, P_d^{\vphantom{'}}\rangle$, so $[P_1',\dots,P_{d-1}',P^{\vphantom{'}}_d]\otimes [Q_1',\dots,Q_{d-1}',Q^{\vphantom{'}}_d]$ is also a Coxeter pair. This proves the induction step.
\end{proof}

Note that lines in the dual space $V^\vee$ are in natural bijection with hyperplanes in $V$. Explicitly, for a nonzero functional $\varphi\in V^{\vee}$, a line $\langle \varphi \rangle$ corresponds to the hyperplane $\ker\varphi\subset V$. Thus, points of the projectivization $\PP(V^\vee)$ are in bijection with hyperplanes in $\PP(V)$. For any non-degenerate projective simplex $\Pc=(P_1,\dots,P_d)$ consider its dual 
$\Pc^{\vee}=(H_1,\dots,H_d)$
for $H_i=\langle P_1,\dots,\widehat{P_i},\dots,P_d \rangle\in \PP(V^{\vee}).$ This gives rise to a well-defined duality isomorphism
    \begin{equation}\label{FormulaDualityMapSt}
    D\colon \StH(V)\stackrel{\sim}{\lra} \StH\bigl(V^{\vee}\bigr)
    \end{equation}
defined by the formula  $D([\Pc]\otimes[\Qc])=[\Qc^{\vee}]\otimes [\Pc^{\vee}].$

\begin{lemma}
The dual of a Coxeter pair $(\Pc,\Qc)$ is again a Coxeter pair.
\end{lemma}
\begin{proof}
Consider a Coxeter pair $(\Pc,\Qc).$ By definition, $P_1=Q_1$ and $Q_i\in \langle P_{i-1}, P_i \rangle$ for $2\leq i\leq d.$ Consider the faces $p_i=\langle P_1,\dots,\widehat{P}_i,\dots, P_d\rangle$ and  $q_i=\langle Q_1,\dots,\widehat{Q}_i,\dots, Q_d\rangle$ of $\Pc$ and $\Qc.$ To show that $(\Qc^{\vee},\Pc^{\vee})$ is a Coxeter pair we need to show that $q_d=p_d$ and $q_i\cap q_{i+1}\subseteq p_i$ for $1\leq i\leq d-1.$
To see the first statement, notice that since $P_1=Q_1$ and $Q_i\in \langle P_{i-1}, P_i \rangle$ for $2\leq i\leq d$ we have $\langle Q_1,\dots, Q_{d-1} \rangle\subseteq \langle P_1,\dots, P_{d-1} \rangle.$ The equality follows from the general position assumption in 
Definition~\ref{DefinitionSpecialPairs}. Next, ${q_i\cap q_{i+1}}=\langle Q_1,\dots, Q_{i-1},Q_{i+2},\dots, Q_d\rangle$
and for $j\in\{1,\dots,i-1,i+2,\dots,d\}$ we have  $Q_j\in \langle P_{j-1},P_j \rangle \subseteq p_i.$ This proves the statement.
\end{proof}

Dualizing the statement of Theorem~\ref{thm:Stn_li_basis} we then get the following.

\begin{corollary}
\label{thm:Stn_I_basis}
 Let $V$ be a vector space over $F$ of dimension $d;$ let $L\subseteq V$ be a line. Then Coxeter pairs $[\Pc]\otimes[\Qc]$ with $P_1=Q_1=L$ generate $\StH(V)$  as a $\Q$-vector space.
\end{corollary}

\subsection{Generic Coxeter pairs, Steinberg polylogarithms and Steinberg iterated integrals}
\label{sec:steinbergLiI}
In this section, we define generic Coxeter pairs and discuss two ways to parametrize them. For $F=\Q$, these parametrizations are related to multiple polylogarithms and iterated integrals. 

\begin{definition}
A Coxeter pair $(\Pc,\Qc)$  is called \emph{generic} if $P_{i}\neq Q_i$ for $2\leq i\leq d.$
\end{definition}

Non-generic Coxeter pairs are decomposable in the $\VB$-algebra $\StH(V)$:

\begin{proposition} \label{PropositionNonGenericPairs}
Suppose that a Coxeter pair $(\Pc,\Qc)$ is \emph{not} generic. Then
\[
[\Pc]\otimes [\Qc]\in \sum_{\substack{V=V_1\oplus V_2\\ \dim(V_i)\geq1}} m(\StH(V_1)\otimes \StH(V_2)).
\]
Equivalently, the projection of $[\Pc]\otimes [\Qc]\in \StH(V)$ to $\StL(V)$ vanishes.
\end{proposition}
\begin{proof}
If $(\Pc,\Qc)$ is not generic, then there exists $1\leq k<d$ such that $Q_{k+1}=P_{k+1}$. If we denote $\Pc'=(P_1,\dots,P_{k})$, $\Pc''=(P_{k+1},\dots,P_{d})$, $\Qc'=(Q_1,\dots,Q_{k})$, $\Qc''=(Q_{k+1},\dots,Q_{d})$, then $(\Pc',\Qc')$ and $(\Pc'',\Qc'')$ are also Coxeter pairs, and 
  \[
  [\Pc]\otimes[\Qc]
  =\Bigl([\Pc']\otimes[ \Qc']\Bigr)\cdot \Bigl([\Pc'']\otimes [\Qc'']\Bigr),
  \]
proving the claim.
\end{proof}

Our next goal is to parametrize generic Coxeter pairs. 

\begin{definition}
Let $v_1,\dots,v_d$ be a basis of a vector space $V$. The \emph{Steinberg polylogarithm} is an element of \( \StH(V) \) given by 
\begin{equation}
\Lup[v_1,\dots,v_d] \coloneqq 
[v_d, v_d+v_{d-1},\dots,v_{d}+\dots+v_1]\otimes [v_d,v_{d-1},\dots,v_1] .
\end{equation}
The \emph{Steinberg iterated integral} is an element of \( \StH(V) \) given by 
\begin{equation}
\I[v_1,\dots,v_d] \coloneqq (-1)^{d}[v_d, v_{d-1},\dots v_1]\otimes [v_d, v_{d-1}-v_d,\dots,v_1-v_2] .
\end{equation}
We write $\Lup^{\Lc}[v_1,\dots,v_d]$ and $\I^{\Lc}[v_1,\dots,v_d]$ for the projections of the above elements to $\StL(V)$. 
\end{definition}

The following lemma gives a parameterization of generic Coxeter pairs.

\begin{lemma} \label{LemmaCoxeterParameterization}
Let $(\Pc,\Qc)$ be a generic Coxeter pair. Then there exists a basis $v_1,\dots,v_d$ of $V$ such that $P_{i}=\langle v_1+\dots+v_i\rangle$ and $Q_{i}=\langle v_i\rangle$, for $1\leq i \leq d.$ In particular, 
\[[\Pc]\otimes [\Qc]=\Lup[v_d,v_{d-1},\dots,v_1].\]
\end{lemma}
\begin{proof}
This is an easy exercise in linear algebra.
\end{proof}

Directly by definition, we have
\begin{align}
\Lup[v_1,v_2,\dots,v_d]=&\,(-1)^d\I[v_1+\dots+v_d,v_2+\dots+v_{d},\dots, v_d], \label{FormulaLviaI}\\
\I[v_1,v_2,\dots,v_d]=&\,(-1)^d\Lup[v_1-v_2,v_2-v_3,\dots, v_{d-1}-v_d,v_d].
\label{FormulaIviaL} \end{align}

We will use the following presentation for symbols of Steinberg polylogarithms and Steinberg iterated integrals:

\begin{lemma} \label{LemmaLCoproductComponent} The following equalities holds in $\Bup_d\St(V)$:
\begin{align*}
s(\Lup[v_1,\dots,v_d])={}& s(\Lup[v_2,\dots, v_{d}])\otimes [v_1] \\[-0.5ex] &{} +\sum_{i=1}^{d-1} s(\Lup[v_1,\dots,v_{i-1},v_i+v_{i+1},v_{i+2},\dots, v_d])\otimes ([v_{i+1}] -[v_{i}]) ,\\[2ex]
s(\I[v_1,\dots,v_d])={}&-s(\I[v_1,\dots, v_{d-1}])\otimes [v_d]\\[-0.5ex]&{}+ \sum_{i=1}^{d-1} (s(\I[v_1,\dots,\widehat{v_{i+1}},\dots, v_d])-s(\I[v_1,\dots,\widehat{v_{i}},\dots, v_d]))\otimes [v_{i+1}-v_i].
\end{align*}
\end{lemma}
\begin{proof}
We prove the first statement; the second statement follows from the first and~\eqref{FormulaIviaL}. We introduce the notation $w_k=v_k+\dots+v_d$. To prove the identity, we compute $\Delta_{d-1,1}s(\Lup[v_1,\dots,v_d])$ using~\eqref{FormulaCoproductLA}. The relevant terms in the formula~\eqref{FormulaCoproductLA} correspond to decompositions of $V$ into a sum of two subspaces: $W_i=\langle w_1,\dots,\widehat{w_i},\dots, w_d \rangle$ and $\langle v_j\rangle$ for $1\leq i,j\leq d$. We have $V=W_i\oplus \langle v_j\rangle$ if and only if $j=i$ or $j=i-1$, and in these cases the corresponding term is
    \begin{align*}
    &(-1)^{i+j}s([w_1,\dots,\widehat{w_i},\dots,w_d]\otimes [v_1,\dots,v_{i-2},v_{i-1}+v_{i},v_{i+1},\dots,v_{d}])\otimes [v_j]
    \\
    = {} &(-1)^{i+j}s(\Lup[v_1,\dots,v_{i-2},v_{i-1}+v_{i},v_{i+1},\dots,v_{d}])\otimes [v_j],
    \end{align*}
if $(i,j)\ne (1,1)$ or
    \begin{align*}
    s([w_2,\dots,w_d]\otimes [v_2,\dots,v_{d}])\otimes [v_1] = s(\Lup[v_2,\dots,v_d]) \otimes [v_1],
    \end{align*}
if $i=j=1$, proving the claim.
\end{proof}

Lemma \ref{LemmaLCoproductComponent} together with the fact that $s(\I[v])=s(-[v] \otimes[v]) =-[v]$ and  $s(\Lup[v])=s([v] \otimes[v]) =[v]$ allows us to inductively compute $s(\Lup[v_1,\dots,v_d])$ and $s(\I[v_1,\dots,v_d]).$

\begin{example}
For $d=2$ we have
\begin{align*}
s(\Lup[v_1,v_2]) &= [v_2 \smid v_1]-[v_1{+}v_2 \smid v_1]+[v_1{+}v_2 \smid v_2],\\
s(\I[v_1,v_2]) &= [v_1 \smid v_2]-[v_1 \smid v_2{-}v_1]+[v_2 \smid v_2{-}v_1].
\end{align*}
Similarly, for $d=3$
\begin{align*}
s(\I[v_1,v_2,v_3]) = {} & 
 -s(\I[v_1,v_2])\otimes[v_3]
+\big(s(\I[v_1,v_3])
-s(\I[v_2,v_3])\big)\otimes[v_2{-}v_1] \\
& 
{} +\big(s(\I[v_1,v_2])-s(\I[v_1,v_3])\big)\otimes[v_3{-}v_2]\\[1ex]
= {} & - [v_1 \smid v_2 \smid v_3]+[v_1 \smid v_2{-}v_1 \smid v_3]-[v_2 \smid v_2{-}v_1 \smid v_3]
\\
 &+[v_1 \smid v_3 \smid v_2{-}v_1]-[v_1 \smid v_3{-}v_1 \smid v_2{-}v_1]+[v_3 \smid v_3{-}v_1 \smid v_2{-}v_1]\\
 &-[v_2 \smid v_3 \smid v_2{-}v_1]+[v_2 \smid v_3{-}v_2 \smid v_2{-}v_1]-[v_3 \smid v_3{-}v_2 \smid v_2{-}v_1]\\
 &+[v_1 \smid v_2 \smid v_3{-}v_2]-[v_1 \smid v_2{-}v_1 \smid v_3{-}v_2]+[v_2 \smid v_2{-}v_1 \smid v_3{-}v_2]\\
 &-[v_1 \smid v_3 \smid v_3{-}v_2]+[v_1 \smid v_3{-}v_1 \smid v_3{-}v_2]-[v_3 \smid v_3{-}v_1 \smid v_3{-}v_2].
\end{align*}
\end{example}

\begin{proposition} \label{PropositionDuality} Let  $v_1,\dots ,v_d$ be a basis of $V$ and $v^1,\dots,v^d$ be the corresponding dual basis of $V^{\vee}$. We have 
\begin{align}
D\bigl(\Lup[v_1,\dots,v_d]\bigr)  &=(-1)^d\I[v^d,\dots,v^1],\label{FormulaDualityL}\\ 
D\bigl (\I[v_1,\dots,v_d]\bigr) &=(-1)^d\Lup[v^d,\dots,v^1].\label{FormulaDualityI}
\end{align}
\end{proposition}
\begin{proof}
We prove (\ref{FormulaDualityL}); the proof of (\ref{FormulaDualityI}) is similar. For this notice that the dual basis of the basis $v_d, v_d+v_{d-1},\dots,v_d+v_{d-1}+\dots+v_1$ is $v^d-v^{d-1}, v^{d-1}-v^{d-2},\dots,v^2-v^1,v^1.$ Thus 
    \begin{align*}
    D\bigl(\Lup[v_1,\dots,v_d]\bigr)&=D \Bigl( 
[v_d, v_d+v_{d-1},\dots,v_{d}+\dots+v_1]\otimes [v_d,v_{d-1},\dots,v_1]\Bigr)\\
&=[v^d,v^{d-1},\dots,v^1]\otimes [v^d-v^{d-1},v^{d-1}-v^{d-2},\dots,v^2-v^1,v^1]\\
    &=[v^1,v^2,\dots,v^d]\otimes[v^1,v^2-v^1,\dots,v^d-v^{d-1}]\\
    &=(-1)^d\I[v^d,\dots,v^1].\qedhere
    \end{align*}
\end{proof}

\subsection{Double Shuffle Relations and their corollaries} \label{SectionDoubleShuffle}
In this section, we prove a family of identities in $\StH(V)$ inspired by the Double Shuffle Relations for polylogarithms. As a consequence, we prove that generic Coxeter pairs span $\StH(V)$.

\begin{proposition}[Double Shuffle Relations] \label{Proposition_double_shuffle} Consider a decomposition $V=V_1\oplus V_2$. Let $v_1,\dots,v_{d_1}$ be a basis of $V_1$ and  $v_{d_1+1},\dots,v_{d_1+d_2}$ be a basis of $V_2.$ The following equalities hold in $\StH(V)$:
\begin{align}
\Lup[v_1,\dots,v_{d_1}]\Lup[v_{d_1+1},\dots,v_{d_1+d_2}]={}&\sum_{\sigma\in \Sigma_{d_1,d_2}}\Lup[v_{\sigma(1)},\dots,v_{\sigma(d)}],\label{eq:shuffle1}\\
\I[v_1,\dots,v_{d_1}]\I[v_{d_1+1},\dots,v_{d_1+d_2}]={}&\sum_{\sigma\in \Sigma_{d_1,d_2}}\I[v_{\sigma(1)},\dots,v_{\sigma(d)}],\label{eq:shuffle2}
\end{align}
where $\Sigma_{d_1,d_2}\subseteq \Sfr_{d}$ is the set of $(d_1,d_2)$-shuffles.
\end{proposition}
\begin{proof}
We will prove only the first identity, the second follows by duality, see Proposition~\ref{PropositionDuality}. Identity~\eqref{eq:shuffle1} is a consequence of the following identity in $\St(V_1\oplus V_2)$:
    \begin{equation*} 
        [v_{1},\dots,v_1+\dots+v_{d_1}]\,[v_{d_1+1},\dots,v_{d_1+1}+\dots+v_{d_1+d_2}]   =
    \sum_{\sigma\in \Sigma_{d_1,d_2}} \!\! (-1)^\sigma[v_{\sigma(1)},\dots,v_{\sigma(1)}+\dots+v_{\sigma(d)}]  .
    \end{equation*}
It suffices to check that this identity holds in $\St(V)$ in the case when $v_i=e_i$ form the standard basis.

First, we show that this identity holds in $\St_d(\Q)$. It can be seen in terms of cones (see~\S\ref{sec:steinbergcones}): $\sgn(\sigma)[e_{\sigma(1)},\dots,e_{\sigma(1)}+\dots+e_{\sigma(d_1+d_2)}]$ represents the cone $C_{\sigma}=\{t\in\Q^{d_1+d_2}\setsep t_{\sigma(1)}\geq t_{\sigma(2)}\geq\dots\geq t_{\sigma(d_1+d_2)}\geq0\}$ and the element on the left hand side represents the cone 
$C=\{t\in\Q^{d_1+d_2} \setsep t_{1}\geq t_{2}\geq\dots\geq t_{d_1}\geq0\,,\,t_{d_1+1}\geq t_{d_1+2}\geq\dots\geq t_{d_1+d_2}\geq0\}$.
The permutations in $\Sigma_{d_1,d_2}$ encode the orderings of $t_i$ extending $t_1\geq\dots\geq t_{d_1}$ and $t_{d_1+1}\geq\dots\geq t_{d_1+d_2}$, therefore $C=\bigcup_{\sigma\in\Sigma_{d_1,d_2}}C_{\sigma}.$ Since  $C_{\sigma}$ and $C_{\sigma'}$ have non-intersecting interiors for $\sigma\ne\sigma'$, this implies the sought after identity in $\St_d(\Q)$.

Next, Theorem \ref{TheoremARByk} implies that for every field $F$ there exists a map $\St_d(\Q)\lra \St_d(F)$ sending an integral apartment  $\St_d(\Q)$ to the corresponding element in $\St_d(V)$ using a natural map $\Z\lra F$. So, the identity above holds in $\St(V)\cong\St_d(F)$.
\end{proof}

\begin{corollary}[Dihedral relations] \label{CorollaryDihedral} Consider a basis $v_1,\dots,v_d$ of $V$ and put $v_{0}=-(v_1 + \cdots + v_d)
.$ The following equalities hold in $\StL(V)$:
\[\Lup^{\Lc}[v_1,\dots,v_d]=\Lup^{\Lc}[v_2,\dots,v_{d},v_{0}]=\Lup^{\Lc}[-v_1,\dots,-v_d]=(-1)^{d+1}\Lup^{\Lc}[v_d,\dots,v_1].
\]
We also have an identity $\I^{\Lc}[v_1,\dots,v_d]=(-1)^{d+1}\I^{\Lc}[v_d,\dots,v_1]$ and
\[
\I^{\Lc}[v_1-v_0,v_2-v_0,\dots,v_d-v_0]=\I^{\Lc}[v_2-v_1,v_3-v_1,\dots,v_0-v_1].
\]
\end{corollary}
\begin{proof} Dihedral relations follow from double shuffle relations, see \cite[Theorem 2.7]{Gon01B}. 
\end{proof}

\begin{corollary}
\label{CoxeterSpecialH} Let $H$ be a hyperplane in a $d$-dimensional vector space $V$. Then generic Coxeter pairs $[\Pc]\otimes [\Qc]$ with $Q_2,\dots,Q_d\in H$ generate $\StL(V)$ as a $\Q$-vector space.
\end{corollary}
\begin{proof}
Let $h$ be a linear functional with $H=\Ker(h)$. Recall from Proposition~\ref{PropositionNonGenericPairs} that non-generic Coxeter pairs vanish in $\StL(V)$. Therefore, by Theorem~\ref{thm:Stn_li_basis}, the elements $\Lup^\Lc[-v_1,v_1-v_2,\dots,v_{d-1}-v_d]$ span $\StL(V)$ as $v_1,\dots,v_d$ run over all $d$-tuples of vectors satisfying $h(v_i)=1$, $i=1,\dots,d$. By Corollary~\ref{CorollaryDihedral}
    \[\Lup^\Lc[-v_1,v_1-v_2,\dots,v_{d-1}-v_d] = \Lup^\Lc[v_1-v_2,\dots,v_{d-1}-v_d,v_d]
     = (-1)^d\I^\Lc[v_1,\dots,v_d],\]
which proves the claim, since generic Coxeter pairs $[\Pc]\otimes [\Qc]$ with $Q_2,\dots,Q_d\in H$ are exactly the elements of the form $\I[v_1,\dots,v_d]$ with $h(v_i)=1$, $i=1,\dots,d$.
\end{proof}

\begin{remark} \label{RemarkDualityGeneratingFunctions}
Note that the shuffle relations for $\Lup$ and $\I$ are dual to each other. This duality extends to higher weight polylogarithms via generating functions as follows. Assume that $F=\Q$. We define the following elements in $\StH(V)\otimes \Sbb(V)[[t_1,\dots,t_d]]$:
    \begin{align*}
    \Lup[v_1,\dots,v_d|t_1,\dots,t_d] \coloneqq {}& 
    \sum_{n_1,\dots,n_d>0}\Lup[v_1,\dots,v_d]\otimes \prod_{j=1}^{d}\frac{v_j^{n_j-1}t_j^{n_j-1}}{(n_j-1)!}\\
    ={}&\Lup[v_1,\dots,v_d]\otimes e^{t_1v_1+\dots+t_dv_d},
    \end{align*}
and 
    \begin{align*}
    \I[v_1,\dots,v_d|t_1,\dots,t_d] \coloneqq {} & 
    \sum_{n_1,\dots,n_d>0}\I[v_1,\dots,v_d]\otimes \prod_{j=1}^{d}\frac{(v_j-v_{j+1})^{n_j-1}(t_1+\dots+t_j)^{n_j-1}}{(n_j-1)!}\\
    ={}&\I[v_1,\dots,v_d]\otimes e^{t_1(v_1-v_2)+(t_1+t_2)(v_2-v_3)+\dots+(t_1+\dots+t_{d})(v_{d})}\\
    ={}&\I[v_1,\dots,v_d]\otimes e^{t_1v_1+\dots+t_dv_d},
    \end{align*}
where we set $v_{d+1}=0$.
Then shuffle relations~\eqref{eq:shuffle1} and~\eqref{eq:shuffle2} clearly also hold for these generating functions, for example
    \[\Lup[v_1,v_2|t_1,t_2]\Lup[v_3|t_3] = \Lup[v_1,v_2,v_3|t_1,t_2,t_3]
                                          +\Lup[v_1,v_3,v_2|t_1,t_3,t_2]
                                          +\Lup[v_3,v_1,v_2|t_3,t_1,t_2].\]
The two generating functions are now related to each other by
    \[\I[v_1,\dots,v_d|t_1,\dots,t_d] = \Lup[v_1-v_2,\dots,v_{d-1}-v_d,v_d|t_1,t_1+t_2,\dots,t_1+\dots+t_d].\]
Using Theorem~\ref{TheoremMain2} one can use this to derive double shuffle relations for multiple polylogarithms, see~\cite[\S\S2.5--2.7]{Gon01}.
Compare this also to the generating functions considered by Goncharov in~\cite[\S4.1]{Gon01B}.
\end{remark}

As an application, we prove that  generic Coxeter pairs generate $\StH(V)$.

 \begin{proposition}\label{TheoremGenericPairsGenerate} Generic Coxeter pairs generate $\StH(V)$ as a $\Q$-vector space.
\end{proposition}
\begin{proof}
We prove the claim by induction on $\dim(V)$. For $\dim(V)=1$ the result is trivial. By Corollary~\ref{CoxeterSpecialH}, together with the inductive assumption, $\StH(V)$ is spanned by generic Coxeter pairs and products of generic Coxeter pairs. Because of Lemma~\ref{LemmaCoxeterParameterization}, any generic Coxeter pair can be written as $\Lup[v_1,\dots,v_d]$. Finally, by Proposition~\ref{Proposition_double_shuffle}, any product $\Lup[v_1,\dots,v_{d_1}]\Lup[v_{d_1+1},\dots,v_{d_1+d_2}]$ can be written as linear combinations of $\Lup[w_1,\dots,w_d]$. Thus generic Coxeter pairs generate $\StH(V)$ as a vector space.
\end{proof}

The following corollary is a rationalized version of \cite[Theorem A]{GKRW20}.
\begin{corollary}\label{CorollaryRognes}
Let $V$ be a nonzero vector space of dimension $d.$ Then
\[
H_0(\GL(V), \StH(V))=\Q.
\]
\end{corollary}
\begin{proof}
Consider the augmentation map $a\colon \Bup_d\St(V)\lra \Q$ sending  $\sum n_i \bigl[P_1^i|\cdots|P_d^i\bigr]$ to $\sum n_i.$ The composition $a\circ s\colon \StH(V)\lra \Q$ is $\GL(V)$-equivariant and surjective, because 
\[
a\circ s([P_1,\dots,P_d]\otimes [P_1,\dots,P_d])=a\Bigg(\sum_{\sigma \in \Sfr_d}[P_{\sigma(1)}|\cdots|P_{\sigma(d)}]\Bigg)=d!.
\]
Thus, we obtain a surjective map $H_0\bigl(\GL(V), \StH(V)\bigr)\twoheadrightarrow\Q.$
On the other hand, Proposition~\ref{TheoremGenericPairsGenerate} and Lemma~\ref{LemmaCoxeterParameterization} together imply that $H_0\bigl(\GL(V),\StH(V)\bigr)$ is at most one-dimensional since $\GL(V)$ acts transitively on generic Coxeter pairs. From here the statement follows.
\end{proof}

In Section~\ref{SectionBasesSteinberg} we will use the following result.
\begin{proposition}
\label{CoxeterSpecialHInt} Let $h\in V^{\vee}$ be a nonzero linear functional on a $d$-dimensional vector space $V$. Then 
\begin{enumerate}
\item\label{CoxeterSpecialHIntParti} $\StL(V)$ is spanned over $\Q$ by elements of the form $\I^\Lc[v_1,\dots,v_d]$ with $h(v_1)=\dots=h(v_d)=1$. 
\item\label{CoxeterSpecialHIntPartii} $\StH(V)$ is spanned over $\Q$ by elements
of the form $\I[v_1,\dots,v_{d-k}]\I[w_1,\dots,w_k]$, $k=0,\dots,d-1$, satisfying $h(v_1)=\dots=h(v_{d-k})=1$ and $h(w_1)=\dots=h(w_k)=0$.
\end{enumerate}
\end{proposition}
\begin{proof}
Part~\ref{CoxeterSpecialHIntParti} follows from the proof of Corollary~\ref{CoxeterSpecialH}. 

We prove part~\ref{CoxeterSpecialHIntPartii} by induction on $\dim(V)$. By part~\ref{CoxeterSpecialHIntParti}, $\StH(V)$ is spanned by elements of the form $\I[v_1,\dots,v_d]$, $h(v_1)=\dots=h(v_d)=1$ (which correspond to $k=0$ in the formulation in~\ref{CoxeterSpecialHIntPartii} above) and decomposable elements. Consider a decomposable element $a\in m(\StH(V_1)\otimes \StH(V_2))$ for $V=V_1\oplus V_2$, $\dim(V_i)=d_i<d$. If $V_1,V_2\not\subseteq \Ker(h)$, then by induction applied to $\StH(V_i)$, $a$ lies in the span of
    \[\I[v_1,\dots,v_{d_1-k_1}]\I[w_1,\dots,w_{k_1}]\I[v_1',\dots,v_{d_2-k_2}']\I[w_1',\dots,w_{k_2}']\]
with $k_i<d_i$, $h(v_1)=\dots=h(v_{d_1-k_1})=1$, $h(v_1')=\dots=h(v_{d_2-k_2}')=1$, and $h(w_1)=\dots=h(w_{k_1})=0$, $h(w_1')=\dots=h(w_{k_2}')=0$. If $V_2\subseteq\Ker(h)$, then $V_1\not\subseteq \Ker(h)$, then the induction assumption for $\StH(V_1)$ and Proposition~\ref{TheoremGenericPairsGenerate} for $\StH(V_2)$ implies that $a$ lies in the span of
    \[\I[v_1,\dots,v_{d_1-k_1}]\I[w_1,\dots,w_{k_1}]\I[w_1',\dots,w_{d_2}'],\]
for $h(v_1)=\dots=h(v_{d_1-k_1})=1$,
$h(w_1)=\dots=h(w_{k_1})=0$, $h(w_1')=\dots=h(w_{d_2}')=0$.
In either case, after applying~\eqref{eq:shuffle2} we see that $a$ can be written as a linear combination of elements of the form $\I[v_1,\dots,v_{d-k}]\I[w_1,\dots,w_k]$ as in the statement of~\ref{CoxeterSpecialHIntPartii}, and since $a$ was an arbitrary decomposable element, this proves the induction step.
\end{proof}

\subsection{Cobracket of Steinberg polylogarithms}
In this section, we prove a formula for the cobracket of a Steinberg polylogarithm viewed as an element of the $\VB$-Lie coalgebra $\StL(V)$.

\begin{proposition}\label{PropositionCobracketCorrelators} Consider a  basis $v_1,\dots, v_d$ of a vector space $V$ and denote $v_0=-(v_1+\dots+v_d)$. Then we have
\begin{equation} \label{FormulaCobracketLiL}
\delta \left( \Lup^\Lc[v_1\dots, v_d]\right) =-\sum_{j=0}^d\sum_{i=1}^{d-1}
\Lup^\Lc[ v_{j+1}, \dots, v_{j+i}] \wedge \Lup^\Lc[  v_{j+i+1}, \dots, v_{j+d}] ,
\end{equation}
where indices are viewed modulo $d+1$.
\end{proposition}
\begin{proof}
For $d=1$ both sides vanish so we will assume that $d\geq2$.
First, we claim that in $\StH$ the $(k,d-k)$-part of the coproduct of $\Lup[v_d,\dots,v_1]$ is
    \begin{align} \notag
    \Delta_{k,d-k}\Lup[v_d,\dots,v_1]
    =\!\!\!\smash[b]{\sum_{\substack{0<i_1<\dots<i_k\leq d,\\
           i_{s{-}1}<j_s'\leq i_s,\ 1\leq s\leq k}}}  \!\!\! 
         (-1)&^{\sum_si_s-j_s'}  \Lup[v_{i_k}+\dots+v_{i_{k-1}+1},\dots,v_{i_1}+\dots+v_{1}]\\[-0.5ex]
        \label{FormulaCoproductLiSt}
        & {} \otimes
        \prod_{s=1}^{k+1}\Lup[v_{j_s'-1},v_{j_s'-2},\dots,v_{i_{s-1}+1}]\cdot \Lup[v_{j_s'+1},v_{j_s'+2},\dots,v_{i_s}],
    \end{align}
where we set $i_0=0$ and $i_{k+1}=j_{k+1}'=d+1$ and take $\Lup$ of an empty sequence to be $1$. To prove the above formula we apply~\eqref{FormulaCoproductLA} to $\Lup[v_d,\dots,v_1]=[w_1,w_2,\dots,w_d]\otimes [v_1,\dots,v_d]$, where $w_i=v_1+\dots+v_i$. 
For a term in~\eqref{FormulaCoproductLA} with $|I|=k$, we have $A_I=\langle v_1+\dots+v_{i_1}, v_{i_1+1}+\dots+v_{i_2},\dots,v_{i_{k-1}+1}+\dots+v_{i_k}\rangle$ and $B_J=\langle v_{j_1},\dots v_{j_{d-k}}\rangle$. Then $V=A_I\oplus B_J$ if and only if $1\leq j_1'\leq i_1$ and $i_{s-1}<j_s'\leq i_s$ for $2\leq s\leq k$. Indeed, to have $0=A_I\cap B_J$ we must have at least one element of $\overline{J}$ in each interval $[1,i_1], [i_1+1,i_2],\dots, [i_{k-1}+1,i_k]$, and since $|\overline{J}|=k$, this accounts for all elements in $\overline{J}$. Next, 
we calculate $A_I\cap \langle B_J,w_{j_s'}\rangle = \langle v_{i_{s-1}+1}+\dots+v_{i_s}\rangle$, and hence
    \begin{align*}
    (A,B)_{I,J} &= [w_{i_1},\dots,w_{i_k}]\otimes 
    [v_1+\dots+v_{i_1},v_{i_1+1}+\dots+v_{i_2},\dots,v_{i_{k-1}+1}+\dots+v_{i_k}]\\
    &= \Lup[v_{i_k}+\dots+v_{i_{k-1}+1},\dots,v_{i_1}+\dots+v_{1}].
    \end{align*}
A similar calculation gives
    \begin{align*}
    (A,B)^{I,J} = \Lup[v_d,v_{d-1},\dots,v_{i_k+1}]\cdot \prod_{s=1}^{k}\Big(\Lup[v_{j_s'-1},v_{j_s'-2},\dots,v_{i_{s-1}+1}]\cdot \Lup[v_{j_s'+1},v_{j_s'+2},\dots,v_{i_s}]\Big).
    \end{align*}
Finally, the sign is $(-1)^{\sigma_I}(-1)^{\sigma_J} = (-1)^{\sigma_{I,J}}$, where $\sigma_{I,J}\in \Sfr_d$ is the unique permutation mapping $I$ to $\overline{J}$ and $\overline{I}$ to $J$ that is monotone on both $I$ and $\overline{I}$. Counting the number of inversions (using the fact that $I$ and $\overline{J}$ are interlaced) gives $(-1)^{\sigma_{I,J}}=(-1)^{\sum i_s-j_s'}$.

Next, we compute the projection of~\eqref{FormulaCoproductLiSt} to $\StL(V)$. Any term in~\eqref{FormulaCoproductLiSt} involving a nontrivial product automatically vanishes in~$\StL(V)$. The remaining terms are the ones for which $I=\{1,2,\dots,{k-l},{d-l+1},d-l+2,\dots,d\}$, for some $0\leq l\leq k$, and $j_{k-l+1}\in \{k-l+1,d-l+1\}$. Then the projection of the right hand side of~\eqref{FormulaCoproductLiSt} to $\StL$ is
    \begin{align*}
    \Lup^\Lc[v_k,\dots,v_1]\otimes \Lup^\Lc[v_d,\dots,v_{k+1}]+\smash[b]{\sum_{l=1}^{k}}&\Lup^\Lc[v_d,\dots,v_{d-l+2},v_{d-l+1}+\dots+v_{k-l+1},v_{k-l},\dots,v_1]\\
    & \,\,\,\, {}\otimes \big(\Lup^\Lc[v_{d-l},\dots,v_{k-l+1}]+(-1)^{d-k}\Lup^\Lc[v_{k-l+2},\dots,v_{d-l+1}]\big).
    \end{align*}
(The $l=1$ term in the sum begins with $\Lup^\Lc[v_d+\dots+v_{k},v_{k-1},\dots,v_1]$.)
By dihedral relations for $\Lup^{\Lc}$
    \[\Lup^\Lc[v_d,\dots,v_{d-l+2},v_{d-l+1}+\dots+v_{k-l+1},v_{k-l},\dots,v_1] = \Lup^\Lc[v_{k-l},\dots,v_1,v_0,v_d,\dots,v_{d-l+2}],\]
and so we can rewrite $\Delta_{k,d-k}\Lup^\Lc[v_d,\dots,v_1]$ as
    \begin{align*}
    \sum_{l=0}^{k}&\Lup^\Lc[v_{k-l},\dots,v_1,v_0,v_d,\dots,v_{d-l+2}]
    \otimes \Big(\Lup^\Lc[v_{d-l},\dots,v_{k-l+1}]-\Lup^\Lc[v_{d-l+1},\dots,v_{k-l+2}]\Big),
    \end{align*}
where for $l=0$ the term $\Lup^\Lc[v_{d+1},\dots,v_{k+2}]$ should be omitted. 
Summing this up over $k=1,\dots,d-1$, anti-symmetrizing the tensor products, and rearranging the terms we get
    \[\delta \left( \Lup^\Lc[v_d\dots,v_1]\right) = \sum_{j=0}^d\sum_{i=1}^{d-1}
    \Lup^\Lc[ v_{j+i}, \dots, v_{j+1}] \wedge \Lup^\Lc[v_{j+d}, \dots, v_{j+i+1}] .\]
Finally, applying dihedral relation $\Lup^{\Lc}[w_r,\dots,w_1]=(-1)^{r+1}\Lup^{\Lc}[w_1,\dots,w_r]$ to every polylogarithm in the above formula we get the claim of the proposition.
\end{proof}

\subsection{Steinberg correlators}\label{SectionSteinbergCorrelators}
In this section, we discuss yet another parameterization of generic Coxeter pairs, which, for $F=\Q$, is related to Hodge/motivic correlators \cite{Gon19}, \cite[\S 2]{GR18}. These elements naturally live in the $\VB$-Lie coalgebra $\StL(V)$. 

Let $V$ be a vector space of dimension $d$ over $F$. Let $v_1,\dots,v_d$ be a basis of $V$ and put  ${v_0=-(v_1+\dots+v_d)}$.  A \emph{Steinberg correlator} is an element 
\[
\CupL[v_0,\dots,v_d]\coloneqq \Lup^{\Lc}[v_1,\dots,v_d]\in \StL(V).
\]
We will also use a different parameterization: for affinely independent vectors $u_0,\dots,u_d \in V$ put
\[
\CupL[u_0:u_1:\dots:u_d]\coloneqq \CupL[u_0-u_1, u_1-u_2,\dots, u_d-u_0].
\]
(Recall that vectors $u_0,\dots,u_d$ are called affinely independent if the vectors $u_1-u_0,\dots,u_d-u_0$ are linearly independent.)
Clearly, $\CupL[u_0:\dots:u_d]=\CupL[u_0+u:\dots:u_d+u]$ for any $u\in V$.

By \eqref{FormulaLviaI},  we have 
\[
\CupL[u_0:u_1:\dots:u_d]=\CupL[u_0-u_1, u_1-u_2,\dots, u_d-u_0]=(-1)^d\I^\Lc[u_1-u_0,u_2-u_0,\dots,u_d-u_0].
\]

Corollary~\ref{CorollaryDihedral} implies that both versions of Steinberg correlators are cyclically symmetric:
\[
\CupL[v_0,v_1,\dots,v_d]=\CupL[v_1,v_2,\dots,v_0], \quad \CupL[u_0:u_1:\dots:u_d]=\CupL[u_1:u_2:\dots:u_0].
\]
Next, Proposition~\ref{Proposition_double_shuffle} implies that the following shuffle relations hold:
\begin{align*}
&\sum_{\sigma\in \Sigma_{d_1,d_2}}\CupL[v_0,v_{\sigma(1)},\dots,v_{\sigma(d_1+d_2)}]=0,\\
&\sum_{\sigma\in \Sigma_{d_1,d_2}}\CupL[u_0:u_{\sigma(1)}:\dots:u_{\sigma(d_1+d_2)}]=0.
\end{align*}
Finally, Proposition \ref{PropositionCobracketCorrelators} implies that
\begin{equation} \label{FormulaCobracketCL}
\delta \left( \CupL[u_0:u_1:\dots:u_d]\right)=\sum_{j=0}^d\sum_{i=1}^{d-1}
\CupL[ u_{j}: u_{j+1}: \dots: u_{j+i}] \wedge \CupL[ u_j: u_{j+i+1}: \dots: u_{j+d}] .
\end{equation}
Indeed, if we denote $v_i=u_i-u_{i+1}$ (with indices modulo $d+1$), then
    \begin{align*}
      \Lup^{\Lc}[v_{j+1},\dots,v_{j+i+1}] &= 
      \Lup^{\Lc}[u_{j+1}-u_{j+2},\dots,u_{j+i+1}-u_{j+i+2}]\\
    &= \CupL[u_{j+i+2}:u_{j+1}:u_{j+2}:\dots:u_{j+i+1}] =
      \CupL[u_{j+1}:u_{j+2}:\dots:u_{j+i+2}],
    \end{align*}
where in the last equality we used Corollary~\ref{CorollaryDihedral}. Thus, the formula~\eqref{FormulaCobracketLiL} can be seen to directly imply~\eqref{FormulaCobracketCL}.
For $F=\Q$, formula~\eqref{FormulaCobracketCL} is related to the formula for the cobracket of correlators in \S\ref{SectionMultiplePolylogs}.

Consider a collection of points $A_0, \dots, A_d\in \PP(V)$ in general position. There exist nonzero vectors $v_i\in A_i$ such that $v_0+\dots+v_d=0$.  We put
\[
\CupL[A_0,\dots,A_d]\coloneqq \CupL[v_0,\dots,v_d]\in \StL(V).
\]
 Notice that the collection of vectors $v_i$ is defined uniquely up to proportionality, so the expression above is well-defined.

Next, we discuss the duality for Steinberg correlators. We start with a lemma. 

\begin{lemma} \label{LemmaLinearAlgebra}
Consider a collection of nonzero vectors $v_0, \dots, v_d\in V$ and nonzero functionals $h_0,\dots,h_d\in V^{\vee}$ such that 
\begin{enumerate}
\item $v_0+\dots+v_d=0$, 
\item $h_0+\dots+h_d=0$,
\item $h_{i}(v_j)=0$ for $i\not \equiv j  \Mod{(d+1)}$ and $i\not \equiv j+1 \Mod{(d+1)}$,
\item $h_d(v_d)=1$.
\end{enumerate}
Then the basis $h_1,\dots,h_d$ is dual to the basis $v_1+v_{2}+\dots+v_d, v_{2}+\dots+v_d,\dots, v_d$. 
\end{lemma}
\begin{proof} 
Notice that
$v_{i-1}+v_i=-\sum_{j\neq i,i-1} v_j,$
so 
\begin{equation}\label{FormulaDualityCorrelators1}
h_{i}(v_{i-1})+h_{i}(v_{i})=-\sum_{j\neq i,i-1} h_i(v_j)=0.
\end{equation}
Similarly, $h_{i}(v_{i})+h_{i+1}(v_{i})=0.$ For $1\leq i\neq j\leq d$ we have 
\begin{equation}\label{FormulaDualityCorrelators2}
h_i(v_j+v_{j+1}+\dots+v_d)=0. 
\end{equation}
Indeed, for $i<j$ we have $h_i(v_j)=h_i(v_{j+1})=\dots=h_i(v_d)=0$, which implies (\ref{FormulaDualityCorrelators2}). For $i>j$ we also have 
\[
h_i(v_j+v_{j+1}+\dots+v_d)=-h_i(v_0+v_{1}+\dots+v_{j-1})=-(h_i(v_0)+\dots+ h_i(v_{j-1}))
=0.\]
Next, for $1\leq i\leq d$ we have 
\begin{equation}\label{FormulaDualityCorrelators3}
h_i(v_i+v_{i+1}+\dots+v_d)=h_i(v_i)= -h_{i+1}(v_i)=h_{i+1}(v_{i+1})=h_{i+1}(v_{i+1}+v_{i+2}+\dots+v_d).
\end{equation}
Applying (\ref{FormulaDualityCorrelators3}) repeatedly, we get that $h_i(v_i+v_{i+1}+\dots+v_d)=h_d(v_d)=1$ for $1\leq i\leq d$. This finishes the proof.
\end{proof}

A collection of points $A_0,\dots,A_d\in \PP(V)$ in general position determines a collection of hyperplanes $H_0,\dots,H_d$ in general position by taking
\[
H_i=\langle A_0,\dots,A_{i-2}, \,\widehat{\!A_{i-1}\!}\,, \widehat{A_{i\,}}, A_{i+1},\dots,A_d\rangle \quad \text{for}\quad i\in \Z/(d+1)\Z.
\]
The correlator $\CupL[H_0,\dots,H_d]$ is an element of $\StL(V^{\vee})$. We have the following  duality property. 
\begin{proposition}\label{LemmaCorrelatorDuality} 
The following identity holds in $\StL(V)$:
\[
D(\CupL[H_0,\dots,H_{d-1},H_d])=(-1)^{d+1}\CupL[A_0,\dots,A_{d-1},A_{d}].
\]
\end{proposition}
\begin{proof}
Consider a collection of nonzero vectors $v_0, \dots, v_d\in V$ such that $v_i\in A_i$ and ${v_0+\dots+v_d=0}.$ Next, consider a collection of nonzero functionals $h_0,\dots,h_d\in V^{\vee}$ such that $H_i=\Ker(h_i),$ ${h_0+\dots+h_d=0},$ and $h_d(v_d)=1$. Proposition \ref{PropositionDuality} implies that
\begin{equation}\label{FormulaProofCorrelatorDuality1}
D(\CupL[H_0,\dots,H_{d-1},H_d])=D(\Lup^{\Lc}[h_1,\dots,h_d])=(-1)^d\I^{\Lc}[h^d,\dots,h^1]
=-\I^{\Lc}[h^1,\dots,h^d]
\end{equation}
where $h^1,\dots,h^d$ is the dual basis to the basis $h_1,\dots,h_d$ of $V^{\vee}$.
By Lemma~\ref{LemmaLinearAlgebra}, we have ${h^i=v_i+v_{i+1}+\dots+v_d}$, so
\begin{equation}\label{FormulaProofCorrelatorDuality2}
-\I^{\Lc}[h^1,\dots,h^d]=-\I^{\Lc}[v_1+\dots+v_d,v_2+\dots+v_d,\dots, v_d]=(-1)^{d+1}\Lup^{\Lc}[v_1,\dots,v_d].
\end{equation}
The statement follows from (\ref{FormulaProofCorrelatorDuality1}) and (\ref{FormulaProofCorrelatorDuality2}).
\end{proof}

\begin{remark}\label{RemarkGoncharovDihedral} In \cite[\S 4]{Gon98}, Goncharov introduced the  modular complex $M^{\bullet}(L)$, for a lattice $L$ of finite rank; it is closely related to the Chevalley-Eilenberg complex of  $\StL(\Q^d)$. It is easy to see that $M^{1}$ is a $\VB_{\Z}$-module, see Remark \ref{RemarkVBZ}. Moreover, it is a  $\VB_{\Z}$-Lie coalgebra and $M^{\bullet}$ is its Chevalley-Eilenberg complex in $\VB_{\Z}$-sense.  As an abelian group, $M^{1}(L)$ is generated by certain elements $\llbracket v_1,\dots,v_d \rrbracket$ where $v_1,\dots,v_d$ is a basis of $L$.\footnote{In \cite{Gon98}, Goncharov uses the notation $[v_1,\ldots,v_d]$, but this clashes with our notation for apartments, hence we write $\llbracket v_1,\ldots,v_d \rrbracket$, to avoid confusion. Similarly, we use $\llangle\cdots\rrangle$ in place of Goncharov's $\langle\cdots\rangle$ notation.} Goncharov introduced two more parameterizations of the same generating set. First, for every basis $v_1,\dots,v_d$ consider the unique vector $v_0$ such that $v_0+v_1+\dots+v_d=0$ and define
 \[
 \llangle v_0,v_1, \dots,v_d\rrangle \coloneqq \llbracket v_1,\dots,v_d \rrbracket.
 \]
Next, consider  a tuple of vectors $u_0, \dots, u_d\in L$  such that $(u_0,1),\dots, (u_d,1)$ is a basis of $L\oplus\Z$ and define
 \[
 \llangle u_0:\dots:u_{d}\rrangle\coloneqq  \llangle u_0',\dots,u_{d}'\rrangle \quad \text{for} \quad u_i'=u_{i+1}-u_i,
 \]
where indices are considered modulo $d+1$.

It is easy to see that for every lattice $L$ of rank $d$ there exists a unique $\GL(L)$-equivariant map $\Sup\colon M^1(L) \lra \StL(L_{\Q})$ such that
\begin{align*}
&\Sup(\llbracket v_1,\dots,v_d\rrbracket)=\Lup^\Lc[v_1,\dots,v_d],\\
&\Sup(\llangle v_0,\dots,v_d\rrangle )=\CupL[v_0, v_1,\dots,v_d], \\
&\Sup(\llangle u_0:\dots:u_d\rrangle )=\CupL[u_0: u_1:\dots:u_d]. 
\end{align*}
Moreover, $\Sup$ is a map of $\VB_{\Z}$-Lie coalgebras.
\end{remark}

\subsection{Bases for \texorpdfstring{$\StL_d$}{St\textasciicircum{}∞\_d} and \texorpdfstring{$\StH_d$}{St\textasciicircum{}2\_d}} 
\label{SectionBasesSteinberg}

In this section, we construct a family of bases for $\StL_d$ parameterized by hyperplanes in $V.$ Let $H$ be a hyperplane in a nonzero vector space $V$ of dimension $d.$ The space
\[
\Bup_d^H\St(V) = \bigoplus_{\substack{V=P_1\oplus\dots\oplus P_d \\ P_i\not \subseteq H}}\St(P_1)\otimes \dots \otimes \St(P_d)
\]
is a direct summand of $\Bup_d\St(V).$ The projection $\Bup_d\St(V)\lra \Bup_d^H\St(V)$ induces the projection 
\[
p_H\colon \Bup_d\St(V)\otimes_{\Sfr_d}Lie^c_d\lra \Bup_d^H\St(V)\otimes_{\Sfr_d}Lie^c_d.
\]
As before, $Lie^c$ is the coLie cooperad, see \S\ref{SectionRecapOfLieCoalgebras}.

The main result of this section is the following theorem.

\begin{proposition} \label{TheoremSteinbergLBasis}
Let $V$ be a nonzero vector space of dimension $d;$ let $H\subseteq V$ be a hyperplane. The composition 
\[
\StL(V)\stackrel{s}{\lra} \Bup_d\St(V)\otimes_{\Sfr_d}Lie^c_d\stackrel{p_H}{\lra} \Bup_d^H\St(V)\otimes_{\Sfr_d}Lie^c_d
\]
is an isomorphism of $\Q$-vector spaces.  
\end{proposition} 

\begin{proof}
Assume that the hyperplane $H$ is defined by an equation $h(v)=0$ for $h\in V^{\vee}$. Consider a map $\I_H\colon \Bup_d^H\St(V)\lra \StL(V)$ defined by the formula
    \[\I_H([P_1|\cdots|P_d])=(-1)^d\I^\Lc[v_1,\dots,v_d],\]
where the vectors $v_i\in P_i$ are such that $h(v_i)=1$. Proposition~\ref{Proposition_double_shuffle} implies that this map vanishes on all nontrivial shuffles, and thus factors through the projection
    \[\Bup_d^H\St(V)\lra \Bup_d^H\St(V)\otimes_{\Sfr_d}Lie^c_d.\] 
Therefore it gives a map $\I_H\colon \Bup_d^H\St(V)\otimes_{\Sfr_d}Lie^c_d \lra \StL.$ 
Surjectivity of $\I_H$ follows from Proposition~\ref{CoxeterSpecialHInt}\ref{CoxeterSpecialHIntParti}.

Notice that $v_i-v_j\in H$, so Lemma \ref{LemmaLCoproductComponent} implies that $s(\I^{\Lc}[v_1,\dots,v_d])$ has only one term $\pm [P_1|\dots|P_d]$ with all $P_i$ not in $H$. Therefore, we have
\[
p_H(s(\I_H([P_1|\cdots|P_d])))=[P_1|\cdots|P_d],
\]
and thus $p_H\circ s \circ \I_H =\Id.$ It follows that $\I_H$ is injective, so an isomorphism. Furthermore, $p_H\circ s$ is its inverse.
\end{proof}

The above construction can be extended to give a family of bases for the $\VB$-Hopf algebra $\StH_d$. For a hyperplane~$H$ and $0\leq k<d$ we define the following subspaces of $\Bup_d\St(V)$:
\[
\Bup_d^{H,k}\St(V) = \!\!\! \bigoplus_{\substack{V=U\oplus W \\ W\subseteq H,\, \dim(W)=k}}\!\!\!\Bup_{d-k}^{H\cap U}\St(U) \otimes \StH(W).
\]
Note that $\Bup_d^{H,0}\St(V) = \Bup_d^{H}\St(V)$ and $\Bup_d^{H,k}\St(V)\subseteq\widetilde{\Bup}_d^{H,k}\St(V)$, where
\[
\widetilde{\Bup}_d^{H,k}\St(V) \coloneqq  \!\!\! \bigoplus_{\substack{V=U\oplus W \\ W\subseteq H,\, \dim(W)=k}} \!\!\! \Bup_{d-k}^{H\cap U}\St(U) \otimes \Bup_k\St(W).
\]
The space $\widetilde{\Bup}_d^{H,k}\St(V)$ has a canonical projection $r_{H,k}$ from $\Bup_d\St(V)$; denote $r_H=\bigoplus_{k}r_{H,k}$.

\begin{proposition} \label{TheoremSteinbergHBasis}
Let $V$ be a nonzero vector space of dimension $d$ and let $H\subseteq V$ be a hyperplane. Then the composition 
\[
\StH(V)\stackrel{s}{\lra} \Bup_d\St(V)\stackrel{r_{H}}{\lra} \bigoplus_{k=0}^{d-1}\Bup_d^{H,k}\St(V)
\]
is an isomorphism of $\Q$-vector spaces.  
\end{proposition}
\begin{proof}
First, we note the following chain of inclusions
    \[\StH(V)
    \subseteq\bigoplus_{\substack{U\oplus W=V\\ \dim(W)=k}} \StH(U)\otimes\StH(W)
    \subseteq\bigoplus_{\substack{U\oplus W=V\\ \dim(W)=k}} \Bup_{d-k}\St(U)\otimes\Bup_{k}\St(W)
    = \Bup_d\St(V)
    \]
that directly follows the definition of the differential $\partial_k$ from~\eqref{BarDifferential}. In particular, this implies that the projection from $\StH(V)$ to $\widetilde{\Bup}_d^{H,k}\St(V)$ indeed lands in $\Bup_d^{H,k}\St(V)$.

As in the proof of Proposition~\ref{TheoremSteinbergLBasis}, we assume that $H=\Ker(h)$ for a nonzero linear functional $h\in V^{\vee}$. Consider the mapping $\varphi\colon \bigoplus_{k=0}^{d-1}\Bup_d^{H,k}\St(V)\to \StH(V)$, defined on each direct summand $\Bup_{d-k}^{H\cap U}\St(U) \otimes \StH(W)$ by
    \[\varphi([P_1|\cdots|P_{d-k}]\otimes a) = (-1)^{d-k}\I[v_1,\dots,v_{d-k}]\cdot a,\]
where the vectors $v_i\in P_i$ are chosen to satisfy $h(v_i)=1$. A direct calculation using Lemma~\ref{LemmaLCoproductComponent} shows that the composition $r_{H,k}\circ s\circ \varphi$ restricts to identity on each $\Bup_d^{H,k}\St(V)$ and moreover, the image under $r_H\circ s\circ \varphi$ of $\Bup_d^{H,k}\St(V)$ is contained in $\bigoplus_{j\geq k}\Bup_d^{H,j}\St(V)$. This shows that $r_H\circ s\circ \varphi$ is an isomorphism, and so $\varphi$ is injective. 

To see that $\varphi$ is surjective, recall from Proposition~\hyperref[CoxeterSpecialHIntPartii]{\ref*{CoxeterSpecialHInt}\ref*{CoxeterSpecialHIntPartii}} that $\StH(V)$ is spanned by elements of the form $\I[v_1,\dots,v_{d-k}]\I[w_1,\dots,w_k]$ with $w_1,\dots,w_k\in H$ and $h(v_1)=\cdots=h(v_{d-k})=1$ for some linear functional $h$ satisfying $\Ker(h)=H$. Surjectivity then follows from the calculation 
    \[\varphi([P_1|\cdots|P_{d-k}]\otimes \I[w_1,\dots,w_k]) = (-1)^{d-k}\I[v_1,\dots,v_{d-k}]\I[w_1,\dots,w_k],\]
where $P_i=\langle v_i\rangle$, $i=1,\dots,d-k$.
\end{proof}

Iteratively applying the map $\varphi$ constructed in the proof of the above theorem, one gets a basis of $\StH(V)$ consisting of Coxeter pairs. Here is an example.
For $\dim(V)=2$ and a hyperplane $H\subseteq V$, choose a functional $h$ such that $H=\Ker(h)$ and a nonzero vector $w\in H$. 
Then the union of the sets
\begin{alignat*}{2}
& \bigl\{\I[v_1,v_2] &&\setsep v_1, v_2\ \text{is a basis of $V$},\, h(v_1)=h(v_2)=1\bigr\}
,
\\
 & \bigl\{\I[v]\cdot\I[w] &&\setsep v \in V,\,h(v)=1\bigr\}
\end{alignat*}
forms a basis for $\StH(V)$.

\section{The Steinberg module and Milnor \texorpdfstring{$K$}{K}-theory} \label{SectionSteinbergMilnor}

The goal of this section is to explain a connection between the Steinberg module $\St_d(\Q)$ and Milnor $K$-theory. The results of this section are not used later, yet we think that they clarify the origin of the connection between $\St_d(\Q)$ and polylogarithms on the torus.
We begin by explaining the main idea. The theorem of Ash--Rudolph \cite[Theorem 4.1]{AR79} states that $\St_d(\Q)$ is generated as a $\Q$-vector space by integral apartments: elements $[v_1,\dots,v_d]\in\St_d(\Q)$ such that $v_i\in \Z^d$ and $\det(v_1,\dots,v_d)=\pm1$. The theorem of Bykovski\u{\i} \cite{Byk}  states that all relations between integral apartments follow from the following three:
\begin{align}
    \tag*{{\bf (R1)}} \label{bykovskii:rel1x} [v_{\sigma(1)},v_{\sigma(2)},\dots,v_{\sigma(d)}]&=(-1)^{\sigma}[v_1,v_2,\dots,v_d] \text{ for }\sigma \in \Sfr_d, \\
    \tag*{{\bf (R2)}} \label{bykovskii:rel2x} [v_1,v_2,\dots,v_d]&=[-v_1,v_2,\dots,v_d], \\
    \tag*{{\bf (R3)}} \label{bykovskii:rel3x}
    [v_1,v_2,v_3,\dots,v_d]&=[v_1,v_1+v_2,v_3,\dots,v_d]+[v_1+v_2,v_2,v_3,\dots,v_d].
\end{align} 
Using the presentation above, it is possible to construct an injective map $\Kmap_d$ from the Steinberg module $\St_d(\Q)$ to a certain quotient of the Milnor $K$-group of the function field of a $d$-dimensional algebraic torus over $\overline{\Q}$. In \S \ref{SectionARByk}, we present a different construction of this map, which does not use the Bykovski\u{\i} presentation of $\St_d(\Q)$. Related constructions have appeared in the literature before, see \cite[\S 3]{SV24} and \cite[Theorem 1]{XU24}.

In \S \ref{SectionARByk}, we give a new proof of the theorems of Ash--Rudolph and Bykovski\u{\i}, which is inspired by the construction of the norm map on Milnor $K$-theory.

\subsection{The Steinberg module and Milnor \texorpdfstring{$K$}{K}-theory} \label{SteinbergAndMilnor}

The Milnor ring $K^M(F)$ of a field $F$ is a quotient of the tensor algebra on $F^{\times}$ by the homogeneous ideal generated by elements $a\otimes (1-a)$ for $a, 1-a \in F^{\times}$, see \cite{Mil69}. It follows that $K_0^M(F)\cong\Z$ and $K_1^M(F)\cong F^{\times}$. We denote the projection of $a_1\otimes \dots \otimes a_n\in (F^{\times})^{\otimes n}$ to $K_n^M(F)$ by $\{a_1,\dots,a_n\}$. The Milnor ring is graded-commutative \cite[Lemma 1.1]{Mil69}.

Next, we discuss residue  maps and norm maps on Milnor \texorpdfstring{$K$}{K}-groups. Let $\nu\colon F^{\times}\lra \Z$ be a discrete valuation with valuation ring $\mathcal{O}_\nu$, residue field $F_\nu$, and uniformizer $\pi$. For a unit $u\in \mathcal{O}_\nu^{\times}$ its residue class $\overline{u}$ lies in  $F_\nu^{\times}$. The residue map $\res_\nu\colon K_n^M(F)\lra K_{n-1}^M(F_\nu)$ is uniquely characterized by the properties:
\begin{align*}
\res_\nu(\{u_1, \dots, u_n\}) &=0 \,, \text{ and } \\
\res_\nu(\{\pi, u_2,\dots, u_{n}\}) &=\{\overline{u}_2,\dots, \overline{u}_n\}
\end{align*}
for any units $u_1, \dots, u_n \in \mathcal{O}_\nu^{\times}$. The residue map does not depend on the choice of the uniformizer.

For a finite field extension $L/F$ there exists a norm map $\Nm_{L/F}\colon K_n^M(L)\lra K_n^M(F)$, see  \cite{BT73, Kat79}. The construction is quite involved in general but simplifies when working with rationalized Milnor $K$-groups and Galois extensions. In this case, the norm map can be described as follows: for $a_1,\dots,a_n \in L^{\times}$, the element
\[
\sum_{\sigma \in \Gal(L/F)}\{\sigma(a_1),\dots,\sigma(a_n)\}
\]
lies in the image of the embedding $K_n^M(F)_{\Q} \hookrightarrow K_n^M(L)_{\Q}$; the corresponding element in  $K_n^M(F)_{\Q}$ equals to $\Nm_{L/F}(\{a_1,\dots,a_n\})$, see \cite{Sus79}. 

For an algebraic torus $\Tup$, we have an embedding $\Xup(\Tup)\hookrightarrow K_1^M(\overline{\Q}(\Tup))$ sending a character $\chi\colon \Tup\lra \Tup_1$ to the corresponding element $\chi \in \overline{\Q}(\Tup)^{\times}$ (here $\overline{\Q}(\Tup)$ is the field of rational functions on $\Tup$, equal to $\overline{\Q}(x_1,\dots,x_d)$). We have a map
\begin{equation} \label{FormulaMultiplicationMap}
K_{n-1}^M\bigl(\overline{\Q}(\Tup)\bigr) \otimes \Xup(\Tup)\lra K_{n-1}^M\bigl(\overline{\Q}(\Tup)\bigr) \otimes  K_1^M(\overline{\Q}(\Tup)) \lra K_n^M\bigl(\overline{\Q}(\Tup)\bigr),
\end{equation} 
given as the composition of the embedding discussed above and multiplication in the Milnor ring. We introduce the following notation for the rationalization of the cokernel of the map above:
\[
\Kbar_n(\Tup):= \Coker\left(K_{n-1}^M\bigl(\overline{\Q}(\Tup)\bigr) \otimes \Xup(\Tup) \lra K_n^M\bigl(\overline{\Q}(\Tup)\bigr) \right )\otimes_{\Z} \Q.
\]

The main result of this section is the following theorem. 
 
\begin{theorem} \label{TheoremSteinbergMilnor}
There exists a unique $\Q$-linear map
\[
\Kmap\colon \St_d(\Q)\lra  \Kbar_d(\Tup_d)
\]
such that for an integral apartment $[v_1,\dots,v_d]\in \St_d(\Q)$ we have
\begin{equation}\label{FormulaImageofIntegralAppartment}
\Kmap([v_1,\dots,v_d])=\{1-x^{v_1},\dots, 1-x^{v_d}\} \in \Kbar_d(\Tup_d).
\end{equation}
\end{theorem}
The uniqueness follows immediately from the theorem of Ash--Rudolph. We give two proofs of the existence statement. The first proof relies on the theorem of Bykovski\u{\i}. The second proof uses the properties of the norm map on $K$-groups.
\begin{proof}[First proof of Theorem \ref{TheoremSteinbergMilnor}]
The existence can be easily deduced from the theorem of Bykovski\u{\i}. Indeed, it is sufficient to show that the map $\Kmap$ annihilates relations \ref{bykovskii:rel1x}--\ref{bykovskii:rel3x}. 
 For \ref{bykovskii:rel1x}, we have 
\[
\Kmap([v_{\sigma(1)},\dots, v_{\sigma(d)}])=\{1-x^{v_{\sigma(1)}},\dots, 1-x^{v_{\sigma(d)}}\}=(-1)^{\sigma}\{1-x^{v_1},\dots, 1-x^{v_d}\}= \Kmap((-1)^{\sigma}[v_1,\dots,v_d]).
\]
For \ref{bykovskii:rel2x}, we have 
\begin{align*}
\Kmap([-v_1,v_2, \dots,v_d])&=\{1-x^{-v_1},1-x^{v_2},\dots, 1-x^{v_d}\}\\
&=\{-(1-x^{v_1}),1-x^{v_2},\dots, 1-x^{v_d}\}-\{x^{v_1},1-x^{v_2},\dots, 1-x^{v_d}\}\\
&=\Kmap([v_1,v_2, \dots,v_d])+\{-1,1-x^{v_2},\dots, 1-x^{v_d}\}-\{x^{v_1},1-x^{v_2},\dots, 1-x^{v_d}\}.
\end{align*}
The element $\{x^{v_1},1-x^{v_2},\dots, 1-x^{v_d}\}$ lies in the image of \eqref{FormulaMultiplicationMap} and thus vanishes in $\Kbar_d\bigl(\Tup_d\bigr).$ The element $\{-1,1-x^{v_2},\dots, 1-x^{v_d}\}$ is $2$-torsion, and thus vanishes since $\Kbar_d\bigl(\Tup_d\bigr)$ is a $\Q$-vector space.
So, $\Kmap$ annihilates \ref{bykovskii:rel2x}.

It remains to show that $\Kmap$  annihilates \ref{bykovskii:rel3x}. It is sufficient to show that the identity
\[
\{1-x^{v_1}, 1-x^{v_2}\}  -\{ 1-x^{v_1},1-x^{v_1+v_2}\} -\{1-x^{v_1+v_2},1-x^{v_2}\}=0
\]
holds in $\Kbar_2\bigl(\Tup_d)$.
First, observe that since $\{1-x^{v_1+v_2},1-x^{v_1+v_2}\}$ is torsion, we have
\[
\{1-x^{v_1}, 1-x^{v_2}\}  -\{ 1-x^{v_1},1-x^{v_1+v_2}\} -\{1-x^{v_1+v_2},1-x^{v_2}\}
=\Bigl\{\frac{1-x^{v_1}}{1-x^{v_1+v_2}},\frac{1-x^{v_2}}{1-x^{v_1+v_2}}\Bigr\}.
\]
Next, 
\[\Bigl\{\frac{1-x^{v_1}}{1-x^{v_1+v_2}},\frac{1-x^{v_2}}{1-x^{v_1+v_2}}\Bigr\}= \Bigl\{\frac{1-x^{v_1}}{1-x^{v_1+v_2}},\frac{(1-x^{v_2})x^{v_1}}{1-x^{v_1+v_2}}\Bigr\}-\Bigl\{\frac{1-x^{v_1}}{1-x^{v_1+v_2}},x^{v_1}\Bigr\}
\]
The element $\bigl\{\frac{1-x^{v_1}}{1-x^{v_1+v_2}},\frac{(1-x^{v_2})x^{v_1}}{1-x^{v_1+v_2}}\bigr\}$ vanishes since 
\[
\frac{1-x^{v_1}}{1-x^{v_1+v_2}}+\frac{(1-x^{v_2})x^{v_1}}{1-x^{v_1+v_2}}=1.
\]
The element $\bigl\{\frac{1-x^{v_1}}{1-x^{v_1+v_2}},x^{v_1}\bigr\}$ vanishes because it lies in the image of \eqref{FormulaMultiplicationMap}.   This finishes the proof of \ref{bykovskii:rel3x}, and the Theorem.
\end{proof}
\begin{proof}[Second proof of Theorem \ref{TheoremSteinbergMilnor}]
Lee--Szczarba~\cite{LS76} proved that the $\Q$-vector space $\St_d(\Q)$ is generated by (not necessarily integral) apartments $[v_1,\dots,v_d]$ for $v_1,\dots, v_d\in \Z^d$ in general position subject to the following relations:
\begin{enumerate}
\item \label{Steinberg1} $[v_{\sigma(1)},v_{\sigma(2)}, \dots,v_{\sigma(d)}]=(-1)^{\sigma}[v_1,v_2, \dots,v_d]$ for $\sigma \in \Sfr_d,$
\item \label{Steinberg2}  $[m v_1,v_2, \dots,v_d]=[v_1,v_2, \dots,v_d]$  for $m \in \Z\sm \{0\},$
\item \label{Steinberg3}  $\sum_{i=0}^{d}(-1)^i[v_0,\dots,\widehat{v_{i\,}},\dots, v_{d}]=0$  for any vectors $v_0,\dots,v_{d}\in \Z^d$ in general position.
\end{enumerate}
For an isogeny $p\colon \Tup'\lra  \Tup$ we have a norm map (since an isogeny defines a finite extension of function fields)
\[
\Nm_p\colon K_n^M\bigl(\overline{\Q}(\Tup')\bigr)\lra K_n^M\bigl(\overline{\Q}(\Tup)\bigr).
\]
It is easy to see that this map induces a map from $\Kbar_n\bigl(\Tup'\bigr)$ to $\Kbar_n\bigl(\Tup\bigr)$ which we also call the norm map and denote in the same way. For a matrix $A\in \Mup_{d}(\Z)$ with $\det(A)\neq0$ we define an isogeny $p_A \colon \Tup_d\lra \Tup_d$ by the formula
\[
p_A(x_1,\dots,x_d)=\left(\prod_{i=1}^d x_i^{a_{i1}},\dots,\prod_{i=1}^d x_i^{a_{id}} \right).
\] 
Let $\adj(A)$ be the adjugate matrix to $A$, defined by $\adj(A)=\det(A)A^{-1}$, and  $p_{\adj(A)}\colon \Tup_d \lra \Tup_d$ be the corresponding isogeny. Consider an apartment $[v_1,\dots,v_d]=[A e_1,\dots,Ae_d]$ for $A\in \Mup_d(\Z)$ with $\det(A)\neq 0$. We define
\[
\Kmap([v_1,\dots,v_d])\coloneqq \Nm_{p_{\adj(A)}}\{1-x_1,\dots,1-x_d\}\in \Kbar_d\bigl(\Tup_d\bigr).
\]
More explicitly, we have
\begin{equation} \label{FormulaL}
\Kmap([v_1,\dots,v_d])=
N^{-1}\sum_{\substack{y_1^{N}=x_1 \\ \hspace{\widthof{$y_1^{N}$}}\vdotsB\hspace{\widthof{$x_d$}} \\ y_d^{N}=x_d}}
\biggl\{1-\prod_{i=1}^d y_i^{(v_{1})_i}, \dots ,1-\prod_{i=1}^d y_i^{(v_d)_{i}}\biggr\},
\end{equation}
where $N=|\det(v_1,\dots,v_d)|.$

We claim that the map $\Kmap$ is well-defined and satisfies (\ref{FormulaImageofIntegralAppartment}). The latter follows immediately from (\ref{FormulaL}), since if $[v_1,\dots,v_d]$ is an integral apartment, we have $N=1$. So, it remains to prove that $\Kmap$ annihilates relations \ref{Steinberg1}, \ref{Steinberg2}, and \ref{Steinberg3}.  By the tower theorem for norms, for any matrices $A,B \in \Mup_d(\Z)$ with nonzero determinants, we have 
$\Nm_{\adj(B)}\,\Nm_{\adj(A)}=\Nm_{\adj(A)\adj(B)}=\Nm_{\adj(BA)}$, so
\begin{equation}\label{FormulaSTNorm}
\Kmap(B\cdot[v_1,\dots,v_d] )=\Nm_{p_{\adj(B)}}\bigl (\Kmap([v_1,\dots,v_d])\bigr).
\end{equation}
 Because of \eqref{FormulaSTNorm}, it is sufficient to prove relations \ref{Steinberg1}, \ref{Steinberg2} for $v_1=e_1,\dots,v_d=e_d$. After that, we prove \ref{Steinberg3}, reducing it to the case  $v_0=e_1+\dots+e_d$ using \eqref{FormulaSTNorm} and \ref{Steinberg2}.

To see that  $\Kmap$ annihilates \ref{Steinberg1}, notice that
\[
\Kmap([e_{\sigma(1)},\dots, e_{\sigma(d)}])=\{1-x_{\sigma(1)},\dots,1-x_{\sigma(d)}\}=(-1)^{\sigma}\{1-x_1,\dots,1-x_d\}=(-1)^{\sigma}\Kmap([e_{1},\dots, e_{d}]).
\]
Next, for $m>0$ we have
\begin{align*}
\Kmap([m e_{1}, e_2,\dots, e_{d}])&=\frac{1}{m}\sum_{\substack{y_1^{m}=x_1 \\ \hspace{\widthof{$y_1^{m}$}} \vdotsB \hspace{\widthof{$x_d$}} \\ y_d^{m}=x_d}}\{1-y_1^m,1-y_2,\dots,1-y_d\}\\
&=\frac{1}{m}\sum_{y_1^{m}=x_1}\{1-y_1^m,1-x_2,\dots,1-x_d\}\\
&=\{1-x_1,1-x_2,\dots,1-x_d\}.
\end{align*}
Since we work in the cokernel of the map (\ref{FormulaMultiplicationMap}), we have 
\[
\Kmap([-e_{1}, e_2,\dots, e_{d}])=\{1-x_1^{-1},1-x_2,\dots,1-x_d\}=\{1-x_1,1-x_2,\dots,1-x_d\}
\]
This proves  that  $\Kmap$ annihilates  \ref{Steinberg2}. To prove that  $\Kmap$ annihilates  \ref{Steinberg3}, notice that
\[
\sum_{i=0}^{d}(-1)^i\Kmap([e_0,\dots,\widehat{e_{i\,}},\dots, e_{d}])=
\Bigl\{\frac{1-x_1}{1-x_1\cdots x_d},\dots,\frac{1-x_d}{1-x_1\cdots x_d} \Bigr\}.
\]
In $\Kbar_d\bigl(\Tup_d\bigr)$, we have 
\[
\Bigl\{\frac{1-x_1}{1-x_1\dots x_d},\dots,\frac{1-x_d}{1-x_1\dots x_d} \Bigr\}=
\Bigl\{\frac{1-x_1}{1-x_1\cdots x_d},\frac{x_1(1-x_2)}{1-x_1\cdots x_d},\dots,\frac{x_1\cdots x_{d-1}(1-x_d)}{1-x_1\cdots x_d} \Bigr\}.
\] 
The latter element vanishes because $\{a_1,\dots,a_{n}\}=0$ for $a_1+\dots+a_n=1$, see \cite[Lemma 1.3]{Mil69}. This finishes the proof of the theorem.
\end{proof}

\subsection{Injectivity of the map \texorpdfstring{$\Kmap$}{K}} \label{InjectivitySteinbergMilnor}

Our next goal is to show that the map $\Kmap$ is injective. The proof is based on a Milnor $K$-theory version of Parshin's residue \cite{Par75,Par76}.

\begin{proposition} \label{Proposition Kbar injectivity}
For any $d\geq 0$ the map $\Kmap\colon \St_d(\Q)\lra  \Kbar_d\bigl(\Tup_d\bigr)$   is injective.
\end{proposition}
\begin{proof}
Recall that $\St_d(\Q)$ is defined as the reduced $(d-2)$-homology of the Tits building $\Tc_{V}$ for $V=\Q^d.$ The vector space of $(d-2)$-chains $C_{d-2}(\Tc_{V},\Q)$ is a $\Q$-vector space on the set of complete flags
\[
\Fc_{\bullet}=\Bigl(0=\Fc_0\subsetneq \Fc_1\subsetneq \Fc_2 \subsetneq \dots \subsetneq \Fc_{d-1}\subsetneq\Fc_d=V\Bigr).
\]
  The space $\Tc_{V}$ has dimension $d-2,$ so we have an embedding $i\colon \St_d(\Q)\hookrightarrow C_{d-2}(\Tc_{V})$.  Recall that by (\ref{FormulaFromSteinbergToChains}) we have
\[
i([v_1,v_2,\dots,v_d])=\sum_{\sigma \in \Sfr_d}(-1)^{\sigma}\Bigl (0\subsetneq \langle v_{\sigma(1)} \rangle \subsetneq \langle v_{\sigma(1)}, v_{\sigma(2)} \rangle \subsetneq \dots \subsetneq \langle v_{\sigma(1)}, v_{\sigma(2)},\dots, v_{\sigma(d-1)}\rangle \subsetneq  V \Bigr).
\]
To show that $\Kmap$ is injective it is sufficient to construct a map 
$j\colon \Kbar_d\bigl(\Tup_d\bigr)\lra C_{d-2}(\Tc_{V})$ such that $i=j\circ \Kmap.$ To construct this map, we need some preparation. 

For a subspace $W\subseteq V=\Xup(\Tup_d)_{\Q}$ we denote by $\Tup_W$ the torus corresponding to the sublattice $W\cap \Xup(\Tup_d)$ in $\Xup(\Tup_d)$. For $W_2\subseteq W_1 \subseteq V$ of dimensions $d_2<d_1\leq d$ such that with $d_2=d_1-1$ we have a map 
\[
\res_{W_1,W_2}\colon \Kbar_{d_1}\bigl(\Tup_{W_1}\bigr)\lra \Kbar_{d_2}\bigl(\Tup_{W_2}\bigr)
\] 
induced by the residue map for the discrete valuation associated to the codimension~$1$ subvariety $\Tup_{W_2}$ in $\Tup_{W_1}$ (more explicitly, $\Tup_{W_2}$ is the vanishing set of $\chi-1$ for some character on $\Tup_{W_1}$, and we choose the discrete valuation associated to the irreducible polynomial $\chi-1$).
For a  complete flag $\Fc_\bullet$ in $V$ denote the composition
\[
\Kbar_d\bigl(\Tup_d\bigr)=\Kbar_d\!\bigl(\Tup_{\Fc_d}\bigr)\xrightarrow{\res_{\Fc_d,\Fc_{d{\!-\!}1}}} \Kbar_{d{\!-\!}1}\!\bigl(\Tup_{\Fc_{d{\!-\!}1}}\bigr)\xrightarrow{\res_{\Fc_{d{\!-\!}1},\Fc_{d{\!-\!}2}}}\cdots \xrightarrow{\res_{\Fc_{1},\Fc_{0}}} \Kbar_0\!\bigl(\Tup_{\Fc_0}\bigr)\cong \Q
\]
by $\res_{\Fc}$.
Notice that for a nonzero function $f\in \overline{\Q}(\mathrm{T})$ there are only finitely many discrete valuations $\nu$ corresponding to subtori, for which $\nu(f)\neq 0$. It follows that for an element $a\in \Kbar_d\bigl(\Tup_d\bigr)$ the residue $\res_{\Fc}(a)$ is not equal to zero for only finitely many flags $\Fc$.
We define the map $j\colon \Kbar_d\bigl(\Tup_d\bigr)\lra C_{d-2}(\Tc_{V})$ by the formula
\[
j(a)=\sum_{\Fc}\res_{\Fc}(a) \Fc\in C_{d-2}(\Tc_V,\Q).
\]

To finish the proof, it is sufficient to check the formula $i=j\circ \Kmap$ on integral apartments. All our constructions are $\GL_d(\Z)$-equivariant, so it is sufficient to check the statement for the apartment $[e_1,\dots,e_d]$. This follows by induction from the fact that that the residue $\res_{V,W}[e_1,\dots,e_d]$
is nonzero only when $W=\langle e_1, \dots, \widehat{e_i}, \dots, e_d \rangle$ for some $i\in \{1,\dots,d\}$; for such W we have \[
\res_{V,W}\{1-x_1,\dots,1-x_d\}=(-1)^i\{1-x_1,\dots,\widehat{1-x_i},\dots, 1-x_d\}.
\]
It follows that $\res_\Fc(\{1-x_1,\dots,1-x_d\})$ is nonzero only for flags
\[
\Fc_{\sigma}=\Bigl (0\subsetneq \langle e_{\sigma(1)} \rangle \subsetneq \langle e_{\sigma(1)}, e_{\sigma(2)} \rangle \subsetneq \dots \subsetneq \langle e_{\sigma(1)}, e_{\sigma(2)},\dots, e_{\sigma(d-1)} \rangle \subsetneq V\Bigr)
\]
for some $\sigma\in \Sfr_d$ and that $\res_{\Fc_{\sigma}}(\{1-x_1,\dots,1-x_d\})=(-1)^{\sigma}.$
Thus
\[
j(\Kmap([e_1,\dots,e_d]))=\sum_{\sigma \in \Sfr_d}(-1)^{\sigma}\Bigl (0\subsetneq \langle e_{\sigma(1)} \rangle \subsetneq \langle e_{\sigma(1)}, e_{\sigma(2)} \rangle \subsetneq \dots \subsetneq \langle e_{\sigma(1)}, e_{\sigma(2)},\dots, e_{\sigma(d-1)} \rangle \subsetneq V\Bigr)
\]
and so $j(\Kmap([v_1,\dots,v_d]))=i([v_1,\dots,v_d])$ for any integral apartment $[v_1,\dots,v_d]\in \St_d(\Q)$. The statement follows.
\end{proof}

\begin{remark}
The map $j\colon \Kbar_d\bigl(\Tup_d\bigr)\lra C_{d-2}(\Tc_{V})$,
defined in the proof of Proposition \ref{Proposition Kbar injectivity}, is surjective. Since the image of $j$ is invariant under the group $\GL_d(\Z)$, it is sufficient to construct an element $a\in\Kbar_d\bigl(\Tup_d\bigr)$ such that there exists a unique flag $\Fc$ for which $\res_{\Fc}(a)\neq 0$. One can take 
\[
a=\{1-x_1,2-(x_1+x_2),\dots,d-(x_1+\dots+x_d)\}
\] 
as such an element.
\end{remark}

\subsection{A simple proof of Bykovski\u{\i} theorem}\label{SectionARByk}
In this section, we give a simple proof of the theorems of Ash--Rudolph and Bykovski\u{\i} which is inspired by the classical result of Milnor \cite[Theorem 2.1]{Mil69}. The proof is completely elementary and does not use the results of \S\ref{SteinbergAndMilnor} and \S\ref{InjectivitySteinbergMilnor}. We prove the result in the setup where $\St(V)$ is a $\Q$-vector space, but note that the proof can be trivially adapted to the case when $\St(V)$ is viewed as an abelian group.

\begin{theorem}[Bykovski\u{\i}, 2003, \cite{Byk}] \label{TheoremARByk} The Steinberg module $\St_d(\Q)$ is generated by integral apartments $[v_1,\dots,v_d]$. All relations between integral apartments arise from the following three:
\begin{align}
    \tag*{{\bf (R1)}} \label{bykovskii:rel1}[v_{\sigma(1)},v_{\sigma(2)},\dots,v_{\sigma(d)}]&=(-1)^{\sigma}[v_1,v_2,\dots,v_d] \text{ for }\sigma \in \Sfr_d, \\
    \tag*{{\bf (R2)}} \label{bykovskii:rel2} [v_1,v_2,\dots,v_d]&=[-v_1,v_2,\dots,v_d], \\
    \tag*{{\bf (R3)}} \label{bykovskii:rel3} 
    [v_1,v_2,v_3,\dots,v_d]&=[v_1,v_1+v_2,v_3,\dots,v_d]+[v_1+v_2,v_2,v_3,\dots,v_d].
\end{align}

\begin{proof}
\newcommand{\wrapper}[1]{(#1)}
\renewcommand{\widetilde}{\reallywidetildelow}
\renewcommand{\tilde}{\reallywidetilde}
Let $\Lambda$ be a lattice of rank $d$ and put $V=\Lambda_{\Q}$. Consider a group $\StB(\Lambda)$ generated by symbols $[v_1,\dots,v_d]$ for bases $v_1,\dots,v_d$ of $\Lambda$ subject to relations \ref{bykovskii:rel1}--\ref{bykovskii:rel3}. Since these relations hold in $\St(V)$, we have a well-defined map 
\[
\Bup \colon \StB(\Lambda)\lra \St(V).
\]
We need to show that $\Bup$ is an isomorphism. We argue by induction on $d;$ the base case $d=1$ is trivial. 

For a line $P$ in $V$ consider a \emph{residue map} 
\[
\partial_P\colon \St(V)\lra \St(V/P)
\]
sending $[v_1,\dots,v_d]$ to zero if none of the vectors $v_i$ lies in $P$ and to $(-1)^{i-1}[\overline{v}_1,\dots,\widehat{\overline{v}}_i,\dots,\overline{v}_d]$ if $v_i\in P.$ It is easy to see that this map is well-defined. The same rule also gives  residue map $\delta_P\colon \StB(\Lambda)\lra \StB(\Lambda/(\Lambda\cap P))$. Denote by ${\Bup}_P$  the natural map sending $\StB(\Lambda/(\Lambda\cap P))$ to $\St(V/P)$; we have ${\Bup}_P \circ \delta_P=\partial_P \circ \Bup$.

Choose an arbitrary surjective linear function $h\colon \Lambda\lra \Z$ and denote by $H\subseteq V$ the hyperplane spanned by the lattice $\Ker(h)\subset \Lambda.$ We have the following commutative diagram:
\[
\begin{tikzcd}[row sep=large]
\StB(\Lambda) \arrow{r}{\Bup} \arrow{d}{\oplus \delta_{P}} & \St(V) \arrow[d, "\oplus \partial_{P}"] \\
\displaystyle\bigoplus_{P\not\subseteq H}\StB(\Lambda/(\Lambda\cap P)) \arrow{r}{\oplus \Bup_P} & \displaystyle\bigoplus_{P\not\subseteq H}\St(V/P)
\end{tikzcd}
\]

The map $\oplus \partial_{P} \colon \St(V)\lra \bigoplus_{P\not\subseteq H}\St(V/P)$ is an isomorphism by Proposition \ref{PropositionSteinbergSubspace}.  By the induction hypothesis, the map $\oplus \Bup_P$ is also an isomorphism. It remains to show that the map 
\[
\oplus \delta_{P} \colon \StB(\Lambda)\lra \bigoplus_{P\not\subseteq H}\StB(\Lambda/(\Lambda\cap P))
\]
is an isomorphism.

Consider an increasing filtration  $\Ic_\bullet\StB(\Lambda)$ such that $\Ic_k\StB(\Lambda)$ is generated by symbols $[v_1,\dots,v_d]$ such that $|h(v_i)|\leq k$.
For every line $P$ choose a vector $v_P\in \Lambda$ generating the lattice $\Lambda\cap P$.  For a line $P$ with $|h(v_P)|\geq k$  the residue maps $\delta_P$  vanishes on $\Ic_{k-1}\StB(\Lambda),$ so for any $k\geq 1$ we have an induced map
\begin{equation}\label{FormulaResidueOnGr}
\oplus \delta_{P} \colon\gr^\Ic_k\StB(\Lambda) \lra \!\!\! \bigoplus_{|h(v_P)|=k} \!\!\! \StB(\Lambda/(\Lambda\cap P)).
\end{equation}
Since $\Ic_0\StB(\Lambda)=0$, it is sufficient to show that (\ref{FormulaResidueOnGr}) is an isomorphism for any $k\geq 1$.

For a  line $P$ with $|h(v_P)|=k$ consider a \emph{coresidue} map $c_P\colon \StB(\Lambda/(\Lambda\cap P))\lra \gr^\Ic_k\StB(\Lambda)$  sending an element $[u_1,\dots,u_{d-1}]$ to $[v_P,\tilde{u}_1,\dots,\tilde{u}_{d-1}],$ where $\tilde{u} \in V$ is the unique lift of $u\in V/P$ such that $0<h(\tilde{u})<k.$ We need to check that this map is well-defined, i.e., satisfies the properties \ref{bykovskii:rel1}--\ref{bykovskii:rel3}. It is enough to check \ref{bykovskii:rel3} for $d=3$ and  \ref{bykovskii:rel2} for $d=2;$ \ref{bykovskii:rel1} is obvious.

For \ref{bykovskii:rel2}, we have
\begin{align*}
&c_P(u)-c_P(-u)=[v_P,\tilde{u}]-[v_P,\widetilde{\wrapper{-u}}]=[v_P,\tilde{u}]-[v_P,v_P-\tilde{u}]=[v_P-\tilde{u},\tilde{u}]\in \Ic_{k-1}\StB(\Lambda).
\end{align*}
Next, for \ref{bykovskii:rel3} we need to show that
\begin{align*}
&c_P([u_1,u_2])-c_P([u_1,u_1+u_2])-c_P([u_1+u_2,u_2])\\
={} &[v_P,\tilde{u}_1,\tilde{u}_2]-[v_P,\tilde{u}_1,\widetilde{\wrapper{u_1\!+\!u_2}}]-[v_P,\widetilde{\wrapper{u_1\!+\!u_2}},\tilde{u}_2] \, \in \, \Ic_{k-1}\StB(\Lambda).
\end{align*}
If $h(\tilde{u}_1)+h(\tilde{u}_2)<k,$ we have $\widetilde{\wrapper{u_1\!+\!u_2}}=\tilde{u}_1+\tilde{u}_2$ and the statement follows. Otherwise, $\widetilde{\wrapper{u_1\!+\!u_2}}=\tilde{u}_1+\tilde{u}_2-v_P$ and we have
{\allowdisplaybreaks
\begin{align*}
&[v_P,\tilde{u}_1,\tilde{u}_2]-[v_P,\tilde{u}_1,\widetilde{\wrapper{u_1\!+\!u_2}}]-[v_P,\widetilde{\wrapper{u_1\!+\!u_2}},\tilde{u}_2]\\[1ex]
={}
&([v_P,\tilde{u}_1,\tilde{u}_2]-[v_P,\tilde{u}_1,\tilde{u}_1+\tilde{u}_2]-[v_P,\tilde{u}_1+\tilde{u}_2,\tilde{u}_2,])\\*[-0.5ex]
&-([v_P,\tilde{u}_1,\widetilde{\wrapper{u_1\!+\!u_2}}]-[v_P,\tilde{u}_1,\tilde{u}_1+\tilde{u}_2])-([v_P,\widetilde{\wrapper{u_1\!+\!u_2}},\tilde{u}_2]-[v_P,\tilde{u}_1+\tilde{u}_2,\tilde{u}_2])\\[0.5ex]
={}&-([v_P,\tilde{u}_1,\tilde{u}_1+\tilde{u}_2-v_P]-[v_P,\tilde{u}_1,\tilde{u}_1+\tilde{u}_2])-([v_P,\tilde{u}_1+\tilde{u}_2-v_P,\tilde{u}_2]-[v_P,\tilde{u}_1+\tilde{u}_2,\tilde{u}_2])\\
={}&[\tilde{u}_1,\tilde{u}_1+\tilde{u}_2,\tilde{u}_1+\tilde{u}_2-v_P]-[\tilde{u}_2,\tilde{u}_1+\tilde{u}_2,\tilde{u}_1+\tilde{u}_2-v_P]\\
={}&[\tilde{u}_1,\tilde{u}_2,\tilde{u}_1+\tilde{u}_2-v_P] \, \in \, \Ic_{k-1}\StB(\Lambda).
\end{align*}}

Clearly, $\delta_P\circ c_P=\Id$ so it only remains to show that the map $\sum c_P$ is surjective. For that, consider a filtration $\Gc_\bullet$ on $\gr^\Ic_k\StB(\Lambda)$ where $\Gc_s \gr^\Ic_k\StB(\Lambda)$ is spanned by elements $[v_1,\dots,v_d]$ with at most $s$ vectors satisfying $|h(v_i)|=k.$ Since $\Gc_1$ coincides with image of $\sum c_P$ it is sufficient to show that 
$
\gr^\Gc_s\gr^\Ic_k\StB=0
$
for $s\geq 2.$ This follows from the fact that if $h(v_1)=h(v_2)=k$ then $[v_1,v_2]=[v_1,-v_2]=[v_1,v_1-v_2]-[v_1-v_2,-v_2]\in \Gc_1.$
\end{proof}

\end{theorem}

\section{Polylogarithms on a torus}
\subsection{Formal multiple polylogarithms}\label{SectionMultiplePolylogs}
 
For a field $F$, the Hopf algebra of formal multiple polylogarithms $\Hc^{\mathrm{f}}(F)$ was defined in \cite{CMRR24}.  Henceforth we will omit the superscript ${}^{\fup}$ for notational simplicity, as we only work with the Hopf algebra of formal multiple polylogarithms.  For a finite field we declare \( \Hc_0(F) = \Q \), \( \Hc_{n}(F) = 0 \), \( n > 0 \).  When \( F \) is an infinite field, the construction is as follows.

Firstly, define a positively graded \( \mathbb{Q} \)-vector space \( \mathcal{A}(F) = \bigoplus_{n\geq1}\mathcal{A}_n(F) \) generated by symbols \( \llp x_0, x_1, \ldots, x_n \rrp \in \mathcal{A}_n(F) \), \( x_i \in F \), subject to the following relations \cite[\S2.1]{CMRR24}:
\begin{align}
    \tag*{$\bf (A_1)$}\label{Acycle} & \llp x_0,x_1,\dots,  x_n \rrp = \llp x_1,x_2,\dots,x_n,x_0 \rrp \,,\\
    \tag*{$\bf (A_2)$}\label{Atranslate} & \llp x_0+b,x_1+b,\dots,  x_n+b \rrp = \llp x_0,x_1,\dots,x_n \rrp \text{ for $b\in F$,}\\
    \tag*{$\bf (A_3)$}\label{Alog} & \llp 0,x \rrp + \llp 0,y \rrp = \llp 0,xy \rrp\ \text{for $ x,y \in F^{\times}$,}\\
    \tag*{$\bf (A_4)$}\label{Ascale} & \llp mx_0,mx_1,\dots, mx_n\rrp = \llp x_0,x_1,\dots,x_n\rrp \ \text{for $m\in F^{\times}$ and $n\geq 2$,}\\
    \tag*{$\bf (A_5)$}\label{Azero} & \llp 0,0,\dots,0 \rrp =0\,,\\
    \tag*{$\bf (A_6)$}\label{Aone} & \llp 1,0,\dots,0\rrp =0\,.
\end{align}
The cobracket \( \delta \colon \mathcal{A}(F) \to \Lambda^2 \mathcal{A}(F) \) is defined by the following explicit formula, 
\begin{equation}\label{CobracketA}
\delta \llp x_0,\dots, x_n\rrp =\sum_{j=0}^n\sum_{i=1}^{n-1}\llp x_{j}, x_{j+1}, \dots, x_{j+i}\rrp \wedge \llp x_{j},  x_{j+i+1}, \dots, x_{j+n}\rrp \,,
\end{equation}
 where we use the convention that $x_{i+n+1} = x_{i}$ for all $i$.  This makes \( \mathcal{A}(F) \) into a Lie coalgebra.  There  is a specialization map \( \Sp_{t\to t_0} \mathcal{A}_n(F(t)) \to \mathcal{A}_n(F) \) \cite[\S2.2]{CMRR24}; this specialization is the counterpart of a limit for continuous functions.  

One then defines  \cite[\S2.3]{CMRR24} a space of relations \( \mathcal{R}_n(F) \subset \mathcal{A}_n(F) \) and the Lie coalgebra of formal multiple polylogarithms \( \Lc_n(F) = \mathcal{A}_n(F) / \mathcal{R}_n(F) \) inductively: consider  \( R \in \mathcal{A}_n(F(t)) \)  such that \( \delta R \) vanishes in \( \Lambda^2 \Lc(F(t)) \), where \( \Lc_k(F') \) is already defined inductively for \( k < n \) and all fields \( F' \).  Then \( R \) should be constant\footnote{The element \( R \) defines a class in the first Chevalley-Eilenberg cohomology of \( \Lc(F(t)) \), which is conjectured in \cite{CMRR24} to be isomorphic to \( \gr_\gamma^n K_{2n-1}(F(t))_\mathbb{Q} \), and is then conjecturally isomorphic to \( \gr_\gamma^n K_{2n-1}(F)_\mathbb{Q} \) by the Beilinson-Soul\'e vanishing conjecture, for \( n \geq 2 \), using the inclusion \( F \hookrightarrow F(t) \).}, and the space of relations \( \mathcal{R}_n(F) \) is given as the span of \( \Sp_{t\to0}(R) - \Sp_{t\to1}(R) \) for such \( R \).  The projection of \( \llp x_0, x_1,\ldots,x_n \rrp \) to \( \Lc_n(F) \) is denoted by \( \Cor(x_0, x_1,\ldots,x_n) \) and is called the \emph{correlator}.   As a $\Q$-vector space, $\Lc(F)$ is spanned by correlators, by construction.   After projecting \eqref{CobracketA} to \( \Lc(F) \), the cobracket of correlators is given by
\begin{equation}\label{FormulaCoproductCorrelators}
\delta \Cor(x_0,\dots, x_n)=\sum_{j=0}^n\sum_{i=1}^{n-1} \Cor(x_{j}, x_{j+1}, \dots, x_{j+i})\wedge  \Cor(x_{j},  x_{j+i+1}, \dots, x_{j+n}).
\end{equation}
(The relations \ref{Acycle}--\ref{Aone} and cobracket \eqref{CobracketA} on \( \mathcal{A}(F) \) are motivated by such formulas for the \emph{motivic correlators} \cite{Gon19}, \cite[Eq. (66)]{GR18}.)

The Hopf algebra of formal multiple polylogarithms \( \Hc(F) \) is then abstractly defined as the universal coenveloping coalgebra of \( \Lc(F) \).  This means that $\Lc(F) = Q(\Hc(F))$ is the Lie coalgebra of indecomposable elements of $\Hc(F)$.  The projection of an element $a\in \Hc(F)$ to  $\Lc(F)$ is denoted by $a^{\Lc}$. Then \( \Hc(F) \) is a commutative graded connected Hopf algebra over $\Q$,  and it has a distinguished set of generators
    \[\I(x_0;x_1,\dots, x_{n};x_{n+1})\in \Hc_n(F)\]
for $x_0,\dots,x_{n+1}\in F$, called \emph{iterated integrals}, such that 
\begin{equation}\label{FormulaIteratedIntegralsCorrelators}
\I^{\Lc}(x_0;x_1,\dots,x_n;x_{n+1}) = \Cor(x_1,\dots,x_{n+1})- \Cor(x_0,\dots,x_n).
\end{equation}

We have $\I(x_0;x_1)=1\in \Hc_0(F).$ In weight 1, there is an isomorphism $\Hc_1(F)\cong F^{\times}_\Q$ and we use the notation $\log(x)$ for an element in $\Hc_1(F)$ corresponding to $x\in F^{\times}$; it holds that
\begin{equation}\label{IasLog}
\I(x_0;x_1;x_2)=\log(x_1-x_2)-\log(x_1-x_0) \,.
\end{equation}
Notice that since we work with $\Q$-vector spaces, $\log(\zeta)=0$ for any root of unity $\zeta$; note in particular that this means \( \log(\zeta x) = \log(x) \) for any root of unity \( \zeta \). The coproduct of iterated integrals is given by the following formula:
    \begin{equation} \label{Icoproduct}
    \begin{split}
    &\Delta(\I(x_0;x_1,\dots,x_n;x_{n+1}))=\\
    &\sum_{\substack{0=i_0<i_1<\cdots \\ \cdots<i_r<i_{r+1}=n+1}}\I(x_{i_0};x_{i_1},\dots,x_{i_r};x_{i_{r+1}})\otimes \prod_{p=0}^{r}\I(x_{i_p};x_{i_{p}+1},\dots,x_{i_{p+1}-1};x_{i_{p+1}}).
    \end{split}
    \end{equation}
    
Multiple polylogarithms are related to iterated integrals by a simple change of variable. For $n_1,\dots,n_k \in \N$ and $a_1,\dots,a_k\in F^\times$ we define   multiple polylogarithm of depth $k:$ 
    \begin{equation}\label{LinDef}
    \begin{split}
    &\Li_{n_1,\dots,n_k}(a_1,a_2,\dots,a_k)\\
    &=(-1)^{k}\I(0;1,\underbrace{0,\dots,0,a_1}_{n_1},\dots,\underbrace{0,\dots,0,a_1a_2\cdots a_{k-1}}_{n_{k-1}},\underbrace{0,\dots,0;a_1a_2\cdots a_{k}}_{n_k}).
    \end{split}
    \end{equation}
In particular, we have $\Li_{1}(a)=-\log(1-a)\in \Hc_1(F).$ 

To summarize, we have two families of generators of $\Hc(F)$: iterated integrals and multiple polylogarithms. We also have a family of generators of $\Lc(F)$ called correlators. At first, it seems that there is not much difference between these families as they are all related to each other in a simple way. Yet, it appears to be fruitful to consider all three families simultaneously. These families of generators are related to the three families of generators of the stable Steinberg module $\StL(F)$, see \S \ref{SectionSteinbergCorrelators}, and, in particular, Remark \ref{RemarkGoncharovDihedral}.

Recall some properties of iterated integrals discussed in \cite[\S3.3]{CMRR24}. First, we have the \emph{shuffle relation} \cite[Lemma 24 (P2)]{CMRR24}:
\begin{equation}\label{ShuffleRelation}
\I(a;x_1,\dots,x_n;b) \cdot \I(a;x_{n+1},\dots,x_{n+m};b)=\sum_{\sigma \in \Sigma_{n,m}} \I(a;x_{\sigma(1)},\dots,x_{\sigma(n+m)};b),
\end{equation}
where $\Sigma_{n_1,n_2}\subseteq \Sfr_{n_1+n_2}$ is the set of $(n_1,n_2)$-shuffles, i.e., the set of permutations $\sigma \in \Sfr_{n_1+n_2}$ such that $\sigma^{-1}(1)<\dots<\sigma^{-1}(n_1)$ and $\sigma^{-1}(n_1+1)<\dots<\sigma^{-1}(n_1+n_2)$. 
Next, we have the \emph{path composition relation} \cite[Lemma 24 (P3)]{CMRR24}
\begin{equation}\label{PathComposition}
\I(x_0;x_1,\dots,x_n;x_{n+1})=\sum_{k=0}^n\I(x_0;x_1,\dots,x_k;a)\I(a;x_{k+1},\dots,x_n;x_{n+1}).
\end{equation}
We also have the \emph{reversal relation}

\begin{lemma}
For any \( n \geq 1 \), and any \( x_0, \ldots, x_{n+1} \in F \), we have
\begin{equation}\label{PathReversal}
    \I(x_0;x_1,\dots,x_n;x_{n+1}) =
    (-1)^n\I(x_{n+1};x_n,\dots,x_1;x_{0}).
\end{equation}
\begin{proof}
The base case follows from \eqref{IasLog}.  By indiction one readily sees
\[
    \Delta' ( \I(x_0;x_1,\dots,x_n;x_{n+1}) - 
    (-1)^n\I(x_{n+1};x_n,\dots,x_1;x_{0}) ) = 0 .
\] 
Thus, by Lemma \ref{LemmaMilnorMoore}, it suffices to check the identity in the Lie coalgebra $\Lc_n(F)$.  This was done in \cite[Proposition 22]{CMRR24}.
\end{proof}
\end{lemma}

Another important property of iterated integrals is invariance under rescaling:  
\begin{lemma} \label{LemmaAffineInvariance}
For any $x_0,\dots,x_{n+1}\in F$ with $x_0\ne x_1$, $x_{n}\ne x_{n+1}$, and any $\lambda\in F^{\times}$ we have
    \[\I(x_0;x_1,\dots,x_n;x_{n+1}) =
      \I(\lambda x_0;\lambda x_1,\dots,\lambda x_n;\lambda x_{n+1}).\]
The same identity holds if $\lambda$ is a root of unity, without any restrictions on $x_0,\dots,x_{n+1}$.
\end{lemma}
\begin{proof}
We prove this by induction on $n$. For $n=1$ the identity follows from 
    \[\I(x_0;x_1;x_2)-\I(\lambda x_0;\lambda x_1;\lambda x_2) = 
    \log(x_1-x_2)-\log(\lambda x_1-\lambda x_2)-\log(x_0-x_1)+\log(\lambda x_0-\lambda x_1) = 0\]
provided that $x_1\ne x_0$ and $x_1\ne x_2$ or provided that $\lambda$ is a root of unity. Assuming the identity is known in weights $<n$, it is easy to see that the difference $R$ of the left hand side and the right hand side satisfies $\Delta' R=0$. Thus, by Lemma \ref{LemmaMilnorMoore}, it suffices to check the identity in the Lie coalgebra $\Lc_n(F)$. The latter follows from \cite[Proposition 20]{CMRR24}.
\end{proof}

Multiple polylogarithms also satisfy distribution relations, which we now recall. 

\begin{lemma} \label{LemmaDistributionLi}
Let $F$ contain a primitive $m$-th root of unity.
Then for any $x_1,\dots,x_{k}\in F$ we have
    \[\sum_{\zeta_1^m=\dots=\zeta_k^m=1}\Li_{n_1,\dots,n_k}(\zeta_1 x_1,\dots,\zeta_k x_k) =
    m^{k-n}\Li_{n_1,\dots,n_k}(x_1^m,\dots,x_k^m),\]
where $n=n_1+\dots+n_k$.
\end{lemma}
\begin{proof}
In \cite[Lemma~26]{CMRR24} it is proven that for any $z_0,\dots,z_{n+1}\in F$
    \begin{equation} \label{EquationDistributionIteratedIntegrals}
    \I(z_0^m;z_1^m,\dots,z_n^m;z_{n+1}^m)=\sum_{\zeta_1^m=1,\dots,\zeta_{n}^m=1}\I(z_0;\zeta_{1}z_1,\dots,\zeta_{n}z_n;z_{n+1})\,.
    \end{equation}
Letting 
    \[(z_0,z_1,\dots,z_n,z_{n+1})=(0,1,\underbrace{0,\dots,0,x_1}_{n_1},\dots,\underbrace{0,\dots,0,x_1x_2\dots x_{k-1}}_{n_{k-1}},\underbrace{0,\dots,0,x_1x_2\dots x_{k}}_{n_k})\]
and using~\eqref{LinDef}, together with the second claim from Lemma~\ref{LemmaAffineInvariance}, we get from~\eqref{EquationDistributionIteratedIntegrals} that
    \[\Li_{n_1,\dots,n_k}(x_1^m,\dots,x_k^m) = m^{(n_1-1)+\dots+(n_k-1)}\sum_{\zeta_1^m=\dots=\zeta_k^m=1}\Li_{n_1,\dots,n_k}(\zeta_1 x_1,\dots,\zeta_k x_k).\]
This proves the claim, since $(n_1-1)+\dots+(n_k-1)=n-k$.
\end{proof}

The analog of Lemma \ref{LemmaDistributionLi}  for $m=-1$ is known as the inversion relation and is more involved, see \cite[Lemma 38]{CMRR24}. We will need only the statement for polylogarithms of depth one:
\begin{equation}\label{LemmaInversionDepthOne}
 \Li_n(x)+(-1)^n\Li_n(1/x) = {-}\frac{\log(x)^n}{n!} = {-}\I(0; 0,\dots,0; x).
\end{equation}
The proof is similar to that of Lemma \ref{LemmaAffineInvariance}.

\subsection{The Hopf algebra of polylogarithms on a torus}\label{SectionHopfAlgebraPolylogarithmsOnTorus}
The goal of the section is to define the Hopf algebra $\Hc(\Tup)$ of polylogarithms on a torus $\Tup$.  Informally, elements of  $\Hc(\Tup)$ are products of polylogarithms
\[
\Li_{n_1,\dots,n_k}\left(\zeta_1\prod_{i=1}^d (\!\sqrt[N]{x_i}\,)^{a_{i1}},\dots,\zeta_k\prod_{i=1}^d (\!\sqrt[N]{x_i}\,)^{a_{ik}}\right) 
\]
where $\zeta_i$ are roots of unity and the exponent vectors $(a_{1j},\dots,a_{d\indsp  j})$ are linearly independent.
The Hopf algebra $\Hc(\Tup)$ is defined as a colimit of more explicit Hopf algebras $\Hc_{\Tup'}$ over all tori $\Tup'$ covering $\Tup$.

Let  
\[
\Tup_d=\Spec\left(\overline{\Q}[x_1^{\vphantom{-1}},x_1^{-1},\dots,x_d^{\vphantom{-1}},x_d^{-1}]\right)
\]
be an algebraic torus of rank $d;$ we put  $\Tup_0=\Spec(\overline{\Q})$.  The character lattice of $\Tup$ is denoted $\Xup(\Tup)$.

\begin{definition} \label{definitionHT} The space $\Hc_{\Tup_d}$ is the $\Q$-vector subspace of $\Hc(\overline{\Q}(\Tup_d))$ (see \S\ref{SectionMultiplePolylogs}) spanned by products of multiple polylogarithms
\begin{equation} \label{FormulaHodgePolylogarithms}
\Li_{n_1,\dots,n_k}\left(\zeta_1\prod_{i=1}^d x_i^{a_{i1}},\dots,\zeta_k\prod_{i=1}^d x_i^{a_{ik}}\right)\in \Hc_n(\overline{\Q}(\Tup_d))
\end{equation}
such that $(a_{ij})$ is a $d\times k$ integral matrix of rank $k$ and $\zeta_i$ are arbitrary roots of unity. In particular, $\Hc_{\Tup_0}\cong \Q$.
\end{definition}

\begin{lemma} \label{LemmaGeneratorsOfHT}
The space $\Hc_{\Tup_d}$ is generated by products of iterated integrals $\I(0;z_1,\dots,z_n;z_{n+1})$, $n \geq 1$, where each $z_i$ is either $0$ or is of the form $\zeta\prod_{j=1}^{d}x_j^{a_j}$, for \( \zeta \) a root of unity, and the exponent vectors of the nonzero $z_i$, $1\leq i\leq n+1$, are affinely independent.
\end{lemma}
\begin{proof}
By the definition of $\Li_{n_1,\dots,n_k}$, the space spanned by the iterated integrals of the form given in the statement of lemma clearly contains $\Hc_{\Tup_d}$, so we only need to prove the reverse inclusion. Let $\I(0;z_1,\dots,z_n;z_{n+1})$ be any element satisfying the conditions of the lemma. We may assume that $z_{n+1}\ne0$ since otherwise the iterated integral vanishes. Let $1\leq j\leq n+1$ be the smallest index for which $z_j\ne 0$. If $j=1$, note that $z_n\ne z_{n+1}$, since otherwise $z_n,z_{n+1}\ne0$ and the exponent vectors of $z_n$ and $z_{n+1}$ would be affinely dependent, contradicting the assumption. Therefore, by Lemma~\ref{LemmaAffineInvariance}, $\I(0;z_1,\dots,z_n;z_{n+1})=I(0;z_1/z_{n+1},\dots,z_n/z_{n+1};1)$ and since $z_1\ne0$ the latter element lies in $\Hc_{\Tup_d}$ by the definition of $\Li_{n_1,\dots,n_k}$. Next, for $j\leq n$ we can prove the claim by induction in $j$ in the same way as for \cite[Corollary 30]{CMRR24}: using the shuffle product \eqref{ShuffleRelation} against \( I(0; 0, \ldots, 0; z_{n+1}) \) to reduce the number of leading zeros in the integral.  For the conclusion, note that any subset of an affinely independent set of vectors is affinely independent. Finally, for $j=n+1$ the claim follows directly from~\eqref{LemmaInversionDepthOne}.
\end{proof}

\begin{corollary}
The space $\Hc_{\Tup_d}$ is a Hopf subalgebra of $\Hc(\overline{\Q}(\Tup_d)).$ 
\end{corollary}
\begin{proof}
Recall: a graded connected bialgebra is automatically a Hopf algebra, as the coproduct uniquely determines the antipode.  Therefore, a subspace of a graded connected commutative Hopf algebra is a Hopf subalgebra if and only if it is a subalgebra and a subcoalgebra.  By Definition \ref{definitionHT}, $\Hc_{\Tup_d}$ is a subalgebra of $\Hc(\overline{\Q}(\Tup_d)),$ so it is sufficient to check that  $\Hc_{\Tup_d}$ is a subcoalgebra. 

First, we check that for any element $\I(0;z_1,\dots,z_n;z_{n+1})$ satisfying the conditions of Lemma~\ref{LemmaGeneratorsOfHT}, we have 
\begin{equation}\label{FormulaToProve}
\Delta\Bigl( \I(0;z_1,\dots,z_n;z_{n+1}) \Bigr)\in \Hc_{\Tup_d}\otimes  \Hc_{\Tup_d}.
\end{equation}
From the formula for the coproduct of iterated integrals stated in \S \ref{SectionMultiplePolylogs} we see that
it suffices to show that for any $0<i_0<\dots<i_{k+1}\leq n+1$ both $\I(0;z_{i_0},\dots,z_{i_k};z_{i_{k+1}})$ and $\I(z_{i_0};z_{i_1},\dots,z_{i_k};z_{i_{k+1}})$ belong to $\Hc_{\Tup_d}$. For $\I(0;z_{i_1},\dots,z_{i_k};z_{i_{k+1}})$ this is obvious since passing to subsets preserves affine independence. For $\I(z_{i_0};z_{i_1},\dots,z_{i_k};z_{i_{k+1}})$ this follows from (\ref{PathComposition}) and (\ref{PathReversal}) by splitting the path of integration at 0:
\begin{align*}
\I(z_{i_0};z_{i_1},\dots,z_{i_k};z_{i_{k+1}})&
=\sum_{j=0}^k\I(z_{i_0};z_{i_1},\dots,z_{i_j};0)\I(0;z_{i_{j+1}},\dots,z_{i_{k}};z_{i_{k+1}})\\
&=\sum_{j=0}^k(-1)^{j}\I(0;z_{i_j},\dots,z_{i_1};z_{i_0})\I(0;z_{i_{j+1}},\dots,z_{i_{k}};z_{i_{k+1}}).
\end{align*}

By  Lemma~\ref{LemmaGeneratorsOfHT}, the elements  $\I(0;z_1,\dots,z_n;z_{n+1})$ generate $\Hc_{\Tup_d}$ as an algebra, so (\ref{FormulaToProve}) implies  that 
$\Delta \bigl(\Hc_{\Tup_d}\bigr)\subseteq  \Hc_{\Tup_d} \otimes  \Hc_{\Tup_d}$. So, $\Hc_{\Tup_d}$ is a  subcoalgebra of $\Hc(\overline{\Q}(\Tup_d))$ and thus it is a Hopf subalgebra as well.  
\end{proof}

Consider a category  whose objects are algebraic tori over $\overline{\Q}$ and morphisms are isogenies; let $\Cc_{\Tup}$ be the corresponding category of objects over $\Tup$. An object of $\Cc_{\Tup}$ is an isogeny $p\colon\Tup'\lra \Tup$. For every isogeny $p\colon\Tup'\lra \Tup$, we have a well-defined map
\[
p^*\colon \Hc(\overline{\Q}(\Tup))\lra \Hc(\overline{\Q}(\Tup'))
\]
associated to the field extension $\overline{\Q}(\Tup)\subseteq  \overline{\Q}(\Tup')$.

\begin{lemma} \label{LemmaIsogeniesFunctorial} We have $p^*(\Hc_\Tup)\subseteq \Hc_{\Tup'}$.
\end{lemma}
\begin{proof}
Since $p$ is an isogeny, the map $p^*\colon \Xup(\Tup_d)\lra \Xup(\Tup_d')$ embeds the lattice $\Xup(\Tup_d)=\langle x_1,\dots,x_d\rangle $ into $\Xup(\Tup_d')=\langle y_1,\dots,y_d\rangle$ as a sublattice of full rank. Thus 
\begin{align*}
& p^*\left(\Li_{n_1,\dots,n_k}{\left(\zeta_1\prod_{i=1}^d x_i^{a_{i1}},\dots,\zeta_k\prod_{i=1}^d x_i^{a_{ik}}\right)}\right)\\
&=\Li_{n_1,\dots,n_k}\left(\zeta_1\prod_{i=1}^d p^*(x_i)^{a_{i1}},\dots,\zeta_k\prod_{i=1}^d p^*(x_i)^{a_{ik}}\right)=\Li_{n_1,\dots,n_k}\left(\zeta_1\prod_{i=1}^d y^{b_{i1}},\dots,\zeta_k\prod_{i=1}^d y_i^{b_{ik}}\right)
\end{align*}
for some matrix $(b_{ij})$ of rank $k$.
\end{proof}

By Lemma \ref{LemmaIsogeniesFunctorial}, we obtain a contravariant functor from $\Cc_{\Tup}$ to the category of commutative graded connected Hopf algebras which sends a torus $\Tup'$ to the Hopf algebra $\Hc_{\Tup'}.$ The category  $\Cc_{\Tup}$ is cofiltered, so the category $\Cc_{\Tup}^{\op}$ is filtered.

\begin{definition} Let $\Tup$ be an algebraic torus. The Hopf algebra $\Hc(\Tup)$ of polylogarithms on $\Tup$ is the filtered colimit
\[
\Hc(\Tup)=\varinjlim\Hc_{\Tup'}
\]
of Hopf algebras $\Hc_{\Tup'}$ over the category $\Cc_{\Tup}^{\op}$.
\end{definition}

\begin{remark} For an isogeny $p\colon\Tup'\longrightarrow \Tup$ we have a canonical isomorphism $\Hc(\Tup)\stackrel{\sim}{\longrightarrow} \Hc(\Tup')$ obtained from the final functor, sending tori over $\Tup'$ to tori over $\Tup.$ 
\end{remark}

\subsection{Primitive elements of the Hopf algebra of polylogarithms on a torus} \label{SectionPrimitiveElements}
Recall that we denote by $P(H)$ the space of primitive elements of a Hopf algebra $H$. Recall also that for a graded connected commutative Hopf algebra, we always have $H_1\subseteq P(H)$. Our goal is to show that $P(\Hc(\Tup))=\Hc_1(\Tup)$ for any torus $\Tup$.

\begin{lemma} \label{LemmaCohomologyOfHT} 
We have $P(\Hc_{\Tup_d})=\Hc_{\Tup_d,1}.$
\end{lemma}
\begin{proof}
It is sufficient to show that $\bigl[ P(\Hc_{\Tup_d})\bigr]_n=0$ for $n\geq 2.$ From now on we fix $n\geq 2.$
Denote by $\Lc_{\Tup_d}$ the Lie coalgebra of indecomposables of  $\Hc_{\Tup_d}.$ 
By Lemma \ref{LemmaMilnorMoore}, it is sufficient to show that $\bigl[H^1(\Lc_{\Tup_d},\Q)\bigr]_n=0$.  By \cite[Corollary 11]{CMRR24}, the map $i\colon \Lc(\overline{\Q})\lra \Lc(\overline{\Q}(\Tup_d))$ induces an isomorphism 
$i_{*}\colon \bigl[ H^1\bigl(\Lc(\overline{\Q}),\Q\bigr)\bigr]_n\lra \bigl[H^1\bigl(\Lc(\overline{\Q}(\Tup_d)),\Q\bigr)\bigr]_n$. Its inverse is induced by the composition of specialization maps (see \cite[\S2.2]{CMRR24} for the definition and relevant properties of specialization)
\[
S=\Sp_{x_{1}\to 0} \circ\Sp_{x_{2}\to 0} \circ  \dots\circ \Sp_{x_d\to 0}\colon \Lc\bigl(\overline{\Q}(\Tup_d)\bigr)\lra \Lc(\overline{\Q}).
\]
Consider a tower of fields 
\[
\overline{\Q}(\Tup_0)\subset  \overline{\Q}(\Tup_1) \subset \dots\subset \overline{\Q}(\Tup_d).
\]

We claim that
\begin{align}\label{FormulaSpecializationImage}
\Sp_{x_k\to 0}(\Lc_{\Tup_k})\subseteq \Lc_{\Tup_{k-1}}\quad \text{for} \quad k \in \{1,\dots,d\}.
\end{align}
For $k=1$, we need to show that 
$\Sp_{x_1\to 0}\Li_n^{\Lc}(\zeta x_1^{a_1})=0$
for $a_1\neq 0.$ We assume that $a_1>0$; the case $a_1<0$ would follow from (\ref{LemmaInversionDepthOne}).  We have
\[
\Sp_{x_1\to 0}\Li_n^{\Lc}(\zeta x_1^{a_1})=\Sp_{x_1\to 0}\bigl(-\Cor(0,\dots,0,\zeta x_1^{a_1},1)\bigr)=
-\Cor(0,\dots,0, 1)=0.
\]
Assume that $k\geq 2$. Lemma \ref{LemmaGeneratorsOfHT} and identity (\ref{FormulaIteratedIntegralsCorrelators}) imply that $\Lc_{\Tup_k}$ is generated by elements $\Cor(z_0,\dots,z_{n})$ where each $z_i$ is either $0$ or is of the form $\zeta\prod_{j=1}^{k}x_j^{a_j}$, and the exponent vectors of nonzero $z_i$ are affinely independent. Since $n\geq 2$, we can use the affine invariance of correlators and reduce to the case when $\nu_k(z_0),\dots, \nu_k(z_n)\geq 0$ and $\nu_k(z_i)=0$ for some $i$, where $\nu_k$ is the valuation with respect to $x_k$. The affine independence of exponents implies that at least one of the exponents is strictly positive. In this case, the definition of specialization implies that  $\Sp_{x_k\to 0}\Cor(z_0,\dots,z_{n})=\Cor(z_0',\dots,z_{n}')$ where 
\[
z_i'=
\begin{cases}
z_i &\text{if } \nu_k(z_i)=0,\\
0 & \text{if } \nu_k(z_i)>0.
\end{cases}
\]
Such an element lies in  $\Lc_{\Tup_{k-1}}$ because a subset of an affinely independent set of vectors remains affinely independent. This proves  (\ref{FormulaSpecializationImage}).

Applying (\ref{FormulaSpecializationImage}) repeatedly, we see that $S(\Lc_{\Tup_d})=0$.  Since 
\[
S_*\colon \bigl[H^1(\Lc\bigl(\overline{\Q}(\Tup_d)\bigr),\Q)\bigr]_n\lra \bigl[H^1(\Lc(\overline{\Q}),\Q)\bigr]_n
\]
is an isomorphism, we get that the induced map
\[
\bigl[H^1(\Lc_{\Tup_d},\Q)\bigr]_n \lra \bigl[H^1(\Lc\bigl(\overline{\Q}(\Tup_d)\bigr),\Q)\bigr]_n
\]
is zero. On the other hand, this map is clearly injective. Indeed, the first cohomology group of a Lie coalgebra is isomorphic to the kernel of the cobracket. Since the Lie coalgebra $\Lc_{\Tup_d}$ is a Lie subcoalgebra of $\Lc\bigl(\overline{\Q}(\Tup_d)\bigr)$, the space $\bigl[H^1(\Lc_{\Tup_d},\Q)\bigr]_n$ is embedded into the space $\bigl[H^1(\Lc\bigl(\overline{\Q}(\Tup_d)\bigr),\Q)\bigr]_n$.
It follows that $\bigl[H^1(\Lc_{\Tup_d},\Q)\bigr]_n=0$ and thus $[P(\Hc_{\Tup_d})]_n=0.$
\end{proof}

From the above lemma, and the fact that primitives commute with filtered colimits,  we immediately obtain the initial claim on primitive elements in \( \Hc(\Tup) \).

\begin{corollary} For any torus $\Tup$ we have $P(\Hc(\Tup))= \Hc_1(\Tup).$ 
\end{corollary}

\subsection{The depth filtration on \texorpdfstring{$\Hc(\Tup)$}{H(T)}} \label{SectionDepthFiltrationOnH}
Our goal is to define a filtration on the Hopf algebra $\Hc_{\Tup}$ which we call the depth filtration. For the standard depth filtration on the Hopf algebra of multiple polylogarithms (see \cite[\S 4.1]{CMRR24}), the space $\Dc_k\Hc(F)$ is spanned by multiple polylogarithms $\Li_{n_1,\dots,n_k}$; in particular $\Dc_0\Hc(F)=\Q$. Importantly, this filtration is not compatible with the coalgebra structure: the coproduct
\begin{equation} \label{FormulaCoproductDilogarithm}
\Delta(\Li_2(x))=1\otimes \Li_2(x) - \log(1-x)\otimes \log(x)+ \Li_2(x)\otimes 1
\end{equation}
of an element $\Li_2(x)\in \Dc_1$ is not contained in $\Dc_0\otimes \Dc_1+ \Dc_1\otimes D_0$. The depth filtration on  $\Hc(\Tup)$ is defined in such a way that elements $\log(x_i)$ lie in $\Dc_0\Hc(\Tup)$, so (\ref{FormulaCoproductDilogarithm}) 
is no longer problematic.

Let $\Tup$ be an algebraic torus. Consider the $\Q$-subspace 
\[
\Log(\Tup) \subseteq \Hc_{\Tup,1}
\]
spanned by elements $\log(x_i)$ for $1\leq i\leq d.$  We have a canonical isomorphism $\Log(\Tup) \cong \Xup(\Tup)\otimes_{\Z} \Q$. 

\begin{definition} The depth filtration on $\Hc_{\Tup}$ is a filtration by subspaces  $\Dc_k\Hc_{\Tup}$ for $k\geq 0$ spanned by elements $a\cdot \prod_{i=1}^m a_{k_i}$ with $k_1+\dots+k_m\leq k,$ where 
$a$ is a product of elements in $\Log(\Tup)$ and $a_{k_i}$ is a  polylogarithm (\ref{FormulaHodgePolylogarithms}) of depth $k_i.$  For an isogeny $p\colon \Tup'\lra \Tup$, the corresponding map $p^* \colon \Hc_{\Tup}\lra \Hc_{\Tup'}$  preserves the depth filtration, thus we obtain the depth filtration on $\Hc(\Tup)$ as the filtered colimit
\[
\Dc_k\Hc(\Tup)=\varinjlim\Dc_k\Hc_{\Tup'}
\]
over the category $\Cc_{\Tup}^{\op}$.
\end{definition}

Notice that for any $\Tup$ we have $\Dc_0\Hc_{\Tup}\cong \Sbb^{\bullet}\Log(\Tup)$. For an isogeny $p\colon \Tup'\lra \Tup$ the map $p^*\colon \Log(\Tup)\lra \Log(\Tup')$ is an isomorphism, so $\Dc_0\Hc(\Tup)\cong\Sbb^{\bullet}\Log(\Tup).$ 

\begin{lemma}\label{LemmaDepthFiltration} The depth filtration makes $\Hc(\Tup_d)$ a filtered Hopf algebra. 
\end{lemma}
\begin{proof} It is sufficient to prove the statement for $\Hc_{\Tup_d}.$
From the definition of the depth filtration it is clear that  $\Hc_{\Tup_d}$ is a filtered algebra, i.e., 
\begin{align}\label{FormulaDepthFiltrationAlgebra}
\Dc_{i}\Hc_{\Tup_d}\cdot \Dc_{j}\Hc_{\Tup_d}\subseteq  \Dc_{i+j}\Hc_{\Tup_d}.
\end{align}
Since the antipode is uniquely determined by the product and coproduct, it is enough to check that $\Hc_{\Tup_d}$ is a filtered coalgebra:
\begin{align}\label{FormulaDepthFiltrationCoAlgebra}
\Delta(\Dc_k\Hc_{\Tup_d})\subseteq \sum_{i+j=k} \Dc_i\Hc_{\Tup_d} \otimes \Dc_j\Hc_{\Tup_d}.
\end{align}
By~\eqref{FormulaDepthFiltrationAlgebra}, it is sufficient to check the inclusion~\eqref{FormulaDepthFiltrationCoAlgebra} on polylogarithms~\eqref{FormulaHodgePolylogarithms}.

Consider an iterated integral $\I(z_0;z_1,\dots,z_n;z_{n+1})$ with each $z_i$ equal to either $0$, $1$, or to a monomial times a root of unity.  Let $k$ be the number of nonzero elements $z_i$ for $1\leq i\leq n$. We claim that $\I(z_0;z_1,\dots,z_n;z_{n+1})\in \Dc_k\Hc_{\Tup_d}.$ Indeed, the identity
    \[
    \I(z_0;0,\dots,0;z_{n+1})=\frac{(\log(z_{n+1})-\log(z_{0}))^n}{n!}\in \Dc_0\Hc_{\Tup_d}
    \]
shows that this is true for $k=0$. The identity
    \[\I(z_{0};z_{1},\dots,z_{n};z_{n+1})=\sum_{j=0}^n\I(z_{0};z_{1},\dots,z_{j};0)\I(0;z_{j+1},\dots,z_{n};z_{n+1})\]
together with (\ref{PathReversal}) shows that it is enough to deal with the case $z_0=0$. Finally, for $z_0=0$ and $1\leq k\leq d$ this can be seen the same way as in the proof of Lemma~\ref{LemmaGeneratorsOfHT}.

Now any polylogarithm of depth $k$ as in~\eqref{FormulaHodgePolylogarithms} is equal to $I(z_0;z_1,\dots,z_n;z_{n+1})$ with exactly $k$ nonzero entries among $z_i$, $i=1,\dots,n$. The statement of the lemma now simply follows by counting nonzero elements in each term of the coproduct.
\end{proof}

Elements $\log(x_i)$ are primitive in $\Hc(\Tup_d)$, so the quotient
    \[\overline{\Hc}(\Tup_d)\coloneqq\Hc(\Tup_d)/\bigl(\Log(\Tup_d)\cdot\Hc(\Tup_d)\bigr)\]
is again a Hopf algebra; we denote the coproduct by $\overline{\Delta}$. The depth filtration on $\Hc(\Tup_d)$ induces the one on $\overline{\Hc}(\Tup_d);$ we have $\Dc_0(\overline{\Hc}(\Tup_d))=\Q.$ Lemma \ref{LemmaDepthFiltration} implies that
\begin{align}\label{FormulaDepthFiltrationBar}
\overline{\Delta}'(\Dc_k\overline{\Hc}(\Tup_d))\subseteq \sum_{\substack{i+j=k\\ i,j\geq 1}} \Dc_i\overline{\Hc}(\Tup_d) \otimes \Dc_j\overline{\Hc}(\Tup_d).
\end{align}
In particular, on associated graded pieces we have
\begin{align}\label{FormulaDepthFiltrationBarGraded}
\overline{\Delta}'(\gr_{k}^{\Dc}\overline{\Hc}(\Tup_d))\subseteq \bigoplus_{\substack{i+j=k\\ i,j\geq 1}} \gr_{i}^{\Dc}\overline{\Hc}(\Tup_d) \otimes \gr_{j}^{\Dc}\overline{\Hc}(\Tup_d).
\end{align}
 
 Let $\Lc(\Tup_d)$ and $\overline{\Lc}(\Tup_d)$  be the Lie coalgebras of indecomposables of $\Hc(\Tup_d)$ and $\overline{\Hc}(\Tup_d)$. Notice that 
 we have an exact sequence
 \[
 0\longrightarrow \Log(\Tup_d) \longrightarrow \Lc_1(\Tup_d)\longrightarrow\overline{\Lc}_1(\Tup_d)\longrightarrow 0
 \]
 and $\Lc_n(\Tup_d)=\overline{\Lc}_n(\Tup_d)$ for $n\geq 2.$

We will need the following lemma.

 \begin{lemma}\label{LemmaDirectSumDecomposition} We have $\gr_0^{\Dc}\left(\Lambda^m\Lc(\Tup_d)\right )\cong \Lambda^m \Log(\Tup_d)$.  For $k\geq 1$ we have a canonical direct sum decomposition
\[
\gr_k^{\Dc}\left(\Lambda^m\Lc(\Tup_d)\right)\cong \bigoplus_{i=0}^{m-1} \left(\Lambda^{i} \Log(\Tup_d) \otimes \gr_k^{\Dc}\Lambda^{m-i}\overline{\Lc}(\Tup_d)\right).
\]
\end{lemma}
\begin{proof}
The statement follows from the formula for the associated graded of a wedge power of a vector space and the following facts: $\gr_0^{\Dc}(\Lc(\Tup_d))\cong\Log(\Tup_d)$ and $\gr_k^{\Dc}(\Lc(\Tup_d))\cong\gr_k^{\Dc}(\overline{\Lc}(\Tup_d))$ for $k\geq 1$.
\end{proof}

\subsection{Coproduct of polylogarithms in \texorpdfstring{$\gr_k^{\Dc} \label{SectionCoproductOfPolylogs}
(\overline{\Hc}(\Tup))$}{gr\_k\textasciicircum{}D(H bar(T))}}
Let $\overline{\Delta}_{k-1,1}$ denote the part of the coproduct $\overline{\Delta}$ landing in $\gr_{k-1}^{\Dc}\overline{\Hc}(\Tup)\otimes \gr_1^{\Dc}\overline{\Hc}(\Tup)$ in the decomposition~\eqref{FormulaDepthFiltrationBarGraded}.  The goal of this section is to compute $\overline{\Delta}_{k-1,1}\Li_{n_1,\dots,n_k}(x_1,\dots,x_k)$. This computation will play the key role in the proof of Proposition \ref{PropositionTruncatedSymbolMap}. This computation is also closely related to some computations of Goncharov, see \cite[\S 6]{Gon01}.

Consider the generating function 
    \[\Li(x_1,\dots,x_k;t_1,\dots,t_k) = \sum_{n_1,\dots,n_k>0}\Li_{n_1,\dots,n_k}(x_1,\dots,x_k)t_1^{n_1-1} \!\cdots t_k^{n_k-1} \in \gr_k^{\Dc}(\overline{\Hc}(\Tup))[\![t_1,\dots,t_k]\!].\]
In the formulation of the following lemma we extend $\overline{\Delta}$ coefficientwise to power series in $t_1,\dots,t_k$.
\begin{lemma} \label{LemmaCoproductGeneratingFunction}
The generating function $\Li(x_1,\dots,x_k;t_1,\dots,t_k)$ satisfies the identity
    \begin{align*}
    \overline{\Delta}_{k-1,1} & \Li(x_1,\dots,x_k;t_1,\dots,t_k) \,\, = \,\, 
    \Li(x_2,\dots,x_k;t_2,\dots,t_k)\otimes \Li(x_1;t_1)\\ 
    &+ \sum_{i=1}^{k-1}\Li(x_1,\dots,x_{i}x_{i+1},\dots,x_k;t_1,\dots,\widehat{t_{i+1}},\dots,t_k) \otimes \Li(x_{i+1};t_{i+1}-t_{i})\\
    &- \sum_{i=1}^{k-1}\Li(x_1,\dots,x_{i}x_{i+1},\dots,x_k;t_1,\dots,\widehat{t_i},\dots,t_k) \otimes \Li(x_{i};t_i-t_{i+1}).
    \end{align*}
\end{lemma}

We will need the following lemma.
\begin{lemma} \label{LemmaDivergentLis}
For any $k,l\geq0$ and any $z_1,z_2,z_3\in\mu_{\infty}\Xup(\Tup)$, where $\mu_{\infty}$ is the set of all roots of unity, we have the following identity in $\overline{\Hc}(\Tup)$
    \[\I(z_1;\underbrace{0,\dots,0}_{k},z_2,\underbrace{0,\dots,0}_{l};z_3) = 
    (-1)^{k+1}\binom{k+l}{k}\Big(\Li_{k+l+1}(z_3/z_2)-\Li_{k+l+1}(z_1/z_2)
    \Big).\]
\end{lemma}
\begin{proof}
Note that for $z\in\{z_1,z_3\}$ both elements $\I(0;0,\dots,0;z)$ and $\I(z;0,\dots,0;0)$ vanish in $\overline{\Hc}(\Tup)$, since each is equal to a rational multiple of some power of $\log(z)$. Combining this observation with the path decomposition formula gives the following identity (again, only in $\overline{\Hc}(\Tup)$)
\[
\I(z_1;\underbrace{0,\dots,0}_{k},z_2,\underbrace{0,\dots,0}_{l};z_3) = 
\I(0;\underbrace{0,\dots,0}_{k},z_2,\underbrace{0,\dots,0}_{l};z_3) + 
\I(z_1;\underbrace{0,\dots,0}_{k},z_2,\underbrace{0,\dots,0}_{l};0).
\]
Therefore, it suffices to prove the lemma in the case when $z_1=0$, which we henceforth assume. We argue by induction on $k$. For $k=0$ the identity follows from the definition of $\Li_{k+l+1}$. Assume that $k>0$ and we know the formula
\[
\I(0;\underbrace{0,\dots,0}_{k'},z_2,\underbrace{0,\dots,0}_{l};z_3) = 
    (-1)^{k'+1}\binom{k'+l}{k'}\Li_{k'+l+1}(z_3/z_2)
\]
for all $k'<k$. By shuffle product identity
\[
\I(0;0;z_3)\I(0;\underbrace{0,\dots,0}_{k-1},z_2,\underbrace{0,\dots,0}_{l};z_3) = 
k\I(0;\underbrace{0,\dots,0}_{k},z_2,\underbrace{0,\dots,0}_{l};z_3)
+l\I(0;\underbrace{0,\dots,0}_{k-1},z_2,\underbrace{0,\dots,0}_{l+1};z_3).
\]
Since the left hand side vanishes in $\overline{\Hc}(\Tup)$, and $k\binom{k+l}{k}=l\binom{k+l}{k-1}$, this completes the proof of the induction step.
\end{proof}

\begin{proof}[Proof of Lemma~\ref{LemmaCoproductGeneratingFunction}]
Recall~\eqref{LinDef} and~\eqref{Icoproduct}:
    \begin{equation} \label{LitoInt}
    \begin{split}
    \Li_{n_1,\dots,n_k}&(x_1,x_2,\dots,x_k)\\
    &=(-1)^{k}\I(0;1,\underbrace{0,\dots,0}_{n_1-1},x_1,\dots,\underbrace{0,\dots,0}_{n_{k-1}-1},x_1x_2\dots x_{k-1},\underbrace{0,\dots,0}_{n_k-1};x_1x_2\dots x_{k})
    \end{split}
    \end{equation}
and
    \begin{align*}
    &\Delta(\I(z_0;z_1,\dots,z_n;z_{n+1}))=\\
    &\sum_{\substack{0=i_0<i_1<\dots \\ \cdots <i_r<i_{r+1}=n+1}}\I(z_{i_0};z_{i_1},\dots,z_{i_r};z_{i_{r+1}})\otimes \prod_{p=0}^r\I(z_{i_p};z_{i_{p}+1},\dots,z_{i_{p+1}-1};z_{i_{p+1}}).
    \end{align*}
We set $(z_0,z_1,\dots,z_n,z_{n+1})$ to be the list of arguments appearing in the expression of $\Li_{n_1,\dots,n_k}$ as an iterated integral given in~\eqref{LitoInt}.
Let $I\coloneqq \{i_0,\dots,i_{r+1}\}$ correspond to a term in $\overline\Delta_{k-1,1}\Li_{n_1,\dots,n_k}(x_1,x_2,\dots,x_k)$, and set $\overline{I}\coloneqq \{0,\dots,n+1\} \sm I$. 
We will first show that the term in the coproduct is zero unless $\overline{I}$ is an interval of consecutive integers. Note that the depth of $\I(z_{i_0};z_{i_1},\dots,z_{i_r};z_{i_{r+1}})$ is bounded by the number of nonzero entries among $z_{i_1},\dots,z_{i_r}$ and that the total number of nonzero entries among $z_1,\dots,z_n$ is $k$.
Assume that for some integers $0\leq s<t\leq n$ we have $i_{s+1}-i_s\geq2$ and $i_{t+1}-i_t\geq2$.
Since $\I(y;\{0\}^{m};y')$, $m\geq1$, belongs to the ideal generated by $\log(\Tup)$, if $z_{i_s+1}=\dots=z_{i_{s+1}-1}=0$ or $z_{i_t+1}=\dots=z_{i_{t+1}-1}=0$, then the term vanishes. Otherwise, the depth of $I(z_{i_0};z_{i_1},\dots,z_{i_r};z_{i_{r+1}})$ is at most $k-2$ as two nonzero indices were then omitted, and thus the term in the coproduct $\overline\Delta_{k-1,1}\Li_{n_1,\dots,n_k}(x_1,x_2,\dots,x_k)$ also vanishes.

Now let $I\coloneqq \{i_0,\dots,i_{r+1}\}$ correspond to a nonzero term in $\overline\Delta_{k-1,1}\Li_{n_1,\dots,n_k}(x_1,x_2,\dots,x_k)$; by the above discussion, $\overline{I}={\{l+1,\dots,l'-1\}}$ for some $l<l'$. The only term for which $1\in \overline{I}$ must have $l=0$, and thus $l'=n_1+1$, so the corresponding term in the coproduct is
    \[
    (-1)^k\I(0;x_1,\underbrace{0,\dots,0}_{n_2-1},x_1x_2,\underbrace{0,\dots,0}_{n_3-1},x_1x_2x_3,\dots;x_1\dots x_k)\otimes \I(0;1,\underbrace{0,\dots,0}_{n_1-1};x_1),
    \]
which is equal to $\Li_{n_2,\dots,n_k}(x_2,\dots,x_k)\otimes \Li_{n_1}(x_1)$. For any term for which $z_{l} z_{l'}\ne 0$ we must have
    \[\I(z_{l};z_{l+1},\dots,z_{l'-1};z_{l'}) = \I(x_1\dots x_{i-1};\underbrace{0,\dots,0}_{n_i-1},x_1\dots x_{i},\underbrace{0,\dots,0}_{n_{i+1}-1};x_1\dots x_{i+1}),\]
and so the corresponding term in the coproduct is equal to
    \begin{equation}\label{Term1}
    \begin{split}
    \Li_{n_1,\dots,n_{i-1},1,n_{i+2},\dots,n_{k}}&(x_1,\dots,x_{i}x_{i+1},\dots,x_k)\\ 
    &\otimes (-1)^{n_i}\binom{n_i+n_{i+1}-2}{n_i-1}\Big(\Li_{n_i+n_{i+1}-1}(x_{i+1})-\Li_{n_i+n_{i+1}-1}(x_i^{-1})\Big)
    \end{split}
    \end{equation}
The remaining nonzero terms are the ones for which
    \[\I(z_{l};z_{l+1},\dots,z_{l'-1};z_{l'}) = \I(x_1\dots x_{i-1};\underbrace{0,\dots,0}_{n_{i}-1},x_1\dots x_i,\underbrace{0,\dots,0}_{n_{i+1}-a-1};0),\]
with contribution 
    \begin{equation}\label{Term2}
   	\begin{split}
    -\Li_{n_1,\dots,n_{i-1},a+1,n_{i+2},\dots,n_k}&(x_1,\dots,x_{i-1},x_ix_{i+1},x_{i+2},\dots,x_k) \\
    &\otimes (-1)^{n_{i+1}'-1}\binom{n_i+n_{i+1}'-2}{n_i-1}\Li_{n_i+n_{i+1}'-1}(x_i)
    \end{split}
    \end{equation}
for $0<a<n_{i+1}$, where we set $n_{i+1}'=n_{i+1}-a>0$, and the ones for which
    \[\I(z_{l};z_{l+1},\dots,z_{l'-1};z_{l'}) = \I(0;\underbrace{0,\dots,0}_{n_{i}-b-1},x_1\dots x_i,\underbrace{0,\dots,0}_{n_{i+1}-1};x_1\dots x_{i+1}),\]
with contribution
	\begin{equation}\label{Term3}  
	\begin{split}
	\Li_{n_1,\dots,n_{i-1},b+1,n_{i+2},\dots,n_k}&(x_1,\dots,x_{i-1},x_ix_{i+1},x_{i+2},\dots,x_k) \\
	&\otimes (-1)^{n_i'-1}\binom{n_ i'+n_{i+1}-2}{n_{i+1}-1}\Li_{n_i'+n_{i+1}-1}(x_{i+1}),
	\end{split} 
	\end{equation}
for $0<b<n_i$, where we set $n_i'=n_i-b$. Note that~\eqref{Term1} is equal to the sum of~\eqref{Term2} for $a=0$ and ~\eqref{Term3} for $b=0$. 

Let us multiply~\eqref{Term2} by $t_1^{n_1-1}\cdots t_{k}^{n_k-1}$, sum over $a=0,1,\dots,n_{i+1}-1$ and then over $n_1,\dots,n_k>0$. This is equivalent to summing over all $n_1,\dots,n_{i},n_{i+2},\dots,n_k, a, n_{i+1}'>0$. Making use of binomial theorem then gives
	\[-\Li(x_1,\dots,x_ix_{i+1},\dots,x_k;t_1,\dots,t_{i-1},t_{i+1},\dots,t_{k})\otimes
	\Li(x_{i};t_i-t_{i+1})\]
Similarly, summing up~\eqref{Term3} in the same way we get
	\[\Li(x_1,\dots,x_ix_{i+1},\dots,x_k;t_1,\dots,t_{i},t_{i+2},\dots,t_{k})\otimes
	\Li(x_{i+1};t_{i+1}-t_i)\]
Finally, from $\Li_{n_2,\dots,n_k}(x_2,\dots,x_k)\otimes \Li_{n_1}(x_1)$ we obtain
	\[\Li(x_2,\dots,x_k;t_2,\dots,t_k)\otimes \Li(x_1;t_1).\]
Collecting these contributions together gives the claim.
\end{proof}

\subsection{Classical polylogarithms on a torus} 
We have seen that $\Dc_0\overline{\Hc}(\Tup_d)=\Q.$ Notice that the projection $\Dc_1\overline{\Hc}_n(\Tup_d)\lra \gr_{1}^\Dc\overline{\Hc}_n(\Tup_d)$ is an isomorphism for any $n\geq 1$. Our goal in this section is to describe the space $\gr_{1}^\Dc\overline{\Hc}(\Tup_d)$. This space is generated by classical polylogarithms $\Li_{n}(\zeta x_1^{a_1}\cdots x_d^{a_d})$; we will show that all relations between these elements are generated by the usual distribution relations.

Let $\Xupcirc(\Tup_d)\coloneqq \Xup(\Tup_d)\sm\{1\}$. Denote by $\Xup_{+}^{\prim}(\Tup_d)$ the set of primitive lexicographically positive elements of $\Xupcirc(\Tup_d)$. Here by a lexicographically positive character we mean a character whose exponent vector (with respect to $x_1,\dots,x_d$) is lexicographically greater than $0$ in $\Z^d$. Note that for any character $\chi \in\Xupcirc(\Tup_d)$ there exists a unique character $\psi \in \Xup_{+}^{\prim}(\Tup_d)$ and a unique integer $n$ such that $\chi=\psi^n$.

We will need the following lemma.
\begin{lemma}\label{LemmaLinearIndependence}
The elements $\Li_{n}^{\Lc}(\zeta x_1^{a_1}\cdots x_d^{a_d})$ for $\zeta\in\mu_{\infty}$ and $x_1^{a_1}\cdots x_d^{a_d}\in \Xup_{+}^{\prim}(\Tup_d)$ are linearly independent in $\overline{\Lc}_{\Tup_d,n}.$
\end{lemma}

\begin{proof}
We prove the statement by induction on weight. In weight one, we need to show that $\Li_1^{\Lc}(\zeta x_1^{a_1}\cdots x_d^{a_d})$ with primitive and lexicographically positive exponent $(a_1,\dots,a_n)$ are linearly independent in $\overline{\Lc}_{\Tup_d,1}$. The space $\overline{\Lc}_{\Tup_d,1}$ is a subspace of $\Lc_1(\overline{\Q}(\Tup_d))/\Log(\Tup_d)\cong \overline{\Q}(\Tup_d)^{\times}/\langle x_1,\dots,x_d \rangle$, so the statement follows from the fact that Laurent polynomials $1-\zeta x_1^{a_1}\cdots x_d^{a_d}$ for $\zeta \in \mu_{\infty}$ and $x_1^{a_1}\cdots x_d^{a_d}\in X_{+}^{\prim}(\Tup_d)$ are irreducible and pairwise nonassociate in the unique factorization domain $\overline{\Q}[x_1^{\vphantom{-1}},x_1^{-1},\dots,x_d^{\vphantom{-1}},x_d^{-1}]$. That different $1-\zeta x_1^{a_1}\cdots x_d^{a_d}$ are pairwise nonassociate follows from the fact that they vanish on different subsets of $\Tup_d$; to see the irreducibility it suffices to note that there exists an automorphism of $\Tup_d$ that maps the given polynomial to $1-\zeta x_1$, which is irreducible.

In weight $n\geq 2$, we have $\Lc_{\Tup_d,n} \cong \overline{\Lc}_{\Tup_d,n}$, so it is sufficient to prove the linear independence in $\Lc_{\Tup_d,n}$ By Lemma \ref{LemmaDirectSumDecomposition},  the space $\left [\gr_1^{\Dc}\left(\Lambda^2\Lc_{\Tup_d}\right)\right]_n$ is isomorphic to $\Log(\Tup_d)\otimes \gr_1^{\Dc}\overline{\Lc}_{\Tup_d,n-1}$ and the image of 
the cobracket $\delta \Li_n^{\Lc}(\zeta x_1^{a_1}\cdots x_d^{a_d})$ 
in $\gr_1^{\Dc}\left(\Lambda^2\Lc_{\Tup_d}\right)$  equals to 
$-\log(\zeta x_1^{a_1}\cdots x_d^{a_d})\otimes \Li_{n-1}^{\Lc}(\zeta x_1^{a_1}\cdots x_d^{a_d})$. 
The statement follows by the induction hypothesis.
\end{proof}

Fix $n\geq 1$. Consider a map 
\[
L_n \colon \Q\bigl[\mu_{\infty}\times \Xupcirc(\Tup_d)\bigr] \longrightarrow\Dc_1\overline{\Hc}_{\Tup_d,n}
\]
sending $\bigl[\zeta, \chi \bigr]$ to $\Li_{n}(\zeta \chi).$ This map is surjective. Lemma \ref{LemmaDistributionLi} and (\ref{LemmaInversionDepthOne}) imply that the identity
\[
\sum_{\nu^m=1}\Li_{n}(\nu \zeta x_1^{a_1}\cdots x_d^{a_d}) =
    m^{1-n}\Li_{n}(\zeta^m x_1^{ma_1}\cdots x_d^{ma_d})
\]
holds in $\overline{\Hc}(\Tup_d)$ for $m\in \Z\sm \{0\}$. Thus we have a complex
\begin{equation}  \label{FormulaExactSequenceClssicalPolylogs}
\bigoplus_{m\in \Z\sm\{0\}} \Q\bigl[\mu_{\infty}\times\Xupcirc(\Tup_d)\bigr] \xrightarrow{\sum d_{n,m}} \Q\bigl[\mu_{\infty}\times\Xupcirc(\Tup_d)\bigr] \stackrel{L_n}{\longrightarrow}\Dc_1\overline{\Hc}_{\Tup_d,n}\longrightarrow 0
\end{equation} 
where 
\[
d_{n,m}[\zeta, x_1^{a_1}\cdots x_d^{a_d}]=[\zeta^m,  x_1^{ma_1}\cdots x_d^{ma_d}]-m^{n-1}\sum_{\nu^m=1}[\nu\zeta,  x_1^{a_1}\cdots x_d^{a_d}].
\]

\begin{lemma} 
The complex (\ref{FormulaExactSequenceClssicalPolylogs}) is exact.
\end{lemma}
\begin{proof}
We need to show that the kernel of $L$ is generated by the image of the map $\sum d_{n,m}$. Modulo these elements, an element of $\Q\bigl[\mu_{\infty}\times{\Xupcirc}(\Tup_d)\bigr]$ can be expressed as a linear combination of $[\zeta, x_1^{a_1}\cdots x_d^{a_d}]$ with $x_1^{a_1}\cdots x_d^{a_d}\in {\Xupcirc}(\Tup_d)$ primitive and lexicographically positive. It suffices to prove that elements $\Li_n(\zeta x_1^{a_1}\cdots x_d^{a_d})$ for different primitive and lexicographically positive vectors $(a_1,\dots,a_d)$ are linearly independent in $\overline{\Hc}_{\Tup_d}$. This follows from Lemma \ref{LemmaLinearIndependence}.
\end{proof}

Let $p\colon \Tup'_d\longrightarrow \Tup_d$ be an isogeny. We have a commutative diagram
\[
\begin{tikzcd}[column sep=large]
 \smash{\displaystyle\bigoplus\limits_{m\in \Z\sm\{0\}}} \Q\bigl[\mu_{\infty}\times \Xupcirc(\Tup_d)\bigr] \arrow[d, "p^{*}"] \arrow[r,"\sum d_{n,m}"] &\Q\bigl[\mu_{\infty}\times\Xupcirc(\Tup_d)\bigr]  \arrow[d, "p^{*}"] \arrow[r,"L_n"] & \Dc_1\overline{\Hc}_{\Tup_d,n} \arrow[r]  \arrow[d, "p^{*}"]&0\\
\displaystyle\bigoplus\limits_{m\in \Z\sm\{0\}}  \Q\bigl[\mu_{\infty}\times\Xupcirc(\Tup'_d)\bigr]  \arrow[r,"\sum d_{n,m}"] &\Q\bigl[\mu_{\infty}\times\Xupcirc(\Tup'_d)\bigr]  \arrow[r,"L_n"] &  \Dc_1\overline{\Hc}_{\Tup'_d,n} \arrow[r] & 0.
\end{tikzcd}
\]
where $p^*\colon \Q\bigl[\mu_{\infty}\times \Xupcirc(\Tup_d)\bigr] \longrightarrow \Q\bigl[\mu_{\infty}\times\Xupcirc(\Tup'_d)\bigr]$ acts trivially on $\mu_{\infty}$ and acts as a pull-back on $\Xup(\Tup_d).$ Passing to the filtered colimit over the category $\Cc_{\Tup_d}^{\op}$, we obtain an exact sequence
\[
\bigoplus_{m\in \Z\sm\{0\}} \varinjlim\bigl(\Q\bigl[\mu_{\infty}\times\Xupcirc(\Tup'_d)\bigr] \bigr) \xrightarrow{\qquad} \varinjlim\Q\bigl[\mu_{\infty}\times\Xupcirc(\Tup'_d)\bigr]   \xrightarrow{\qquad}  \Dc_1\overline{\Hc}_n(\Tup_d) \xrightarrow{\qquad}   0.
\]
Observe that an injective map between lattices $\Xup(\Tup_1)\hookrightarrow\Xup(\Tup_2)$ induces an isomorphism $\Xup(\Tup_1)_{\Q}\rightarrow\Xup(\Tup_2)_{\Q}$ on their rationalizations. Denote $\Xup(\Tup_d)_{\Q}\sm\{0\}$ by $\Xupcirc(\Tup_d)_{\Q}$. We have a map
\[
\varinjlim\Q\bigl[\mu_{\infty}\times\Xupcirc(\Tup'_d)\bigr] \lra \Q\bigl[\mu_{\infty}\times \Xupcirc(\Tup_d)_{\Q}\bigr]
\]
which can be easily shown to be an isomorphism.

We have proven the following lemma.

\begin{lemma}\label{DepthOneOnColimit} For $n\geq 1$, we have an exact sequence 
\begin{equation}  \label{FormulaExactSequenceClssicalPolylogs2}
\begin{tikzcd}[column sep=large]
 {\displaystyle\bigoplus\limits_{m\in \Z\sm\{0\}}} \Q\bigl[\mu_{\infty}\times \Xupcirc(\Tup_d)_{\Q} \bigr] \arrow[r,"\sum d_{n,m}"] &\Q\bigl[\mu_{\infty}\times\Xupcirc(\Tup_d)_{\Q} \bigr]  \arrow[r,"L_n"] & \Dc_1\overline{\Hc}(\Tup_d) \arrow[r] & 0 \end{tikzcd}.
 \end{equation}
\end{lemma}
For convenience, we spell out explicit formulas for the maps $d_{n,m}$ and $L_n$. We identify $\Xupcirc(\Tup_d)_{\Q}$ with $\Q^d\sm \{0\}$. The map $d_{n,m}$ is defined by formula
\[
d_{n,m}[\zeta, (a_1,\dots,a_d)]=[\zeta^m,  (ma_1,\dots,ma_d)]-m^{n-1}\sum_{\nu^m=1}[\nu\zeta,  (a_1,\dots,a_d)].
\]
The element $L_n([\zeta, (a_1,\dots,a_d)])$ can be obtained as follows. Consider $N\in \N$ such that $Na_1,\dots, Na_d \in\Z$ and let $p_N\colon \Tup'_d\lra \Tup_d$ be the isogeny sending $(y_1,\dots,y_d)\in \Tup'_d$ to $(y_1^N,\dots,y_d^N)\in \Tup_d$. Then $L_n([\zeta, (a_1,\dots,a_d)])$ is the image of $\Li_n(\zeta y_1^{Na_1}\cdots y_d^{Na_d})$ under the natural map from  $ \overline{\Hc}_{\Tup'_d}$ to $\overline{\Hc}(\Tup_d).$

\subsection{Primitive elements in \texorpdfstring{$\overline{\Hc}(\Tup)$}{H bar(T)}}
Classical polylogarithms, i.e., polylogarithms $\Li_n$, of depth 1, are primitive elements of $\overline{\Hc}(\Tup)$. It turns out that the converse is true as well: every primitive element 
of $\overline{\Hc}(\Tup)$ lies in $\Dc_1\overline{\Hc}(\Tup)$. This statement should be compared to the Goncharov Freeness Conjecture \cite[Conjecture B]{Gon95B}, which states that the space of primitive elements of the Hopf algebra $\Hc(F)/\left(F^{\times}\Hc(F) \right)$ is generated by classical polylogarithms.

\begin{proposition}\label{PropositionWeakDepthConjecture}
An element $a\in \Hc_n(\Tup_d)$ with $\overline{\Delta}(a)=1\otimes a+a\otimes 1$ can be written as a sum of classical polylogarithms $\Li_{n}(\zeta x_1^{a_1}\cdots x_n^{a_n})$  and elements in the ideal generated by $\Log(\Tup)$. We have the following isomorphism:
    \[
   P\bigl(\overline{\Hc}(\Tup_d)\bigr)\cong\gr_{1}^\Dc\Bigl(\overline{\Hc}(\Tup_d)\Bigr).
    \]
    
\end{proposition} 
\begin{proof} Notice that a primitive element in $\overline{\Hc}(\Tup_d)$ is an image of a primitive element in  $\overline{\Hc}_{\Tup'}$ for some torus $\Tup'$ over $\Tup$.
It is sufficient to show that primitive elements in $\overline{\Hc}_{\Tup}$ have depth one because then the same would be true for every torus $\Tup'$ over $\Tup$ and the statement would follow. Next, since the natural projection $P(\overline{\Hc}_{\Tup}) \lra \Ker(\overline{\delta})\subseteq \overline{\Lc}_{\Tup_d}$ is an isomorphism by Milnor-Moore (see Lemma~\ref{LemmaMilnorMoore}), we need to show that any element $a\in \overline{\Lc}_{\Tup_d}$ with $\overline{\delta}(a)=0$ lies in $\Dc_1\overline{\Lc}_{\Tup}$. We may assume that $a$ is homogeneous of weight $n$. 

We argue by induction on $n$.  For $n=1$ the statement is obvious because every element in  $\overline{\Lc}_{\Tup_d,1}$ has depth $1$. Assume that $a$ has weight $2$. Recall that $\overline{\Lc}_{\Tup_d,2}\cong \Lc_{\Tup_d,2}$, and so we can look at the cobracket $\delta(a)\in \Lambda^2\Lc_{\Tup_d,1}$. Since $\overline{\delta}(a)=0$, the cobracket $\delta(a)$ lies in the subspace $\Lc_{\Tup_d,1}\wedge \Log(\Tup_d)$ of $\Lambda^2\Lc_{\Tup_d,1}$ and we can express it as follows:
\begin{equation}\label{FormulaCobracketWeight2}
\delta(a)=\sum_{\substack{\zeta\in \mu_{\infty} \\ \chi\in \Xup_+^{\prim}(\Tup_d)}}
n_{\zeta,\chi}\Li_{1}(\zeta \chi)\wedge \log(\psi_{\zeta,\chi})+\sum_{i<j} c_{ij}\log(x_i)\wedge\log(x_j)
\end{equation}
for some $n_{\zeta,\chi}, c_{ij}\in \Q, \: \psi_{\zeta,\chi}\in \Xup^{\prim}_{+}(\Tup_d)$. 

Recall that for any field $K$, discrete valuation $\nu$, uniformizer $\pi$, and residue field $k$ we have a residue map $\res\colon\Lambda^2 K^{\times}_{\Q}\lra k^{\times}_{\Q}$ such that $\res(u_1\wedge u_2)=0$ and $\res(u_1 \wedge \pi)=\overline{u_1}$ for any units $u_1,u_2\in \mathcal{O}_{K}^{\times}$. It is easy to see that the residue map vanishes on classical dilogarithms: for any $x\in K$ we have
\[
\res\bigl(\delta(\Li_2^{\Lc}(x))\bigr)=\res(x\wedge (1-x))=0.
\] 
Since $\Lc_2(K)$ is spanned by classical dilogarithms \cite[\S 4.4]{CMRR24}, we have the identity $\res(\delta(a))=0$ for any $a\in \Lc_2(K)$.

Consider a field $K=\overline{\Q}(\Tup_d)$ and recall that $\Lc_{\Tup_d,1} \cong K^{\times}_\Q$. Every character $\chi \in \Xup_+^{\prim}(\Tup_d)$ and a root of unity $\zeta$ defines a discrete valuation on $K$ corresponding to the divisor $\{\zeta\chi=1\}$ of $\Tup_d$. Let $\res_{\chi,\zeta}$ be the residue map corresponding to this discrete valuation and a uniformizer $\pi=1-\zeta\chi$.
Elements $1-\zeta \chi$ for $\chi\in \Xup_+^{\prim}(\Tup_d)$ and $\zeta\in \mu_{\infty}$ are irreducible and pairwise non-associated in the ring $\overline{\Q}[\Tup_d]$. Thus for any $\chi'\in \Xup_+^{\prim}(\Tup_d)$ and $\zeta'\in \mu_{\infty}$  we have
\[
\res_{\chi,\zeta}\bigl((1-\chi'\zeta')\wedge \psi\bigr)=
\begin{cases}
\overline{\psi} & \text{if} \   \chi=\chi' \   \text{and}\  \zeta =\zeta',\\
0 & \text{otherwise}.
\end{cases}\]
Moreover, $\res_{\chi,\zeta}(x_i\wedge x_j)=0$ for any $i,j\in \{1,\dots, d\}$.

Applying the map $\res_{\chi,\zeta}$ to (\ref{FormulaCobracketWeight2}) we obtain 
\[
0=\res_{\chi,\zeta}(\delta(a))=n_{\zeta,\chi} \res_{\chi,\zeta}\bigl(-(1-\zeta \chi)\wedge \psi_{\zeta,\chi}\bigr)=-n_{\zeta,\chi}\overline{\psi_{\zeta,\chi}}.
\]
It follows that either $n_{\zeta,\chi}=0$ or the image of
$\psi_{\zeta,\chi}$  vanishes in $\bigl(\overline{\Q}[\Tup_d]/(1-\zeta\chi)\bigr)_{\Q}^{\times}$. In the latter case, there exists $N\in \N$ such that $1-\psi_{\zeta,\chi}^N$ is contained in the ideal $(1-\zeta\chi)$. From here, since $\chi,\psi_{\zeta,\chi}\in \Xup^{\prim}_{+}(\Tup_d)$, it is easy to deduce that $\psi_{\zeta,\chi}$ must be equal to $\chi$.
In either case, 
\begin{equation}\label{FormulaLog}
n_{\zeta,\chi}\log(\psi_{\zeta,\chi})= n_{\zeta,\chi}\log(\chi) =n_{\zeta,\chi}\log(\zeta\chi).
\end{equation}
Using \eqref{FormulaLog}, we can rewrite \eqref{FormulaCobracketWeight2} as
\[
\delta\biggl(a-\sum_{\substack{\zeta\in \mu_{\infty} \\ \chi\in \Xup_+^{\prim}(\Tup)}}
n_{\zeta,\chi}\Li_{2}^{\Lc}(\zeta \chi)\biggr )= \sum_{i<j} c_{ij}\log(x_i)\wedge\log(x_j).
\]
Applying residue maps to the divisors $x_i=0$, we see that $c_{ij}=0$. It follows that 
\[
\delta\biggl(a-\sum_{\substack{\zeta\in \mu_{\infty} \\ \chi\in \Xup_+^{\prim}(\Tup)}}
n_{\zeta,\chi} \Li_{2}^{\Lc}(\zeta \chi)\biggr)=0,
\]
and so $a$ lies in $\Dc_1\Lc_{\Tup_d,2}$  by Lemma \ref{LemmaMilnorMoore} and Lemma \ref{LemmaCohomologyOfHT}.

Now, assume that $a$ has weight $n\geq 3$. Since $\overline{\delta}(a)=0,$ we have 
\[
\delta(a)\in \Lc_{\Tup_d,n-1} \wedge \Log(\Tup_d)\subseteq \Lambda^2 \Lc_{\Tup_d}.
\]
Assume that $\delta(a)=\sum_{i=1}^d a_i\wedge \log(x_i)$
    for $a_i\in \Lc_{\Tup_d,n-1}$. By the coJacobi identity $\overline{\delta}(a_i)=0$, and so, by the induction hypothesis, $a_i\in \Dc_1\Lc_{\Tup_d,n-1}$. So, we have
\[
\delta(a)=\sum_{\substack{\zeta\in \mu_{\infty} \\ \chi\in \Xup_+^{\prim}(\Tup)}}
n_{\zeta,\chi}\Li_{n-1}^{\Lc}(\zeta \chi)\wedge \log(\psi_{\zeta,\chi})
\] 
for some $n_{\zeta,\chi}\in \Q, \: \psi_{\zeta,\chi}\in \Xup^{\prim}_{+}(\Tup_d)$.  By the coJacobi identity,
\[
0=\delta(\delta(a))=\delta \biggl(\sum_{\substack{\zeta\in \mu_{\infty} \\ \chi\in \Xup_+^{\prim}(\Tup)}} \!\!\! 
n_{\zeta,\chi}\Li_{n-1}^{\Lc}(\zeta \chi)\wedge \log(\psi_{\zeta,\chi})\biggr) = \sum_{\substack{\zeta\in \mu_{\infty} \\ \chi\in \Xup_+^{\prim}(\Tup)}} \!\!\! 
n_{\zeta,\chi}\Li_{n-2}^{\Lc}(\zeta \chi)\wedge\log(\zeta \chi)\wedge \log(\psi_{\zeta,\chi}).
\] 
Lemma~\ref{LemmaLinearIndependence} implies that if $n_{\zeta,\chi}\neq 0$ then $\psi_{\zeta,\chi}$ and $\chi$ are  linearly dependent. Since  both $\psi_{\zeta,\chi}$ and $\chi$ lie in $\Xup_+^{\prim}(\Tup_d)$,
we must have $n_{\zeta,\chi} \log(\psi_{\zeta,\chi}) = n_{\zeta,\chi} \log(\zeta\chi)$. It follows that
\[
\delta\biggl(a-\sum_{\substack{\zeta\in \mu_{\infty} \\ \chi\in \Xup_+^{\prim}(\Tup)}}
n_{\zeta,\chi}\Li_{n}^{\Lc}(\zeta \chi)\biggr)=0
\]
and so $a$  lies in $\Dc_1\Lc_{\Tup_d,n}$ by Lemma \ref{LemmaMilnorMoore} and Lemma \ref{LemmaCohomologyOfHT}. This finishes the proof.
\end{proof}

\section{The truncated symbol map} 
In this section, we introduce the $\GL_d(\Q)$-module $\LL_{n}(\Tup_d)$, which was discussed informally in \S \ref{SectionMPandSteinbergModule}. We will study its main properties and construct the truncated symbol map $\STup$ which appears in the statement of Theorem \ref{TheoremMain2}.

\subsection{Module \texorpdfstring{$\LL_{n}(\Tup)$}{L\_n(T\_d)}}\label{SectionModuleLn}
 Let $\Tup$ be an algebraic torus  and consider a $\Q$-vector space $V=\Xup(\Tup)_\Q$. 
 In \S \ref{SectionMPandSteinbergModule} we introduced a $\Q$-vector space $\LL_{n}(\Tup)$ which we now define.

For an isogeny $p\colon \Tup'\lra \Tup$ the Galois group $\Gamma_{\Tup',\Tup}=\Gal\bigl(\overline{\Q}(\Tup')/\overline{\Q}(\Tup)\bigr)$ acts on the Hopf algebra $\Hc\bigl(\overline{\Q}(\Tup')\bigr)$. This action induces an action of $\Gamma_{\Tup',\Tup}$ on the Hopf algebras $\Hc_{\Tup'}$ and  $\overline{\Hc}_{\Tup'}$. Consider the map $p_*\colon \Hc_{\Tup'}\lra \Hc(\Tup)$ defined as a composition of the trace map $\bigl( \sum_{\sigma \in\Gamma_{\Tup',\Tup}} \sigma \bigr)$ from $\Hc_{\Tup'}$ to itself
with the natural map $\Hc_{\Tup'}\lra \Hc(\Tup)$. We call the map $p_*$ \emph{the pushforward along the isogeny $p\colon \Tup'\lra \Tup$}. We also have the induced map $p_*\colon \overline{\Hc}_{\Tup'}\lra \overline{\Hc}(\Tup)$.

\begin{example} \label{ExamplePushforward}
Consider an isogeny of rank two tori $p\colon \Tup'\lra \Tup$ given by $p^{*}(x_1)=y_1^2$, $p^{*}(x_2)=y_2$. Then 
    \[p_{*}(\Li_{1,1}(y_1,y_2))=\Li_{1,1}(\sqrt{x_1},x_2)+\Li_{1,1}(-\sqrt{x_1},x_2).\]
\end{example}

\begin{definition}\label{DefLnDepth} The space $\LL_{n}(\Tup_d)$ is the subspace of $\overline{\Hc}_n(\Tup_d)$ spanned by the images of the functions $\Li_{n_1,n_2,\dots, n_k}(y_1,y_2,\dots,y_k)$ for $k\leq d$ under pushforwards $p_*$ along isogenies $p\colon \Tup'\lra\Tup_d$, where $y_1,\dots,y_d\in \Xup(\Tup')$ is a basis of the character lattice of $\Tup'$.
We define the depth filtration on $\LL_n(\Tup_d)$ by setting $\Dc_k\LL_n(\Tup_d)$ to be the span of all pushforwards of $\Li_{n_1,n_2,\dots, n_{l}}(y_1,y_2,\dots,y_{l})$ for all $n_1+\dots+n_{l}=n$ and $l\leq k$.
\end{definition}

\begin{remark} \label{RemarkDepthFiltration} Recall that in \S \ref{SectionDepthFiltrationOnH} we defined the depth filtration on $\overline{\Hc}_n(\Tup_d)$. It is easy to see that
\begin{equation} \label{FormulaInclusionDepth}
\Dc_k\LL_n(\Tup_d)\subseteq  \LL_n(\Tup_d) \cap \Dc_k\overline{\Hc}_n(\Tup_d).
\end{equation}
Distribution relations (Lemma \ref{LemmaDistributionLi}) imply that for $k=1$ the inclusion \eqref{FormulaInclusionDepth} is an equality. From the proof of Theorem~\ref{TheoremMain2} given in \S\ref{SectionProofoftheoremMain} it follows that inclusion \eqref{FormulaInclusionDepth} is an equality for all $k$.
\end{remark}
Notice that for an isogeny $p\colon\Tup'\lra \Tup$ the space $\LL_{n}(\Tup')$ is a $\bigl(\Gamma_{\Tup',\Tup}\bigr)$-invariant subspace of the filtered colimit
$\displaystyle\varinjlim (\Hc_{\Tup''})^{\Gamma_{\Tup'',\Tup'}}$ over the category $\Cc_{\Tup'}^{\op}$. Thus the space $\LL_{n}(\Tup')$ carries a $\bigl(\Gamma_{\Tup',\Tup}\bigr)$-action. We define a pushforward $p_{*}\colon \LL_{n}(\Tup') \lra \LL_{n}(\Tup)$ by the formula $p_{*}=\sum_{\sigma \in\Gamma_{\Tup',\Tup}} \sigma$; it is clearly functorial. 

The following lemma is a consequence of distribution relations (see Lemma~\ref{LemmaDistributionLi}).
\begin{lemma} \label{LemmaDistribution}
For $N\in \N$ consider an isogeny $p\colon \Tup_d\lra  \Tup_d$ given by the formula
\[
p(x_1,\dots,x_d)=(x_1^N,\dots,x_d^N).
\]
Then we have
\[
p_*\Li_{n_1,n_2,\dots, n_k}(x_1,x_2,\dots,x_k)=
N^{d-n} \Li_{n_1,n_2,\dots, n_k}(x_1,x_2,\dots,x_k).
\]
\end{lemma}

For a matrix $A\in \Mup_{d}(\Z)$ with $\det(A)\neq0$ consider an isogeny $p_A \colon \Tup_d\lra \Tup_d$ defined by the formula
\[
p_A(x_1,\dots,x_d)=\left(\prod_{i=1}^d x_i^{a_{i1}},\dots,\prod_{i=1}^d x_i^{a_{id}} \right).
\] 
For $x\in \LL_{n}(\Tup_d)$ we put 
\[
A\cdot x=\bigl((p_{\det(A)I_d})_{*}\bigr)^{-1}(p_{\adj(A)})_*(x)
\]
where
 $\adj(A)$ is the adjugate matrix of $A$, defined by $\adj(A)=\det(A)A^{-1}$ (when $\det(A)\ne0$). This defines a left action of the monoid  $\{A\in \Mup_{d}(\Z) \setsep \det(A)\neq 0\}$ on  $\LL_{n}(\Tup_d)$.
 
\begin{lemma} \label{LemmaGroupActionExtension}
The action above can be uniquely extended to an action of the group 
$\GL_d(\Q)$ on $\LL_{n}(\Tup_d).$
\end{lemma}
\begin{proof}
By Lemma~\ref{LemmaDistribution}, positive scalar matrices in $\Mup_d(\Z)$ act on $\LL_{n}(\Tup)$ invertibly. Denote by $\rho_A$ the action of $A\in \Mup_{d}(\Z)$ with $\det(A)\ne0$. For $N\in\Z_{>0}$ we also denote by $\rho_N=\rho_{NI_d}$, where $I_d$ is the $d\times d$ identity matrix. 

Note that any element $g\in\GL_d(\Q)$ can be written as $g=N^{-1}A$ for $N\in\Z_{>0}$ and $A\in\Mup_{d}(\Z)$ with $\det(A)\ne0$. Then we define the action of $g=N^{-1}A\in\GL_d(\Q)$ as $\rho_{N}^{-1}\rho^{\vphantom{-1}}_A(x)$. The action is well-defined since if $g=N^{-1}A=M^{-1}B$ for different $N,M\in\Z_{>0}$, then 
\[\rho_{M}^{-1}\,\rho^{\vphantom{-1}}_B = \rho_{NM}^{-1}\,\rho^{\vphantom{-1}}_{NB} = \rho_{MN}^{-1}\,\rho^{\vphantom{-1}}_{MA}=\rho_{N}^{-1}\,\rho^{\vphantom{-1}}_A
\]
(here we used that positive scalar matrices are in the center of $\Mup_d(\Z)$). That this is a group action is also evident.
\end{proof}
It is not hard to calculate that the action of $A\in \Mup_d(\Z)$ with $\lvert\det(A)\rvert=N\ne 0$ on $\Li_{n_1,\dots,n_k}$ is 
\begin{equation}  \label{FotmulaExplicitPushforward}
(A\cdot \Li_{n_1,\dots,n_k})(x_1,\dots,x_d) = N^{n-d-1}\sum_{\substack{y_1^{N}=x_1 \\ \hspace{\widthof{$y_1^{N}$}} \vdotsB \hspace{\widthof{$x_d$}} \\ y_d^{N}=x_d}}
   \Li_{n_1,\dots,n_k}\left(\prod_{i=1}^d y_i^{a_{i1}},\dots,\prod_{i=1}^d y_i^{a_{ik}} \right).
\end{equation}

\begin{remark}\label{RemarkCharacterPsy}  By (\ref{FotmulaExplicitPushforward}), the group $\GL_d(\Q)$ acts on $\LL_0(\Tup_d)\cong \Q$ by the character $\psi(A)= \lvert\det(A)\rvert^{-1}$.
\end{remark}

\begin{example}\label{exampleDepth1dim1} Let us describe $\LL_n(\Tup_1)$ as a $\GL_1(\Q)$-module for $n\geq 1$. Lemma \ref{LemmaDistributionLi} and the inversion relation (\ref{LemmaInversionDepthOne}) imply that this space is one-dimensional, spanned by  $\Li_n(x_1).$ The action of the $1\times1$ matrix $ \begin{pmatrix} q \end{pmatrix} $ is defined by
\[
\begin{pmatrix} q \end{pmatrix} \cdot \Li_n(x_1)= q^{n-1}\Li_n(x_1).
\]
\end{example}

\subsection{The cobar complex for \texorpdfstring{$\LL(\Tup_d)$}{L\_n(T\_d)}} \label{SectionCobarComplexLTd}
The graded $\Q$-vector space $\LL(\Tup_d)$ is a subspace of $\overline{\Hc}(\Tup_d)$, but is neither a subalgebra nor a subcoalgebra. 
In this section we define a subcomplex $\Omega^{\bullet}\LL_{n}(\Tup_d)$ of the cobar complex $\Omega^{\bullet}\overline{\Hc}(\Tup_d)$, which is a substitute for the cobar complex of $\LL(\Tup_d)$.

We start with a computation. Consider the isogeny $p\colon \Tup'\lra \Tup$ as in Example~\ref{ExamplePushforward}. Then 
    \begin{align*}
    \overline{\Delta}' p_{*}(\Li_{1,1}(y_1,y_2))={}& 
      \Li_{1}(x_2)\otimes \Li_{1}(\sqrt{x_1})
     +\Li_{1}(x_2)\otimes \Li_{1}(-\sqrt{x_1})\\
    &+\Li_{1}(x_2\sqrt{x_1})\otimes \Li_{1}(x_2)
     {}+\Li_{1}(-x_2\sqrt{x_1})\otimes \Li_{1}(x_2)\\
    &-\Li_{1}(x_2\sqrt{x_1})\otimes \Li_{1}(\sqrt{x_1}^{\,-1})
     {}-\Li_{1}(-x_2\sqrt{x_1})\otimes \Li_{1}(-\sqrt{x_1}^{\,-1})\\
    ={}& \Li_{1}(x_2)\otimes \Li_{1}(x_1)
      +\Li_{1}(x_1x_2^2)\otimes \Li_{1}(x_2)
      -p_{*}\bigl(\Li_{1}(y_1y_2)\otimes\Li_1(y_1^{-1})\bigr)\\
    ={}& p_{*}\bigl(\Li_{1}(y_2)\otimes \Li_{1}(y_1)\bigr)
      +p_{*}\bigl(\Li_{1}(y_1y_2)\otimes \Li_{1}(y_2)\bigr)
      -p_{*}\bigl(\Li_{1}(y_1y_2)\otimes\Li_1(y_1^{-1})\bigr).
    \end{align*}
Notice that in the last expression each tensor product defines a direct product decomposition of the torus~$\Tup'$, corresponding to a choice of a basis for its character lattice: $(y_2,y_1)$, $(y_1y_2,y_2)$, or $(y_1y_2,y_1^{-1})$. To formulate a general statement about the coproduct of elements in $\LL_n( \Tup_d)$, we need to discuss a certain extension of the pushforward map $p_*$. 

For any isogeny $p\colon \Tup'\to \Tup$ the pushforward maps $p_{*}$ extend to maps $p_{*}\colon\Omega^{m}\Hc_{\Tup'}\to\Omega^{m}\Hc(\Tup)$ for all $m\geq1$. Since the space $\varinjlim(\Omega^m\Hc_{\Tup''})^{\Gamma_{\Tup'',\Tup'}}$  (the filtered colimit taken over $\Cc_{\Tup'}^{\op}$) also carries a $\bigl(\Gamma_{\Tup',\Tup}\bigr)$-action, pushforward extends to a map 
\[
p_{*}\colon \varinjlim\left((\Omega^m\Hc_{\Tup''})^{\Gamma_{\Tup'',\Tup'}} \right) \lra \Omega^{m}\Hc(\Tup).
\]
This map induces a map on the quotient 
\[
p_{*}\colon \varinjlim\left((\Omega^m\overline{\Hc}_{\Tup''})^{\Gamma_{\Tup'',\Tup'}} \right) \lra \Omega^{m}\overline{\Hc}(\Tup).
\]

\begin{lemma}  \label{LemmaCoproductDecomposition}
For $d\geq 1, k\geq 1, n\geq 1$
\begin{equation}
\overline{\Delta}' \Dc_k\LL_n( \Tup_d) \, \subseteq \sum_{\substack{p\colon \Tup_{d_1} \! {\times} \Tup_{d_2} \! \to \Tup_d  \\n_1+n_2 =n\\k_1+k_2 = k}}p_*\bigl(\Dc_{k_1}\LL_{n_1}(\Tup_{d_1})\otimes \Dc_{k_2}\LL_{n_2}(\Tup_{d_2})\bigr ) \, \subseteq \, \Omega^2\overline{\Hc}(\Tup_d)
\end{equation}
where the summation goes over all isogenies $p\colon \Tup_{d_1}\times \Tup_{d_2}\to \Tup_d$ and $d_1, d_2, n_1, n_2, k_1, k_2\geq 1.$
\end{lemma}

\begin{proof}
It suffices to verify the claim for the generators $\Li_{n_1,\dots,n_k}(x_1,\dots,x_k)$, and we may assume that $d=k$. 
Consider an iterated integral $\I(z_0;z_1,\dots,z_n;z_{n+1})$ with each $z_i$ either equal to $0$ or to a monomial of the form $x_1\cdots x_{j_i}$ with all $j_i$ distinct. If $z_0=z_{n+1}=0$, then the integral vanishes. If $z_0=0$, $z_{n+1}\ne0$, then $\I(z_0;z_1,\dots,z_n;z_{n+1})\in \Dc_k\LL_n(\Tup)$ where $\Tup$ is a torus whose character lattice contains $z_j/z_{n+1}$ for nonzero $z_j$, where $k$ is the number of such nonzero $z_j$, $1\leq j\leq n$. Similarly, if $z_{n+1}=0$, $z_0\ne0$, then same is true if the character lattice contains $z_{j}/z_0$ for nonzero $z_j$. In general, from the path decomposition formula
    \[\I(z_0;z_1,\dots,z_n;z_{n+1}) = 
    \sum_{i=0}^{n}\I(z_0;z_1,\dots,z_{i};0)
                  \I(0;z_{i+1},\dots,z_{n};z_{n+1}),\]
for $z_0z_{n+1}\ne 0$ we have $\I(z_0;z_1,\dots,z_n;z_{n+1})\in \sum_{i=0}^{n}\Dc_k\LL(\Tup_i)$ where $\Tup_i$ is any subtorus whose character lattice contains $z_0/z_j$, $j=1,\dots,i$ and $z_{n+1}/z_j$, $j=i+1,\dots,n$, for nonzero $z_j$. Now consider any nonzero term in
    \begin{align*}
    &\Delta(\I(z_0;z_1,\dots,z_n;z_{n+1}))=\\
&\sum_{\substack{0=i_0<i_1<\cdots \\ \cdots<i_r<i_{r+1}=n+1}}\I(z_{i_0};z_{i_1},\dots,z_{i_r};z_{i_{r+1}})\otimes \prod_{p=0}^r\I(z_{i_p};z_{i_{p}+1},\dots,z_{i_{p+1}-1};z_{i_{p+1}})
    ,\end{align*}
where $z_0=0$, $z_1=1$, and $(z_2,\dots,z_{n+1})=(\{0\}^{n_1-1},x_1,\{0\}^{n_2-1},x_1x_2,\dots,\{0\}^{n_k-1},x_1\dots x_k)$ (here $\{0\}^{m}$ is a shorthand for $0,0,\dots,0$, repeated $m$ times), as in~\eqref{LinDef}. Since the term is nonzero, for all $p$ we have either $z_{i_p}\ne 0$ or $z_{i_{p+1}}\ne 0$. Let $k_1$ be the number of nonzero $z_{i_j}$, $j=1,\dots,r$, and $k_2$ is the total number of nonzero $z_j$ with $j\ne i_l$, $l=0,\dots,r+1$ (so that $k=k_1+k_2$). After applying path decomposition formula to each expression on the right hand side we get a sum of terms from $\LL(\Tup')\otimes \LL(\Tup'')$ with $\dim(\Tup')=k_1$, $\dim(\Tup'')=k_2$.
Consider any nonzero $z_j$, $1\leq j\leq n$. If $j=i_p$, then $z_j/z_{n+1}$ is a character on $\Tup''$. If $i_p<j<i_{p+1}$ then either $z_j/z_{n+1}=z_j/z_{i_p}\cdot z_{i_p}/z_{n+1}$ or $z_j/z_{n+1}=z_j/z_{i_{p+1}}\cdot z_{i_{p+1}}/z_{n+1}$, so in either case $z_j/z_{n+1}$ is a product of characters on $\Tup'$ and $\Tup''$. Since $z_j/z_{n+1}$ span the character lattice of $\Tup_k$, $\Tup'$ and $\Tup''$ are its subtori, and $k_1+k_2=k$, we see that $\Tup_k$ decomposes as $\Tup'\times \Tup''$, proving the claim.
\end{proof}

\begin{definition} For $m\geq 1$ define
\[
\Omega^m\LL(\Tup_d)=\sum_{p\colon \Tup_{d_1}\!{\times} \dots {\times}  \Tup_{d_m} \! \to \Tup_d }p_*\bigl(\LL_+(\Tup_{d_1})\otimes\dots \otimes \LL_+(\Tup_{d_m})\bigr)  \subseteq \Omega^m\overline{\Hc}(\Tup_d)
\]
where the summation goes over all isogenies $p\colon \Tup_{d_1}\times \dots \times  \Tup_{d_m}\lra \Tup_d$ with $d_1,\dots,d_m\geq 1.$  The depth filtration on  $\LL(\Tup_d)$ induces the depth filtration on $\Omega^m\LL(\Tup_d).$ 
\end{definition}

By Lemma \ref{LemmaCoproductDecomposition}, we obtain a subcomplex $\Omega^\bullet\LL(\Tup_d)$ of the cobar complex of $\overline{\Hc}(\Tup_d):$
\[
\begin{tikzcd}
\Omega^1\LL(\Tup_d) \arrow[d, hook] \arrow[r] & \Omega^2\LL(\Tup_d) \arrow[d, hook] \arrow[r] & \cdots \arrow[r]& \Omega^d\LL(\Tup_d) \arrow[d, hook]\\
\Omega^1\overline{\Hc}(\Tup_d) \arrow[r] & \Omega^2\overline{\Hc}(\Tup_d) \arrow[r] & \cdots \arrow[r]& \Omega^d\overline{\Hc}(\Tup_d).
\end{tikzcd}
\]

\begin{remark} \label{RemarkCyclicVector} Notice that as a $\GL_d(\Q)$-module the space 
$
\gr_k^{\Dc}\Omega^m\LL(\Tup_d)
$
is generated by elements
\[
\Li_{n_1^{(1)},\dots,n_{k_1}^{(1)}}(x_1^{(1)},\dots,x_{k_1}^{(1)})\otimes  \Li_{n_1^{(2)},\dots,n_{k_2}^{(2)}}(x_{k_1}^{(2)},\dots,x_{k_2}^{(2)}) \otimes \dots \otimes \Li_{n_{1}^{(m)},\dots,n_{k_m}^{(m)}}(x_{1}^{(m)},\dots,x_{k_m}^{(m)}) 
\]
for all decompositions $k=k_1+\dots+k_m$ and $n_i^{(j)}\geq 1$.
\end{remark}

We will also need the following lemma.

\begin{lemma} \label{LemmaProductDecomposition}
For $d\geq 1, k\geq 1, n\geq 1$ we have
\begin{equation}
\Dc_k\LL_n( \Tup_d)\cap \bigl(\overline{\Hc}_+(\Tup_d)\bigr)^2 = \sum_{\substack{p\colon \Tup_{d_1}\!{\times} \Tup_{d_2}\!\to \Tup_d  \\n_1+n_2=n\\k_1+k_2= k}}p_*\bigl(\Dc_{k_1}\LL_{n_1}(\Tup_{d_1})\cdot \Dc_{k_2}\LL_{n_2}(\Tup_{d_2})\bigr )
\end{equation}
where the summation goes over all isogenies $p\colon \Tup_{d_1}\times \Tup_{d_2}\to \Tup_d$ and $d_1, d_2, n_1, n_2, k_1, k_2\geq 1.$
\end{lemma}
\begin{proof}
The (RHS) is is generated by elements 
\begin{equation} \label{FormulaProductPolylogs}
\Li_{n_1,\dots,n_{k_1}}(x_1,\dots,x_{k_1})\cdot \Li_{n_{k_1+1},\dots,n_{k_1+k_2}}(x_{k_1+1},\dots,x_{k_1+k_2})
\end{equation}
as a $\GL_d(\Q)$-module. The quasi-shuffle relation for multiple polylogarithms (see \cite[\S 4.2]{CMRR24}) implies that (\ref{FormulaProductPolylogs}) lies in $\Dc_k\LL_n( \Tup_d)$. Thus $\text{(LHS)}\supseteq \text{(RHS)}$. 

Consider the Hopf algebra $\overline{\Hc}(\Tup_d).$ Let $m\colon \overline{\Hc}(\Tup_d)\otimes \overline{\Hc}(\Tup_d)\lra \overline{\Hc}(\Tup_d)$ be the product, $S\colon \overline{\Hc}(\Tup_d)\lra \overline{\Hc}(\Tup_d)$ be the antipode. Also, let $Y\colon \overline{\Hc}(\Tup_d)\lra \overline{\Hc}(\Tup_d)$ be the grading operator sending $x\in\overline{\Hc}_n(\Tup_d)$ to $nx$. Consider the Dynkin operator $\Dee \colon  \overline{\Hc}(\Tup_d)\lra  \overline{\Hc}(\Tup_d)$ given by the formula $\Dee=m\circ (S\otimes Y) \circ \overline{\Delta}$, see \cite[\S 2.2]{CDG21}. We have $\Ker(\Dee)= \bigl(\overline{\Hc}_+(\Tup_d)\bigr)^2$.

Consider an element $x\in \text{(LHS)}.$ We have 
\[
0=\Dee(x)=m\circ (S\otimes Y) \circ\overline{\Delta}(x)=nx+m\circ (S\otimes Y) \circ\overline{\Delta}'(x).
\]
Lemma \ref{LemmaCoproductDecomposition} implies that for $x\in \Dc_k\LL_n( \Tup_d)$ we have
\[
m\circ (S\otimes Y) \circ\overline{\Delta}'(x)\in  \text{(RHS)}.
\]
So, $x=-\frac{1}{n}m\circ (S\otimes Y) \circ\overline{\Delta}'(x)\in  \text{(RHS)}$. Thus $\text{(LHS)}\subseteq \text{(RHS)}$ and the statement follows.
\end{proof}

\subsection{Truncated symbol map} \label{SectionTruncatedSymbolMap}
In this section, we define the \emph{truncated symbol map} $\STup$, which extends the one appearing in the statement of Theorem \ref{TheoremMain2}. 

We start with a useful convention. Consider a $\VB$-module $M$ and a $\Q$-vector space $V$ with a basis $e_1,\dots,e_d$. For $k\in \{0,\dots,d\}$ consider its subspace $V_k=\langle e_1,\dots, e_k \rangle\subseteq V$ and the parabolic subgroup $\Pup_{k,d-k}=\{g\in \GL_d(\Q) \setsep g(V_k)\subseteq V_k\}$. Denote by $\psi_{V/V_k}\colon \Pup_{k,d-k} \lra \Q^{\times}$ the character sending a matrix
$
\bigl(\begin{smallmatrix}
A_{11} & A_{12}\\
0 & A_{22}
\end{smallmatrix}\bigr) \in \Pup_{k,d-k}
$
to $\lvert\det(A_{22})\rvert^{-1}$; similar character appeared in Remark \ref{RemarkCharacterPsy}. The group $\Pup_{k,d-k}$ acts on the space $M(V_k)$ via the homomorphism $\Pup_{k,d-k}\lra \GL(V_k)$, so we can consider the induction of $M(V_k)\otimes \psi_{V/V_k}$ from the parabolic subgroup $\Pup_{k,d-k}$ to  $\GL_d(\Q)$. We will use the following notation: 
\[
\bigoplus_{\substack{W\subseteq V \\ \dim(W)=k}}M(W)\coloneqq  \Ind_{\Pup_{k,d-k}}^{\GL_d}\bigl( M(V_k)\otimes \psi_{V/V_k}\bigr).
\]

In \S \ref{SectionKoszulitySteinberg}, we introduced a $\VB$-module $\StH$; this module is Koszul by Theorem \ref{TheoremSteinbergKoszulity}. The $\VB$-module ${\StH}\otimes \Sbb$ sending $V$ to $\StH(V)\otimes \Sbb^{\bullet}(V)$  is Koszul by Lemma \ref{LemmaTensorS}. In particular, for a $\Q$-vector space $V$ of dimension $d$ we have an exact complex
\[
0\lra \StH(V)\otimes \Sbb^{\bullet}(V) \stackrel{s}{\lra} \Bup_{d}\bigl( {\St} \otimes \Sbb\bigr)(V) \lra \Bup_{d-1}\bigl( {\St} \otimes \Sbb\bigr)(V)\lra \cdots .
\]
This implies exactness of the complex 
\[
0\lra\bigoplus_{\substack{W\subseteq V \\ \dim(W)=k}} \StH(W)\otimes \Sbb^{\bullet}(W) \stackrel{s}{\lra} \bigoplus_{\substack{W\subseteq V \\ \dim(W)=k}}\Bup_{k}\bigl( {\St} \otimes \Sbb\bigr)(W) \lra \bigoplus_{\substack{W\subseteq V \\ \dim(W)=k}}\Bup_{k-1}\bigl( {\St} \otimes \Sbb\bigr)(W)\lra \cdots .
\]

\begin{proposition}\label{PropositionTruncatedSymbolMap} For $k\in \{0,\dots,d\}$ there exists a unique $\GL_d(\Q)$-equivariant map 
\begin{align}\label{FormulaTruncatedSymbolMap}
\STup\colon \gr_{k}^\Dc \bigl( \LL(\Tup_d) \bigr)\lra  \bigoplus_{\substack{W\subseteq V \\ \dim(W)=k}}\StH(W)\otimes \Sbb^{\bullet}(W).
\end{align}
such that 
\begin{equation}\label{FormulaTruncatedCoproductLi}
\STup\Bigl(\Li_{n_1,\dots, n_k}(x_1,\dots,x_k)\Bigr)=\Lup[e_1,\dots,e_k]\otimes \prod_{i=1}^k \frac{e_i^{n_i-1}}{(n_i-1)!}.
\end{equation}
Moreover, the map $\STup$ is surjective. We call this map the \emph{truncated symbol map}.
\end{proposition}

Here are the key steps of the construction of the map $\STup$. Consider the iterated coproduct map
\[
\overline{\Delta}^{[k-1]}\colon \overline{\Hc}_+(\Tup_d) \lra \bigl(\overline{\Hc}_+(\Tup_d)\bigr)^{\otimes k}.
\]
Lemma~\ref{LemmaCoproductDecomposition} implies that the map $\overline{\Delta}^{[k-1]}$ vanishes on $\Dc_{k-1}\LL(\Tup_d)$ and lands in $\bigl(\Dc_1\overline{\Hc}_+(\Tup_d)\bigr)^{\otimes k}.$ Notice that the latter space is isomorphic to  $\bigl(\gr_{1}^\Dc \overline{\Hc}(\Tup_d)\bigr)^{\otimes k}$. In \S \ref{SectionALPHA}, we construct a map
\begin{equation}\label{FormulaMapAlpha}
\alpha\colon \bigoplus_{\substack{W\subseteq V \\ \dim(W)=k}} \Bup_{k}\bigl( {\St} \otimes \Sbb\bigr)(W) \longrightarrow  \bigl(\gr_{1}^\Dc \overline{\Hc}(\Tup_d)\bigr)^{\otimes k}
\end{equation}
and show its injectivity. It appears that the image of the map $\overline{\Delta}^{[k-1]}$ is contained in the image of $\alpha$.
Furthermore, the image of the map $\overline{\Delta}^{[k-1]}$ can be computed; in \S \ref{SectionProofPropositionTSM} we show that it coincides with the subspace
\begin{equation}\label{FormulaSubspaceStH}
\bigoplus_{\substack{W\subseteq V \\ \dim(W)=k}}\StH(W)\otimes \Sbb^{\bullet}(W) \subseteq \bigoplus_{\substack{W\subseteq V \\ \dim(W)=k}} \Bup_{k}\bigl( {\St} \otimes \Sbb\bigr)(W).
\end{equation}
From here we deduce both the existence of the map $\STup$ and its surjectivity.

\subsection{The embedding \texorpdfstring{$\alpha$}{alpha}} \label{SectionALPHA}
Our goal is to define a $\GL_d(\Q)$-equivariant map  (\ref{FormulaMapAlpha})
and prove its injectivity. Since 
\begin{equation} \label{FormulaInducedRep}
\bigoplus_{\substack{W\subseteq V \\ \dim(W)=k}} \Bup_{k}\bigl( {\St} \otimes \Sbb\bigr)(W)\coloneqq  \Ind_{\Pup_{k,d-k}}^{\GL_d}\bigl(\Bup_{k}\bigl( {\St} \otimes \Sbb\bigr)(V_k)\otimes \psi_{V/V_k}\bigr)
\end{equation}
is an induced representation, we can use the adjunction between induction and restriction functors. So, it is sufficient to construct a map
\[
\alpha_{V_k} \colon \Bup_{k}\bigl( {\St} \otimes \Sbb\bigr)(V_k)\otimes \psi_{V/V_k} \lra \bigl(\gr_{1}^\Dc \overline{\Hc}(\Tup_d)\bigr)^{\otimes k}
\]
such that $\alpha_{V_k}$  commutes with the action of the parabolic group $\Pup_{k,d-k}$. We define the map  $\alpha_{V_k}$ by an explicit formula:
\[
\alpha_{V_k}\left([w_1|\cdots|w_k]\otimes \frac{w_1^{n_1-1}\cdots w_k^{n_k-1}}{(n_1-1)!\cdots (n_k-1)!} \otimes 1\right)
= A\cdot \bigl(\Li_{n_1}(x_1)\otimes\dots\otimes \Li_{n_k}(x_k)\bigr),
\]
where $A=[w_1,\dots,w_k,e_{k+1},\dots,e_d]\in \GL_d(\Q)$. It is easy to see that this map  commutes with the action of the parabolic group $\Pup_{k,d-k}$, and so defines the map $\alpha$ via the adjunction. 

For convenience, we give an explicit description of the map $\alpha$. Let $w_1,\dots,w_k$ be a basis of $W$ consisting of integral vectors. We have
\begin{equation} \label{AlphakExplicit}
\begin{split}
\alpha\biggl([w_1|\cdots|w_k]\otimes {} & \frac{w_1^{n_1-1}\cdots w_k^{n_k-1}}{(n_1-1)!\cdots (n_k-1)!}\biggr) = \\ & N^{n-d-1}\sum_{\substack{y_1^{N}=x_1 \\ \hspace{\widthof{$y_1^{N}$}} \vdotsB \hspace{\widthof{$x_d$}} \\ y_d^{N}=x_d}} \Li_{n_1}\Bigl(\prod_{i=1}^d y_i^{(w_{1})_i}\Bigr)\otimes \dots \otimes \Li_{n_k}\Bigl(\prod_{i=1}^d y_i^{(w_k)_{i}}\Bigr),
\end{split}
\end{equation}
where $N$ is the index of the lattice $\Lambda\coloneqq \langle w_1,\dots,w_k\rangle_{\Z}$ in $W\cap\Z^d$. Indeed, by construction of $\alpha_k$,~\eqref{AlphakExplicit} holds when $w_i=e_i$. Since $\Z^d/W\cap \Z^d$ is torsion-free, and $\Lambda$ has index $N$ in $W\cap\Z^d$, we can find a matrix $A\in \Mup_d(\Z)$ such that $\lvert\det(A)\rvert=N$, and $A(e_i)=w_i$. Since $\alpha_k$ is $\GL_d(\Q)$-equivariant, we get~\eqref{AlphakExplicit} by acting with $A$.

\begin{lemma}\label{LemmaInjectivityAlpha}The map $\alpha$ is injective. 
\end{lemma}
\begin{proof}
To prove that $\alpha$ is injective, we construct its section using the presentation for $\gr_{1}^\Dc\overline{\Hc}(\Tup_d)$ proven in Lemma \ref{DepthOneOnColimit}. 
We need to define a map
\begin{equation}\label{FormulaMapSigma}
\sigma_{n_1,\dots, n_k}\colon \gr_{1}^\Dc\overline{\Hc}_{n_1}(\Tup_d) \otimes \dots \otimes \gr_{1}^\Dc\overline{\Hc}_{n_k}(\Tup_d) \lra \bigoplus_{\substack{W\subseteq V \\ \dim(W)=k}} \!\!  \Bup_{k}\bigl({\St} \otimes \Sbb\bigr)(W)
\end{equation}
for any decomposition $n=n_1+\dots+n_k,$ $n_i\geq 1$
such that $\sigma \circ \alpha = \Id$, where $\sigma = \sum\sigma_{n_1,\dots, n_k}$.

First, we define a map
\begin{equation*}
\Sigma_{n_1,\dots,n_k}\colon \underbrace{\Q\bigl[\mu_{\infty}\times \overset{\circB}{\Xup}(\Tup_d)_{\Q}\bigr]\otimes \dots \otimes\Q\big[\mu_{\infty}\times \overset{\circB}{\Xup}(\Tup_d)_{\Q}\bigr]}_k  \lra \bigoplus_{\substack{W\subseteq V \\ \dim(W)=k}} \!\!  \Bup_{k}\bigl( {\St} \otimes \Sbb\bigr)(W)
\end{equation*}
by the formula
\[
\Sigma_{n_1,\dots,n_k}\Bigl( [\zeta_1 ,v_1]\otimes \dots\otimes [\zeta_k, v_k]\Bigr) = |\omega_k(v_1,\dots,v_k)| \cdot [v_1|\cdots|v_k]\otimes \frac{v_1^{n_1-1}\cdots v_k^{n_k-1}}{(n_1-1)!\cdots (n_k-1)!},
\]
where $\omega_k$ is any of the two volume forms on the subspace $W=\langle v_1,\dots,v_k\rangle$ for which the lattice $\Z^d\cap W$ has covolume $1$. For any $i\in\{1,\dots,k\}$ and $m\in \Z\sm\{0\}$ we have 
\[
\Sigma_{n_1,\dots,n_k}\circ \bigl( \underbrace{\Id\otimes \dots \otimes\Id}_{i-1} \otimes\, d_{n,m} \otimes \underbrace{\Id\otimes \dots\otimes\Id}_{k-i}\bigr)=0,
\]
because
\begin{align*}
&\Sigma_{n_1,\dots,n_k}\Bigl( [\zeta_1 ,v_1]\otimes \dots \otimes [\zeta_i^m ,mv_i]\otimes \dots \otimes [\zeta_k, v_k]\Bigr)\\
={}&|\omega_k(v_1,\dots,m v_i,\dots,v_k)| \cdot [v_1|\cdots|mv_i|\cdots|v_k] \otimes \frac{v_1^{n_1-1} \cdots (mv_i)^{n_i-1} \cdots v_k^{n_k-1}}{(n_1-1)!\cdots(n_i-1)!\cdots (n_k-1)!}\\
={}&|m|\cdot m^{n_i-1} \cdot |\omega_k(v_1,\dots,v_k)| \cdot [v_1|\cdots|v_k]\otimes \frac{v_1^{n_1-1}\cdots v_k^{n_k-1}}{(n_1-1)!\cdots (n_k-1)!} \\
={}&m^{n_i-1}\sum_{\nu^m=1}\Sigma_{n_1,\dots,n_k}\Bigl( [\zeta_1 ,v_1]\otimes \dots \otimes [\nu\zeta_i,v_i]\otimes \dots  \otimes [\zeta_k, v_k]\Bigr).
\end{align*}
Lemma \ref{DepthOneOnColimit} implies that the map $\Sigma_{n_1,\dots,n_k}$ induces the map (\ref{FormulaMapSigma}). The identity $\sigma\circ\alpha=\Id$ follows from the explicit formula~\eqref{AlphakExplicit}, upon noting that by homogeneity it suffices to verify it for integral vectors $w_i$, that $N=|\omega_k(w_1,\dots,w_k)|$ is the index of $\langle w_1,\dots,w_k\rangle_{\Z}$ in $\langle w_1,\dots,w_k\rangle_{\Q}\cap\Z^d$, and that $N^{1-k}=|\omega_k(w_1/N,\dots,w_k/N)|$.
\end{proof}

\subsection{Proof of Proposition \ref{PropositionTruncatedSymbolMap}} \label{SectionProofPropositionTSM}
The case $k=0$ is trivial by Remark~\ref{RemarkCharacterPsy}, so we will assume that $k\geq1$ for the rest of the proof. 
The space $\gr_{k}^\Dc \bigl( \LL_n(\Tup_d) \bigr)$ is generated by $\Li_{n_1,\dots, n_k}(x_1,\dots,x_k)$ as a $\GL_d(\Q)$-module, so uniqueness is obvious. It remains to prove the existence and surjectivity. In the previous section we constructed an embedding $\alpha$:
\[
\begin{tikzcd}[row sep=large]
\gr_{k}^\Dc \LL(\Tup) \arrow[r, "\overline{\Delta}^{[k-1]}"] 
& \bigl(\gr_{1}^\Dc\overline{\Hc}(\Tup_d)\bigr)^{\otimes k}\\
&  \mathllap{ \displaystyle\bigoplus_{{\substack{W\subseteq V \\ \dim(W)=k}}}} \!\! \!\! \Bup_{k}\bigl( {\St} \otimes \Sbb\bigr)(W) . \arrow[hook, u, "\alpha"]
\end{tikzcd}
\]
The embedding of the space $\StH(W)\otimes \Sbb^{\bullet}(W)$ into $\Bup_{k}\bigl({\St}\otimes \Sbb\bigr)(W)$ induces an embedding (\ref{FormulaSubspaceStH}). The following lemma shows that the image of the map $\overline{\Delta}^{[k-1]}$ is contained in the image of the map $\alpha \circ s$.

\begin{lemma}\label{LemmaSymbol} 
For $n_1,\dots,n_k\geq 1$ the following equality holds:
\[
\overline{\Delta}^{[k-1]}\bigl(\Li_{n_1,\dots, n_k}(x_1,\dots,x_k)\bigr)=(\alpha \circ s)\left(\Lup[e_1,\dots,e_k]\otimes \prod_{i=1}^k \frac{e_i^{n_i-1}}{(n_i-1)!}\right).
\] 
\end{lemma}
\begin{proof}
Taking generating series on both sides, we see that it suffices to show
\[
\overline{\Delta}^{[k-1]}\bigl(\Li(x_1,\dots,x_k;t_1,\dots,t_k)\bigr)=(\alpha\circ s) \left(\Lup[e_1,\dots,e_k]\otimes \exp(e_1t_1+\dots+e_kt_k)\right).
\]
Arguing by induction in $k$ (the case $k=1$ being trivial), it is enough to check that the $(k-1,1)$-part of the coproduct, computed in Lemma~\ref{LemmaCoproductGeneratingFunction}, agrees with the formula of Lemma~\ref{LemmaLCoproductComponent}. 
It is easy to see that the three types of terms that appear in Lemma~\ref{LemmaCoproductGeneratingFunction}, namely
\begin{align*}
	&\Li(x_2,\dots,x_k;t_2,\dots,t_k)\otimes \Li(x_1;t_1),\\ 
	&\Li(x_1,\dots,x_{i}x_{i+1},\dots,x_k;t_1,\dots,\widehat{t_{i+1}},\dots,t_k) \otimes \Li(x_{i+1};t_{i+1}-t_{i}),\quad \mbox{and}\\
	&\Li(x_1,\dots,x_{i}x_{i+1},\dots,x_k;t_1,\dots,\widehat{t_i},\dots,t_k) \otimes \Li(x_{i};t_i-t_{i+1})
\end{align*}
correspond to $\Lup[e_2,\dots,e_k]\otimes \Lup[e_1]$, $\Lup[e_1,\dots,e_{i-1},e_i+e_{i+1},e_{i+2},\dots, e_k]\otimes \Lup[e_{i+1}]$, and $\Lup[e_1,\dots,e_{i-1},e_i+e_{i+1},e_{i+2},\dots, e_k]\otimes \Lup[e_i]$, tensored with $\exp(e_1t_1+\dots+e_kt_k)$.
Indeed, for the first case this is immediate, and for the last two it follows from the identity
    \begin{align*}
    \exp(e_it_i+e_{i+1}t_{i+1}) 
    &= \exp\bigl((e_i+e_{i+1})t_i\bigr)\,\exp\bigl(e_{i+1}(t_{i+1}-t_i)\bigr)\\
    &= \exp\bigl((e_i+e_{i+1})t_{i+1}\bigr)\,\exp\bigl(e_{i}(t_{i}-t_{i+1})\bigr).
    \end{align*}
Since these are exactly the terms appearing in Lemma~\ref{LemmaLCoproductComponent} (with the same signs), this concludes the proof.
\end{proof}

We define the truncated symbol map by the formula 
$\STup= s^{-1}\circ \alpha^{-1}\circ \overline{\Delta}^{[k-1]}$. Lemma \ref{LemmaInjectivityAlpha} and Proposition \ref{TheoremGenericPairsGenerate} imply that the map $\STup$ is surjective. This finishes the proof of Proposition \ref{PropositionTruncatedSymbolMap}.\hfill\qedsymbol

\subsection{Extension of the truncated symbol map to cobar complexes} \label{SectionTSMComplexes}
The goal of this section is to extend the truncated symbol map $\STup$ to a $\GL_d(\Q)$-equivariant map of complexes 
\[
\STup^\bullet\colon \gr_k^{\Dc}\Omega^\bullet\LL(\Tup_d) \lra  \displaystyle\bigoplus_{{\substack{W\subseteq V \\ \dim(W)=k}}} \Omega^{\bullet}\bigl( {\StH} \otimes \Sbb\bigr)(W).
\]

Fix $m\in\{1,\dots,k\}$. The space $\Dc_k\Omega^m\LL(\Tup_d)$ is a subspace of $\Dc_k\left (\overline{\Hc}_+(\Tup_d)^{\otimes m}\right)$ so we get a $\GL_d(\Q)$-equivariant map 
\begin{equation}\label{FormulaEmbeddingLintoH}
\gr_k^{\Dc}\Omega^m\LL(\Tup_d) \lra \gr_k^{\Dc}\left (\overline{\Hc}(\Tup_d)^{\otimes m}\right)=\bigoplus_{k_1+\dots+k_m=k}  \gr_{k_1}^{\Dc}\overline{\Hc}(\Tup_d) \otimes \dots \otimes  \gr_{k_m}^{\Dc}\overline{\Hc}(\Tup_d).
\end{equation}
Combining it with the maps $\overline{\Delta}^{[k_1-1]}\otimes \dots \otimes  \overline{\Delta}^{[k_m-1]},$ we obtain a  $\GL_d(\Q)$-equivariant map 
\begin{equation} \label{FormulaTruncatedCobracketMap}
\gr_k^{\Dc}\Omega^m\LL(\Tup_d) \lra \bigoplus_{k=k_1+\dots +k_m} \bigl(\gr_{1}^\Dc\overline{\Hc}(\Tup_d)\bigr)^{\otimes k_1} \otimes \dots \otimes \bigl(\gr_{1}^\Dc\overline{\Hc}(\Tup_d)\bigr)^{\otimes k_m}.
\end{equation}
Consider a subspace $W\subseteq V$ of dimension $k$ and a decomposition $W=W_1\oplus\dots \oplus W_m$ into nonzero subspaces.
By Lemma \ref{LemmaInjectivityAlpha},  we have embeddings
\[
\bigoplus_{\dim(W_i)=k_i}\bigl({\StH}\otimes \Sbb\bigr)(W_i) \hookrightarrow \bigl(\gr_{1}^\Dc\overline{\Hc}(\Tup_d)\bigr)^{\otimes k_i} \quad \text{for}\quad 1\leq i\leq m.
\]
Adding those up, we obtain an embedding $\alpha_m$ of the space
\[
\bigoplus_{{\substack{W\subseteq V \\ \dim(W)=k}}} \Omega^{m}\bigl( {\St} \otimes \Sbb\bigr)(W)=
\bigoplus_{\substack{W=W_1\oplus \dots \oplus W_m\subseteq V\\\dim(W)=k,\,W_i\neq 0}}\left (\left ({\StH}\otimes \Sbb\right ) (W_1) \otimes \dots \otimes   \left ({\StH}\otimes \Sbb\right )(W_m) \right )
\]
into
\[
\bigoplus_{k_1+\dots +k_m=k} \bigl(\gr_{1}^\Dc\overline{\Hc}(\Tup_d)\bigr)^{\otimes k_1} \otimes \dots \otimes \bigl(\gr_{1}^\Dc\overline{\Hc}(\Tup_d)\bigr)^{\otimes k_m}.
\]

We have constructed a diagram
\[
\begin{tikzcd}[row sep=large]
\gr_k^{\Dc}\Omega^m\LL(\Tup_d) \arrow[r] 
& \displaystyle\bigoplus_{k_1+\dots +k_m=k} \bigl(\gr_{1}^\Dc\overline{\Hc}(\Tup_d)\bigr)^{\otimes k_1} \otimes \dots \otimes \bigl(\gr_{1}^\Dc\overline{\Hc}(\Tup_d)\bigr)^{\otimes k_m}.\\
& \displaystyle\bigoplus_{\dim(W)=k} \Omega^{m}\bigl( {\St} \otimes \Sbb\bigr)(W) \arrow[hook, u, "\alpha_m", shorten >=-1ex]
\end{tikzcd}
\]
Remark \ref{RemarkCyclicVector} and Lemma \ref{LemmaSymbol} imply that the image of the map (\ref{FormulaTruncatedCobracketMap})
is contained in the space
$
\alpha_m\left(\bigoplus_{\dim(W)=k} \Omega^m\bigl( {\St} \otimes \Sbb\bigr)(W)\right ).
$ 
Thus we may put
\[
\STup^m=(\alpha_m)^{-1}\left(\sum_{k=k_1+\dots+k_m} \overline{\Delta}^{[k_1-1]}\otimes \dots \otimes  \overline{\Delta}^{[k_m-1]}\right).
\]

\begin{lemma}\label{LemmaTruncatedSymbolOnComplexes} The maps $\STup^m$  combine into a map of complexes 
\[
\STup^\bullet\colon \gr_{k}^\Dc \bigl( \Omega^\bullet\LL(\Tup_d) \bigr)\lra  \bigoplus_{\substack{W\subseteq V \\ \dim(W)=k}} \Omega^{\bullet}\left({\StH}\otimes \Sbb \right)(W).
\]
Moreover, all maps $\STup^m$ are surjective.
\end{lemma} 
\begin{proof}
First we show that the map $\STup^\bullet$ is a map of complexes. The map (\ref{FormulaEmbeddingLintoH}) commutes with the differential by  construction, see \S \ref{SectionCobarComplexLTd}. To see that the map (\ref{FormulaTruncatedCobracketMap}) does, notice that this map is a composition of  the map (\ref{FormulaEmbeddingLintoH}) with the symbol map for the Hopf algebra $\gr_{\bullet}^{\Dc}\overline{\Hc}(\Tup_d)$; the statement follows from the fact that the symbol is a morphism of Hopf algebras, see \S \ref{SectionRecapOfHopfAlgebras}. The map $\alpha$ commutes with the deconcatenation coproduct by (\ref{AlphakExplicit}), so the map
$\alpha_m$ is a map of complexes. From here the statement follows. 

The surjectivity of the map $\STup^m$ is derived from Lemma \ref{LemmaSymbol} and Proposition \ref{TheoremGenericPairsGenerate}.
\end{proof}

Denote by $\LL^{\Lc}(\Tup_d)$ the image of the space $\LL_+(\Tup_d)\subseteq \overline{\Hc}_+(\Tup_d)$  under the projection to the Lie coalgebra of indecomposables  $\overline{\Lc}(\Tup_d)$. By Lemma \ref{LemmaProductDecomposition}, the composition 
\[
 \gr_k^{\Dc}\LL(\Tup_d) \stackrel{\STup}{\lra} \bigoplus_{\substack{W\subseteq V \\ \dim(W)=k}}\bigl( {\StH} \otimes \Sbb\bigr)(W)\lra \bigoplus_{\substack{W\subseteq V \\ \dim(W)=k}}\bigl( {\StL} \otimes \Sbb\bigr)(W)
\]
vanishes on the elements in $\LL(\Tup_d)\cap (\overline{\Hc}_+(\Tup_d))^2$. Thus, we have a Lie coalgebra version of the truncated symbol map 
\[
\STup^{\bullet}\colon  \gr_k^{\Dc} \left(\Omega^{\bullet}\LL^{\Lc}(\Tup_d)\right)\lra \bigoplus_{\substack{W\subseteq V \\ \dim(W)=k}} \Omega^{\bullet}\bigl({\StL} \otimes \Sbb\bigr)(W).
\]

\section{Proofs of the main results}
\subsection{Proof of Theorem \ref{TheoremMain2}}
\label{SectionProofoftheoremMain}
We will denote the map
\[
\STup^m\colon \gr_{k}^\Dc \bigl( \Omega^m\LL(\Tup_d) \bigr)\lra  \bigoplus_{\substack{W\subseteq V \\ \dim(W)=k}} \Omega^m\left({\StH}\otimes \Sbb \right)(W).
\]
from \S \ref{SectionTSMComplexes} by $\STup^m_{d,k}$. Its $n$-th graded component will be denoted by $\STup^m_{d,k,n}$. 

We need to show that the map $\STup^1_{d,d,n}$ is an isomorphism for every $d$ and $n$. We prove the statement by double induction: first on $d$ and then on $n$. The case $d=0$ is trivial. The case $d=1$ follows from Example \ref{exampleDepth1dim1}. Next, assume that $d=2$. We need to show that the map
\[
\STup^1\colon \gr_2^{\Dc} \LL_n(\Tup_2)\lra \StH(V)\otimes \Sbb^{n-2}(V)
\]
is an isomorphism. This map is surjective by Proposition \ref{PropositionTruncatedSymbolMap}, so it remains to prove that it is injective. Let $a\in \Dc_2\LL_n(\Tup_2)$ lie in $\Ker(\STup)$. By the construction of the truncated symbol map, $\overline{\Delta}^{[1]}(a)=\alpha(s(\STup(a)))=0,$ thus $a$ is primitive  in $\overline{\Hc}_n(\Tup_2)$. By Proposition \ref{PropositionWeakDepthConjecture}, $a$ lies in $\Dc_1\overline{\Hc}_n(\Tup_2)$. By Remark \ref{RemarkDepthFiltration}, $a$ lies in $\LL_n(\Tup_2)\cap \Dc_1\overline{\Hc}_n(\Tup_2)=\Dc_1\LL_n(\Tup_2)$, so vanishes in  $\gr_2^{\Dc} \LL_n(\Tup_2)$.

Assume that $d\geq 3$. We need to show that the map $\STup^1_{d,d}$ is injective.

\begin{lemma}\label{LemmaConsequenceOfInductionHyothesis} Assume that the map $\STup^1_{d',d'}$ is an isomorphism for $d'<d$. Then the maps
\[
\STup_{d,k}^m \colon \gr_k^{\Dc} \Omega^m \LL(\Tup_d)\lra \bigoplus_{\substack{W\subseteq V \\ \dim(W)=k}} \Omega^m\bigl( {\StH} \otimes \Sbb\bigr)(W)
\]
are isomorphisms in the following two cases: 
\begin{enumerate*} \item\label{LemmaConsequenceOfInductionHyothesis:part1} $m=1, k<d$,
\item\label{LemmaConsequenceOfInductionHyothesis:part2} $m\geq 2, k\leq d$.
\end{enumerate*}
\end{lemma}
\begin{proof}
We prove \ref{LemmaConsequenceOfInductionHyothesis:part1}; the proof of \ref{LemmaConsequenceOfInductionHyothesis:part2} is similar. The map $\STup_{d,k}^1$  is surjective by Proposition \ref{PropositionTruncatedSymbolMap}. To prove injectivity, we construct a map $\Psi_{d,k}^1$ in the opposite direction and show that $\Psi_{d,k}^1\circ \STup_{d,k}^1=\Id$. For that, consider a map $\pi_k\colon \Tup_d\lra \Tup_k$, such that $\pi_k(x_1,\dots,x_d)=(x_1,\dots,x_k)$. The map
 \[
\gr_k^{\Dc} \LL(\Tup_k)\otimes \psi_{V/V_k}\lra \gr_k^{\Dc}\LL(\Tup_d)
 \]
sending $a\otimes 1$ to $(\pi_k)^*(a)$ is $(\Pup_{k,d-k})$-equivariant. By the induction hypothesis, we obtain a $(\Pup_{k,d-k})$-equivariant map
\begin{equation}\label{FormulaMapRestriction}
\bigl( {\StH} \otimes \Sbb\bigr)(V_k)\otimes \psi_{V/V_k} \lra \gr_k^{\Dc}\LL(\Tup_d)
\end{equation}
sending $\Lup[e_1,\dots,e_k]\otimes \prod_{i=1}^k \frac{e_i^{n_i-1}}{(n_i-1)!}\otimes 1$ to $\Li_{n_1,\dots, n_k}(x_1,\dots,x_k)$. Let $\Psi_{d,k}^1$ be the map corresponding to (\ref{FormulaMapRestriction}) via the adjunction between induction and restriction. The composition $\Psi_{d,k}^1\circ \STup_{d,k}^1$ is a $\GL_d(\Q)$-equivariant map from  $\gr_k^{\Dc}  \LL(\Tup_d)$ to itself which fixes $\Li_{n_1,\dots, n_k}(x_1,\dots,x_k)$. Such a map has to be trivial, so $\Psi_{d,k}^1\circ \STup_{d,k}^1=\Id$. This finishes the proof of the statement.
\end{proof}

The following Lemma plays the main role; we prove it in \S \ref{SectionLemmaMain}.

\begin{lemma} \label{LemmaMain} Assume that the map 
\[
\STup_{d',d',n'}^1\colon \gr_{d'}^\Dc \LL_{n'}(\Tup_{d'}) \lra  \StH(W)\otimes \Sbb^{n'-d'}W
\]
for $W=\Xup(\Tup_{d'})_{\Q}$
is an isomorphism for any $d',n'$ such that either $d'<d$ or   $d'=d, n'<n$.
Then for $a\in \LL_n(\Tup_d)$ such that $\overline{\Delta}^{[2]}(a)=0$ we have $a\in \Dc_2\LL_n(\Tup_d).$
\end{lemma}

Now we are ready to finish the proof of Theorem \ref{TheoremMain2}. Assume that the map
\[
\STup^1_{d,d}\colon \gr_{d}^\Dc  \LL_n(\Tup_d) \lra  \StH(V)\otimes \Sbb^{n-d}(V).
\]
is not injective. Consider the set $S$ of elements $a\in \LL_n(\Tup_d)$ such that $a\not \in \Dc_{d-1}$ and $\STup_{d,d}^1(a)=0$. By our assumption, $S \neq \varnothing$. Choose an element $a\in S$ such that $\overline{\Delta}^{[1]}(a)$ is contained in $\Dc_k\Omega^{2}\LL(\Tup_d)$ for the smallest possible $k\in \N$. 

If $k=1$ then $\overline{\Delta}^{[1]}(a) \in \Dc_1\Omega^{2}\LL(\Tup_d)=0$ and so $a\in  \Dc_1 \overline{\Hc}_n(\Tup_d)$ by Proposition \ref{PropositionWeakDepthConjecture}. By Remark \ref{RemarkDepthFiltration}, $a$ lies in $\LL_n(\Tup_d)\cap \Dc_1\overline{\Hc}_n(\Tup_d)=\Dc_1\LL_n(\Tup_d)$. Since 
$\Dc_1 \LL_n(\Tup_d) \subseteq \Dc_{d-1} \LL_n(\Tup_d)$, we get that $a$ lies in  $\Dc_{d-1} \LL_n(\Tup_d)$, which  contradicts  the choice of $a$.

If $k=2$ then $\overline{\Delta}^{[2]}(a)\in \Dc_2 \Omega^{3}\LL(\Tup_d)=0$, so $a\in \Dc_2 \LL_n(\Tup_d)$ by Lemma \ref{LemmaMain}. Since $d\geq 3$, $a$ lies in $\Dc_{d-1} \LL_n(\Tup_d)$, 
which contradicts our choice of $a$. 

Next, we show that $k\leq d-1$.  By Lemma \ref{LemmaTruncatedSymbolOnComplexes}, we have $\STup_{d,d}^2(\overline{\Delta}^{[1]}(a))=d(\STup_{d,d}^1(a))=d(0)=0$. By Lemma \ref{LemmaConsequenceOfInductionHyothesis},  the map
\[
\STup_{d,d}^2\colon \gr_d^{\Dc}\Omega^2\LL(\Tup_d) \lra \Omega^{2}\bigl( {\StH} \otimes \Sbb\bigr)(V)
\]
is an isomorphism, so $\overline{\Delta}^{[1]}(a)$  vanishes in $\gr_d^{\Dc}\Omega^2\LL(\Tup_d)$. Thus $\overline{\Delta}^{[1]}(a)$ lies in $\Dc_{d-1}\Omega^2\LL(\Tup_d)$ and so $k\leq d-1$.

We have shown that $3\leq k\leq  d-1$. The complex
\begin{equation}\label{FormulaComplexDL} 
0 \lra \gr_k^{\Dc}\LL(\Tup_d) \lra  \gr_k^{\Dc}\Omega^2\LL(\Tup_d) 
\lra \gr_k^{\Dc}\Omega^3 \LL(\Tup_d)
\end{equation}
is exact because it is isomorphic to an exact complex 
\[
0\lra \bigoplus_{\substack{W\subseteq V \\ \dim(W)=k}}  \!\!\! \left( {\StH} \otimes \Sbb^n\right)(W) \lra 
\bigoplus_{\substack{W\subseteq V \\ \dim(W)=k}} \!\!\! \Omega^2\left( {\StH} \otimes \Sbb^n\right)(W) \lra 
\bigoplus_{\substack{W\subseteq V \\ \dim(W)=k}}  \!\!\! \Omega^3\left( {\StH} \otimes \Sbb^n\right)(W) 
\]
by Lemma \ref{LemmaConsequenceOfInductionHyothesis}. By coassociativity, the image of $\overline{\Delta}^{[1]}(a)\in \gr_k^{\Dc}\Omega^2 \LL(\Tup_d)$ vanishes in $\gr_k^{\Dc}\Omega^3 \LL(\Tup_d)$. By exactness of (\ref{FormulaComplexDL}), we can find $\tilde{a}\in \Dc_k\LL_n(\Tup_d)$ such that $\overline{\Delta}^{[1]}(a)$ and $\overline{\Delta}^{[1]}(\tilde{a})$ coincide in $\gr_k^{\Dc}\Omega^2 \LL(\Tup_d)$. It follows that
$\overline{\Delta}^{[1]}(a-\tilde{a})=\overline{\Delta}^{[1]}(a)-\overline{\Delta}^{[1]}(\tilde{a}) \in  \Dc_{k-1}\Omega^2\LL(\Tup_d)$. Since $k\leq d-1$ we have $\tilde{a}\in \Dc_k\LL(\Tup_d)\subseteq \Dc_{d-1}\LL(\Tup_d)$, so $\STup_{d,d}^1(\tilde{a})=0$. It follows that $a-\tilde{a}$ is an element in $S$ with $\overline{\Delta}^{[1]}(a-\tilde{a})\in  \Dc_{k-1}\Omega^2\LL(\Tup_d)$. The existence of such an element contradicts our choice of $a$.

This finishes the proof of Theorem \ref{TheoremMain2}. \hfill\qedsymbol\medskip

It remains to prove Lemma \ref{LemmaMain}.

\subsection{Preparation for the proof of Lemma \ref{LemmaMain}}  
\label{sectionPreparationLemmaMain}

In \S \ref{SectionDepthFiltrationOnH} we defined the Lie coalgebra $\Lc(\Tup_d)$ with the cobracket $\delta$ and the Lie coalgebra $\overline{\Lc}(\Tup_d)$ with the cobracket $\overline{\delta}$.  Consider the projection map $p\colon \Lc(\Tup_d) \lra \overline{ \Lc}(\Tup_d)$ and denote by $\widetilde{\LL}^{\Lc}(\Tup_d)$ the preimage $p^{-1}(\LL^{\Lc}(\Tup_d))$. In weight one, we have an exact sequence
\[
0\lra \Log(\Tup_d)\lra \widetilde{\LL}^{\Lc}_1(\Tup_d)\lra \LL^{\Lc}_1(\Tup_d) \lra 0.
\] 
In weight $n\geq 2$, we have $\widetilde{\LL}^{\Lc}_n(\Tup_d)\cong \LL^{\Lc}_n(\Tup_d)$. Denote by $\Omega^m\widetilde{\LL}^{\Lc}(\Tup_d)$  the preimage of $\Omega^m\LL^{\Lc}(\Tup_d)$ under the map 
\[
\Lambda^m p\colon \Lambda^m \Lc(\Tup_d)\lra \Lambda^m \overline{\Lc}(\Tup_d).
\]
It is easy to see that 
\[
\delta(\Omega^m\LL^{\Lc}(\Tup_d))\subseteq \Omega^{m+1}\LL^{\Lc}(\Tup_d),
\] 
so we may consider the following subcomplex of the Chevalley-Eilenberg complex of  $\Lc(\Tup_d)$:
\[
0\lra \widetilde{\LL}^{\Lc}(\Tup_d) \stackrel{\delta^1}{\lra} \Omega^2\widetilde{\LL}^{\Lc}(\Tup_d) \stackrel{\delta^2}{\lra} \Omega^3\widetilde{\LL}^{\Lc}(\Tup_d) \stackrel{\delta^3}{\lra} \cdots .
\]

Lemma \ref{LemmaDirectSumDecomposition} implies that for $k\geq 1$ we have a direct sum decomposition
\[
\gr_k^{\Dc}\left(\Omega^m\widetilde{\LL}^{\Lc}(\Tup_d)\right)\cong \bigoplus_{i=0}^{m-1} \left(\gr_k^{\Dc}\bigl(\Omega^{m-i}\LL^{\Lc}(\Tup_d) \bigr)\otimes \Lambda^{i} \Log(\Tup_d) \right).
\]
Consider the bigraded vector space
\[
L^{p,q}_k(\Tup_d) \coloneqq \gr_k^{\Dc}\bigl(\Omega^{p}\LL^{\Lc}(\Tup_d)\bigr) \otimes \Lambda^{q} \Log(\Tup_d) \quad \text{for} \quad p,q\geq 0.
\]
The map $\delta^{p+q}$ induces a map
\[
\delta^{p,q}\colon L^{p,q}_k(\Tup_d)\lra \bigoplus_{i=0}^{p+q-1} L^{p+q-i,i}_k(\Tup_d).
\]
The map $\delta$ is a graded derivation and elements in $\Log(\Tup_d)$ lie in the kernel of the cobracket. From here it easily follows that the only nonzero components of the map $\delta^{p,q}$ are the maps
\[
\delta_1^{p,q}\colon L^{p,q}_k(\Tup_d)\lra L^{p+1,q}_k(\Tup_d)
\quad \text{and} \quad \delta_2^{p,q}\colon L^{p,q}_k(\Tup_d)\lra L^{p,q+1}_k(\Tup_d).
\]
So, we see that the bigraded space $L^{\bullet,\bullet}_k(\Tup_d)$ with maps $\delta_1$ and $\delta_2$ is a double complex with the total complex $\gr_k^{\Dc}\left(\Omega^\bullet\widetilde{\LL}^{\Lc}(\Tup_d)\right)$. It is easy to see that the map 
\[
\delta_1^{p,q}\colon 
\gr_k^{\Dc}\Omega^{p}\LL^{\Lc}(\Tup_d) \otimes \Lambda^{q} \Log(\Tup_d)\lra  \gr_k^{\Dc}\Omega^{p+1}\LL^{\Lc}(\Tup_d)\otimes\Lambda^{q} \Log(\Tup_d)
\]
is equal to $\overline{\delta}^{p}\otimes\Id$.

\begin{lemma}\label{LemmaSecondDifferential} The map
\[
\delta_{2}^{0,1}\colon \gr_k^{\Dc}\LL^{\Lc}(\Tup_d)\lra  \gr_k^{\Dc}\LL^{\Lc}(\Tup_d) \otimes\Log(\Tup_d)
\]
satisfies
\[
\delta_{2}^{0,1}\left(\Li_{n_1,\dots,n_k}^{\Lc}(x_1,\dots,x_k)\right)= \sum_{\substack{1\leq i\leq k\\ n_i>1}}\Li_{n_1,\dots,n_i-1,\dots,n_k}^{\Lc}(x_1,\dots,x_k)
\otimes\log(x_i).
\]
\end{lemma}
\begin{proof}
By~\eqref{LinDef}, $\Li_{n_1,\dots,n_k}^{\Lc}(x_1,\dots,x_k)=(-1)^k\I^\Lc(z_0;z_1,\dots,z_n;z_{n+1})$, where we set $z_0=0$, $z_1=1$, and $(z_2,\dots,z_{n+1})=(\{0\}^{n_1-1},x_1,\{0\}^{n_2-1},x_1x_2,\dots,\{0\}^{n_k-1},x_1\cdots x_k)$ (as before, $\{0\}^{m}$ is a shorthand for $0,\dots,0$, repeated $m$ times). Looking at the coproduct formula~\eqref{Icoproduct}, since we are working in the Lie coalgebra of indecomposables, the only nonzero terms are
    \[T_{i,j} \coloneqq  \I^\Lc(z_{0};z_{1},\dots,z_{i-1},z_{i},z_{j},z_{j+1},\dots;z_{n+1})\otimes \I^\Lc(z_{i};z_{i+1},\dots,z_{j-1};z_{j}),\quad 0\leq i<j\leq n+1.\]
Next, since the left hand tensor should be nonzero in $\gr_k^{\Dc}$, we must have $z_{i+1}=\dots=z_{j-1}=0$, and since $\I^{\Lc}(a;\{0\}^m;b)=0$ for $m>1$, we must have $j=i+2$.
Now if $T_{i,i+2}$ is as above, and $z_i=x_1\cdots x_{l-1}$, then necessarily $n_{l}>1$, and
    \[\Li_{n_1,\dots,n_{l}-1,\dots,n_k}^{\Lc}(x_1,\dots,x_k)\otimes\log(x_{l}) = 
    \begin{cases}
    T_{i,i+2}+T_{i+n_{i}-2,i+n_{i}},& n_i>2,\\
    T_{i,i+2},& n_i=2.
    \end{cases}\]
Summing up these identities gives the claim.
\end{proof}

Our next goal is to extend the truncated symbol map $\STup^{\bullet}$ to this setting.
For that,  consider a double complex
\[
S_k^{p,q}(V)= \Biggl( \bigoplus_{\substack{W\subseteq V \\ \dim(W)=k}}\Omega^{p}\bigl( {\StL}\otimes \Sbb\bigr) (W)\Biggr)\otimes \Lambda^{q} V
\cong
 \bigoplus_{\substack{W\subseteq V \\ \dim(W)=k}}\Omega^{p}( {\StL})(W)\otimes \Sbb^{\bullet-k}(W)\otimes \Lambda^{q} V;
\]
with differentials $d_1^{p,q}\colon S_k^{p,q}(V)\lra S_k^{p+1,q}(V)$ and $d_2^{p,q}\colon S_k^{p,q}(V)\lra S_k^{p,q+1}(V)$ defined as follows. 
We put  $d_1^{p,q}=d^p\otimes \Id$. Recall the map $\partial\colon \Sbb^{\bullet}W \lra\Sbb^{\bullet-1}W \otimes W$ given by the formula
\[
\partial(v_1\cdots v_n)=\sum_{i=1}^n (v_1  \cdots  \widehat{v_i}\cdots v_n)\otimes v_i.
\] 
The map ${\Id}\otimes {\partial}\colon \StL(W)\otimes \Sbb^{\bullet}W \lra \StL(W)\otimes \Sbb^{\bullet-1}W \otimes W$ induces a map 
\[
\Omega^p ({\Id}\otimes {\partial}) \colon\bigoplus_{\substack{W\subseteq V \\ \dim(W)=k}}\Omega^{p}\bigl( {\StL}\otimes \Sbb) (W) \lra \bigoplus_{\substack{W\subseteq V \\ \dim(W)=k}}\Omega^{p}\bigl({\StL}\otimes \Sbb)(W) \otimes W 
\]
Finally, the differential $d_2^{p,q}$ is the composition of the map
\[
\Omega^p ({\Id}\otimes {\partial}) \otimes {\Id}\colon \bigoplus_{\substack{W\subseteq V \\ \dim(W)=k}}\Omega^{p}\bigl( {\StL}\otimes \Sbb) (W)  \otimes \Lambda^q V
\lra \bigoplus_{\substack{W\subseteq V \\ \dim(W)=k}}\Omega^{p}\bigl({\StL}\otimes \Sbb)(W) \otimes W \otimes \Lambda^q V
\]
and the map 
\[
{\Id} \otimes {\wedge} \colon
 \bigoplus_{\substack{W\subseteq V \\ \dim(W)=k}}\Omega^{p}\bigl({\StL}\otimes \Sbb)(W) \otimes W \otimes \Lambda^q V
\lra 
\bigoplus_{\substack{W\subseteq V \\ \dim(W)=k}}\Omega^{p}\bigl({\StL}\otimes \Sbb)(W) 
\otimes \Lambda^{q+1} V
\]
One can easily check the identities $(d_1)^2=(d_2)^2=d_1d_2+d_2d_1=0,$ so
$\bigl(S^{p, q}_k(V), d_1, d_2\bigr)$ is a double complex.
  
\begin{lemma} 
The map
\[
\STup^{p,q}={\STup^{p}}\otimes {\Id} \colon L_k^{p,q}(\Tup_d)\lra S_k^{p,q}(V)
\]
is a map of double complexes. 
\end{lemma}
\begin{proof}
The statement follows easily from Lemma \ref{LemmaSecondDifferential}.
\end{proof}

\subsection{Proof of Lemma \ref{LemmaMain}} \label{SectionLemmaMain}
For $d=1, 2$ the statement is obvious, so we assume that $d\geq 3$. Consider an element $a\in \LL_n(\Tup_d)$ such that $\overline{\Delta}^{[2]}(a)=0$. 

Our first goal is to show that 
\begin{equation}\label{FormulaCoproductLiesInDepthTwo}
\overline{\Delta}^{[1]}(a)\in \Dc_2\Omega^2 \LL_n( \Tup_d).
\end{equation}
Consider the smallest $k$ such that $\overline{\Delta}^{[1]}(a)\in \Dc_k\Omega^2 \LL_n( \Tup_d)$; the image of $\overline{\Delta}^{[1]}(a)$ in $\gr_k(\Omega^2 \LL_n( \Tup_d))$ is not zero. If $k\geq 3$, we have a commutative square
\begin{equation} \label{FormulaCommutativeSquare}
\begin{tikzcd}[row sep = large, column sep = huge]
\gr_k(\Omega^2 \LL( \Tup_d))\arrow[d,"\STup"] \arrow[r,"\overline{\Delta}^{[1]}\otimes 1"] &
\gr_k(\Omega^3 \LL( \Tup_d)) \arrow[d,"\STup"]  \\
\hspace{-1em} \displaystyle\bigoplus_{\substack{W\subseteq V \\ \dim(W)=k}} \!\! \Omega^2\left({\StH}\otimes \Sbb \right)(W)\arrow[r,"\overline{\Delta}^{[1]}\otimes 1"]  &  \hspace{-1em} \displaystyle \bigoplus_{\substack{W\subseteq V \\ \dim(W)=k}} \!\! \Omega^3\left({\StH}\otimes \Sbb \right)(W).
\end{tikzcd}
\end{equation}
The vertical maps in (\ref{FormulaCommutativeSquare}) are isomorphisms by  Lemma \ref{LemmaConsequenceOfInductionHyothesis}, and the bottom horizontal map is injective by the Koszulity property of the Steinberg module. It follows that the top horizontal map is also injective.  Since $(\overline{\Delta}^{[1]} \otimes 1)\overline{\Delta}^{[1]}(a)=\overline{\Delta}^{[2]}(a)=0$, we obtain a contradiction. This proves \eqref{FormulaCoproductLiesInDepthTwo}.

Assume that the projection of $a$ to $\LL_n^{\Lc}(\Tup_d)$ lies in $\Dc_2\LL_n^{\Lc}(\Tup_d).$ Our next goal is to show that this implies the lemma. Indeed, in this case, there exists $\tilde{a}\in \Dc_2\LL_n(\Tup_d)$ such that $a-\tilde{a}$ lies in $\LL_n(\Tup_d)\cap \overline{\Hc}_+(\Tup_d)^2$ and satisfies an identity $\overline{\Delta}^{[2]}(a-\tilde{a})=0$. By \eqref{FormulaCoproductLiesInDepthTwo}, $\overline{\Delta}^{[1]}(a-\tilde{a})$ lies in  $\Dc_2\Omega^2 \LL_n( \Tup_d)$ and so $(a-\tilde{a})$ lies in 
\[
m( \Dc_2\Omega^2 \LL_n( \Tup_d) )=\sum_{\substack{p\colon \Tup_{d_1} \! {\times} \Tup_{d_2} \! \to \Tup_d  \\n_1+n_2 =n}}m \Bigl( p_*\bigl(\Dc_{1}\LL_{n_1}(\Tup_{d_1})\otimes \Dc_{1}\LL_{n_2}(\Tup_{d_2})\bigr )\Bigr) \subseteq  \Dc_2 \LL_n( \Tup_d)
\]
by the properties of the Dynkin operator $\Dee$. It follows that $a\in \Dc_2 \LL_n( \Tup_d)$ which implies the statement of the lemma.

From now on, we work in the Lie coalgebra setting. We denote the image of $a$ in $\LL_n^{\Lc}(\Tup_d)$ by the same letter; we have $\overline{\delta}^{[2]}(a)=0$. In \S \ref{sectionPreparationLemmaMain}, we introduced spaces
\[
L^{p,q}_k(\Tup_d)=\gr_k^{\Dc}\bigl(\Omega^{p}\LL^{\Lc}(\Tup_d)\bigr) \otimes \Lambda^{q} \Log(\Tup_d) 
\]
and
\[
S_k^{p,q}(V)= \bigoplus_{\substack{W\subseteq V \\ \dim(W)=k}}\Omega^{p}\bigl( {\StL}\otimes \Sbb\bigr) (W)\otimes \Lambda^{q} V.
\]
Lemma \ref{LemmaConsequenceOfInductionHyothesis} and the assumptions of Lemma \ref{LemmaMain} imply that the map 
\[
\STup\colon L^{p,q}_k(\Tup_d) \lra S_k^{p,q}(V)
\]
is an isomorphism for $p+q\geq 2$ (for $k=d$ and $p=1$ we use the inductive assumption from the statement of Lemma~\ref{LemmaMain} in weights $<n$).

Our next goal is to show that 
\begin{equation}\label{FormulaProofOfMainLemma1}
\delta(a)\in \Dc_2\Omega^2\widetilde{\LL}^{\Lc}(\Tup_d).
\end{equation}
Assume the opposite and consider the smallest $k$ such that $\delta(a)\in \Dc_k\Omega^2\widetilde{\LL}^{\Lc}(\Tup_d);$ we have $k\geq 3$. By the choice of $k$, the image of $\delta(a)$ in
\[
\gr_k^{\Dc}\Omega^2\widetilde{\LL}^{\Lc}(\Tup_d)\cong  
L^{2,0}_k(\Tup_d)\oplus L^{1,1}_k(\Tup_d)
\]
is not equal to zero.  The projection of $\delta(a)$ to the summand $L^{2,0}_k(\Tup_d)$ equals to $\overline{\delta}(a)$ and vanishes by  (\ref{FormulaCoproductLiesInDepthTwo}).  Thus $\delta(a)$ lies in $L^{1,1}_{k}(\Tup_d)$. By the coJacobi identity, we have $\delta^{1,1}_1(\delta(a))=\delta^{1,1}_2(\delta(a))=0$. We have a commutative square
\[
\begin{tikzcd}[row sep = large, column sep = large]
L^{1,1}_{k}(\Tup_d)\arrow[d,"\STup"] \arrow[r,"\delta_1^{1,1}"] &
L^{2,1}_{k}(\Tup_d) \arrow[d,"\STup"]  \\
S^{1,1}_{k}(V)\arrow[r,"d_1^{1,1}"] &   S^{2,1}_{k}(V).
\end{tikzcd}
\]
The vertical maps are isomorphisms and the map $d_1^{1,1}=d\otimes \Id$ is injective by Proposition \ref{PropositionBarCobarCoLie}.  Thus the map $\delta^{1,1}_1$ is injective and $\delta(a)$ vanishes  in $\gr_k^{\Dc}\Omega^2\widetilde{\LL}^{\Lc}(\Tup_d)$. This contradicts the choice of $k$, so we have established (\ref{FormulaProofOfMainLemma1}).

Next, assume that the projection of $\delta(a)$ to $\gr_2^{\Dc}\Omega^2\widetilde{\LL}^{\Lc}(\Tup_d)$ equals to
\[
(a',a'')\in L^{2,0}_2(\Tup_d)\oplus L^{1,1}_2(\Tup_d).
\]
Consider the double complex
\[
\begin{tikzcd}[row sep = large, column sep = large]
L^{1,0}_{2}(\Tup_d)\arrow[d,"\delta_2^{1,0}"] \arrow[r,"\delta_1^{1,0}"] &
L^{2,0}_{2}(\Tup_d) \arrow[d,"\delta_2^{2,0}"] \\
L^{1,1}_{2}(\Tup_d) \arrow[d,"\delta_2^{1,1}"]\arrow[r,"\delta_1^{1,1}"] &  L^{2,1}_{2}(\Tup_d) \\
L^{1,2}_{2}(\Tup_d)& 
\end{tikzcd}
\]
By coJacobi, $\delta_1^{1,1}(a'')=-\delta_2^{2,0}(a'),$ and so $d_1^{1,1}(\STup(a''))=-d_2^{2,0}(\STup(a')).$ Lemma \ref{LemmaSecondDifferential} implies that
\[
\operatorname{Im}\bigl(d_2^{2,0}\bigr)\subseteq \bigoplus_{\substack{W\subseteq V \\ \dim(W)=2}} \!\! \left( \Omega^{2}\bigl( {\StL} \otimes \Sbb\bigr)(W) \otimes W\right).
\]
Thus,
$d_1^{1,1}(\STup(a''))$ lies in the subspace 
\[
\bigoplus_{\substack{W\subseteq V \\ \dim(W)=2}} \!\! \left( \Omega^{2}\bigl( {\StL} \otimes \Sbb\bigr)(W) \otimes W\right) \, \subseteq \, \,\, \bigoplus_{\substack{W\subseteq V \\ \dim(W)=2}} \!\! \Omega^{2}\bigl( {\StL} \otimes \Sbb\bigr)(W) \otimes V\,=\,S^{2,1}(V).
\]
We have a commutative diagram
\[
\begin{tikzcd}[row sep = huge]
0\arrow[r] & \displaystyle \hspace{-1em} \smash{\bigoplus_{\substack{W\subseteq V \\ \dim(W)=2}}} \!\!   \bigl( {\StL} \otimes \Sbb\bigr)(W)\otimes W  \arrow[hook, d,"d_1^{1,0}\otimes \Id", shorten <=1ex] \arrow[r] &
S^{1,1}_{2}(V) \arrow[hook, d,"d_1^{1,1}"]\arrow[r] & \displaystyle \hspace{-1em} \smash{\bigoplus_{\substack{W\subseteq V \\ \dim(W)=2}}} \!\!   \bigl( {\StL} \otimes \Sbb\bigr)(W)\otimes V/W   \arrow[hook, d,"d_1^{1,0}\otimes \Id", shorten <=1ex] \arrow[r] &0 \\
0\arrow[r] & \displaystyle \hspace{-1em}  \bigoplus_{\substack{W\subseteq V \\ \dim(W)=2}} \!\! \Omega^{2}\bigl( {\StL} \otimes \Sbb\bigr)(W) \otimes W \arrow[r] &  S^{2,1}_{2}(V) \arrow[r]  & \displaystyle \hspace{-1em} \bigoplus_{\substack{W\subseteq V \\ \dim(W)=2}} \!\! \Omega^{2}\bigl( {\StL} \otimes \Sbb\bigr)(W) \otimes V/W \arrow[r] & 0
\end{tikzcd}
\]
with exact rows and injective vertical maps, so the left square is a pullback square.
It follows that $\STup(a'')$ lies in the subspace 
\[
\bigoplus_{\substack{W\subseteq V \\ \dim(W)=2}} {\StL}(W) \otimes \Sbb^{\bullet}(W)\otimes W \,\subseteq\, S^{1,1}_2(V).
\]

We claim that there exists an element $\tilde{a}\in \Dc_2 \LL^{\Lc}(\Tup_d)$ such that $\overline{\delta}(\tilde{a})$ and $\overline{\delta}(a)$ coincide in $ L^{2,0}_{2}(\Tup_d)$. Since the truncated symbol map $\STup\colon  L^{1,0}_{2}(\Tup_d) \lra  S^{1,0}_{2}(V)$ is surjective, it is sufficient to find an element $s\in S^{1,0}_{2}(V)$ with $d_1^{1,0}(s)=\STup(a')$ and  $d_2^{1,0}(s)=\STup(a'')$. For that, notice that the complex 
\[
0\lra \!\!\!\! \bigoplus_{\substack{W\subseteq V \\ \dim(W)=2}} \!\!\! \StL(W) \otimes \Sbb^{\bullet}(W)  \xrightarrow{d_{2}^{1,0}}
\!\!\!\! \bigoplus_{\substack{W\subseteq V \\ \dim(W)=2}} \!\!\! \StL(W) \otimes \Sbb^{\bullet}(W)\otimes W
\xrightarrow{d_{2}^{1,1}} \!\!\!\! \bigoplus_{\substack{W\subseteq V \\ \dim(W)=2}} \!\!\! \StL(W) \otimes \Sbb^{\bullet}(W)\otimes \Lambda^2 V
\]
is exact. Since $d_{2}^{1,1}(\STup(a''))=0$, there exists a unique $s$ in
\[
\bigoplus_{\substack{W\subseteq V \\ \dim(W)=2}} \StL(W)\otimes \Sbb^{\bullet}(W)  = S^{1,0}_{2}(V)
\]
such that $d_{2}^{1,0}(s)=\STup(a'').$ Then we have
\[
d_2^{2,0}(d_1^{1,0}(s))=
-d_1^{1,1}(d_2^{1,0}(s))=-d_1^{1,1}(\STup(a''))=d_2^{2,0}(\STup(a')).
\]
Since $d_2^{2,0}$ is injective,  we have
\[
d_1^{1,0}(s)=d_1^{1,0}(\STup(a')).
\]
This proves the existence of an element $\tilde{a}\in \Dc_2 \LL^{\Lc}(\Tup_d)$ that we searched for.

Since $\overline{\delta}(a-\tilde{a})=0$, element $(a-\tilde{a})$  lies in $\Dc_1\Lc(\Tup_d)$ by Proposition \ref{PropositionWeakDepthConjecture}, and so lies in $\Dc_1\LL(\Tup_d)$  by Remark \ref{RemarkDepthFiltration}. Thus, $a$ lies in $\Dc_2\LL^{\Lc}(\Tup_d)$. This finishes the proof.\hfill\qedsymbol

\subsection{Proof of Theorem \ref{TheoremMain}}
\label{SectionLin11}

As explained in the introduction, we will prove Theorem \ref{TheoremMain} in the setting of the Hopf algebra of formal polylogarithms. The precise statement is as follows.

\begin{theorem} \label{TheoremMainFormal}
For any $n_1,\dots,n_d>0$ there exists $N>0$ such that $\Li_{n_1,n_2,\dots, n_d}(x_1^N,x_2^N,\dots,x_d^N)\in\Hc_n(\Q(\Tup_d))$ can be expressed as a polynomial with rational coefficients whose arguments are multiple polylogarithms of the following form:
    \[\Li_{m-k+1,\underbrace{\scriptstyle1,\dots,1}_{k-1}}(\zeta_1x_1^{a_{11}}\cdots x_d^{a_{d1}},\dots,\zeta_kx_1^{a_{1k}}\cdots x_d^{a_{dk}})\]
of depth $k\leq d$ and weight $m\leq n$, where $(a_{ij})\in \Z^{d\times k}$ has rank $k$ and $\zeta_i$ are roots of unity.
\end{theorem}

\begin{proof}
Given any nonzero vector $\lambda\in\Q^d$ consider the following element $\xi_{\lambda}\in \LL_{n}(\Tup_d)$:
    \[\xi_{\lambda} = \sum_{m_1+\dots+m_d=n}\Li_{m_1,m_2,\dots,m_d}(x_1,x_2,\dots,x_d)\lambda_1^{m_1-1}\cdots\lambda_d^{m_d-1},\]
where the sum is taken over all $d$-tuples with $m_i\in\Z_{>0}$. The mapping $\STup$ of Theorem~\ref{TheoremMain2} sends $\xi_{\lambda}$ to
    \[
    \STup(\xi_{\lambda}) = \sum_{m_1+\dots+m_d=n}\Lup[e_1,e_2,\dots,e_d] \otimes \frac{(\lambda_1e_1)^{m_1-1}\cdots (\lambda_de_d)^{m_d-1}}{(m_1-1)!\cdots (m_d-1)!} = \Lup[e_1,e_2,\dots,e_d] \otimes \frac{w^{n-d}}{(n-d)!},
    \]
where $w:=\lambda_1e_1+\dots+\lambda_de_d$ and in the second equality we used the multinomial theorem. By Corollary~\ref{thm:Stn_I_basis}, $\StH(V)$ is spanned by Coxeter pairs $[\Pc]\otimes[\Qc]$ with $P_1=Q_1=\langle w\rangle$. Moreover, by Proposition~\ref{PropositionNonGenericPairs}, any non-generic Coxeter pair is decomposable, so by Lemma~\ref{LemmaCoxeterParameterization} $\StL(V)$ is spanned by elements of the form $\Lup^{\Lc}[v_1,\dots,v_{d-1},w]$. Applying dihedral symmetry for $\Lup^{\Lc}$ from Corollary~\ref{CorollaryDihedral}, we get that elements of the form $\Lup^{\Lc}[w,v_2,\dots,v_{d}]$ also span $\StL(V)$. Therefore,
    \[
    \STup(\xi_{\lambda})^{\Lc} = \sum_{i}c_i\Lup^{\Lc}[w,v_2^{i},\dots,v_d^{i}] \otimes \frac{w^{n-d}}{(n-d)!}
    \]
for some rational numbers $c_i$. Then Theorem~\ref{TheoremMain2} implies that for some collection of $A_i\in\GL_d(\Q)$ we have
    \begin{equation} \label{EquationLinearCombinationLowerDepth}
    \xi_{\lambda}^{\Lc} - \sum_{i}c_i\,A_i\Li^{\Lc}_{n-d+1,1,\dots,1}(x_1,\dots,x_d) \in \mathcal{D}_{d-1}\Lup_n^{\Lc}(\Tup_d).
    \end{equation}
Since $\lambda$ was arbitrary, and since the image of $\lambda\mapsto(\lambda_1^{m_1-1}\cdots\lambda_d^{m_d-1})_{m_1+\dots+m_d=n}$ spans $\Q^{\binom{n-1}{d-1}}$, taking a suitable linear combination of~\eqref{EquationLinearCombinationLowerDepth} for sufficiently many different choices of $\lambda$, we get that
    \begin{equation} \label{EquationMPLLowerDepth}
    \Li_{n_1,\dots,n_d}^{\Lc}(x_1,\dots,x_d) - \sum_{j}c_j'\,A_j'\Li^{\Lc}_{n-d+1,1,\dots,1}(x_1,\dots,x_d) \in \mathcal{D}_{d-1}\Lup_n^{\Lc}(\Tup_d).
    \end{equation}
    
Choose a lift of~\eqref{EquationMPLLowerDepth} to $\Hc_n(\Tup_d)$. Recall that $\Hc_n(\Tup_d)$ is given as a filtered colimit of $\Hc_{\Tup'}$ over all isogenies $p\colon\Tup'\to\Tup_d$. Among all such isogenies, the isogenies $p_N\colon\Tup_{d}\to\Tup_d$, $N\in\Z_{>0}$, given by $p_N(x_1,\dots,x_d)=(x_1^N,\dots,x_d^N)$ form a cofinal subcategory. Therefore, there exists an $N>0$ such that the element
    \[p_N^{*}\Li_{n_1,\dots,n_d}(x_1,\dots,x_d) = \Li_{n_1,\dots,n_d}(x_1^N,\dots,x_d^N)\in \mathcal{H}_{\Tup_{d}}\] 
can be written as a linear combination of $\Li_{n-d+1,1,\dots,1}$, multiple polylogarithms of depth $\leq d-1$, and products of polylogarithms of weights $<n$ and depth $\leq d$. Applying the same argument recursively to each polylogarithm of depth $\leq d-1$ or weight $<n$ and depth $\leq d$, we get a polynomial expression using only elements of type $\Li_{m-k+1,1\dots,1}$, as claimed.
\end{proof}

\appendix
\section{Polylogarithms as functions on real tori}
In this appendix we outline an alternative form of Theorem~\ref{TheoremMain2}, by considering instead of $\LL_n(\Tup^d)$ a certain space of functions on the real torus $\R^d/\Z^d$, obtained from the power series expansion of multiple polylogarithms.

\subsection{The Steinberg module, cones, and partial fraction identities}
\label{sec:steinbergcones}
Let us describe a relation between the Steinberg module $\St_d(\Q)$, the algebra of rational polyhedral cones, and reciprocals of products of linear forms. The main references for this section are the papers of Khovanskii-Pukhlikov~\cite{KP92a,KP92b} and Brion-Vergne~\cite{BV99}.

Let $V$ be a finite-dimensional vector space over $\Q$, and let $d=\dim(V)$. We denote $V_{\R}=V\otimes_{\Q}\R$. A rational polyhedral cone is a subset of $V_{\R}$ of the form $C(v_1,\dots,v_N)=\R_{\geq0}v_1+\dots+\R_{\geq0}v_N$, $v_i\in V$. 
\begin{definition}
The cone algebra $\Cc(V)$ is the $\Q$-subspace of the space of $\Q$-valued functions on $V_{\R}$ generated by the indicator functions $[C]$ of all rational polyhedral convex cones $C\subseteq V_{\R}$. 
\end{definition}
Note that $\Cc(V)$ is generated by simplicial cones, i.e., cones of the form $C(v_1,\dots,v_k)$ where $v_1,\dots,v_k$ are linearly independent. Let $\Cc_0(V)$ be the subspace of $\Cc(V)$ spanned by cones of positive codimension, and let $\Cc_{\Lc}(V)$ be the subspace spanned by cones containing a line in some direction; we also set $\Cc_{0,\Lc}(V)=\Cc_{0}(V)+\Cc_\Lc(V)$.

\begin{proposition} \label{PropositionSteinbergConesLinearForms}
\begin{enumerate}
\item\label{PropositionSteinbergConesLinearFormsParti} We have the following exact sequence
	 \[
	0\longrightarrow \Cc_{0,\Lc}(\Q^d) \longrightarrow \Cc(\Q^d)\longrightarrow\St_{d}(\Q)\longrightarrow 0,
	\]
where the mapping from $\Cc(\Q^d)$ to $\St_{d}(\Q)$ sends $[C(v_1,\dots,v_d)]$ to $\sgn\det(v_1,\dots,v_d)[v_1,\dots,v_d]$ and it sends any lower-dimensional cone to $0$.

\item\label{PropositionSteinbergConesLinearFormsPartii} The $\Q$-linear map $\rho\colon\St_{d}(\Q) \to \Q(z_1,\dots,z_d)$ sending
	\[[v_1,\dots,v_d] \mapsto \frac{\det(v_1,\dots,v_d)}{\prod_{j=1}^{d}\langle v_j, z\rangle}\]
	is well-defined and injective.
\end{enumerate}
\end{proposition}
\begin{proof}
Let us prove part \ref{PropositionSteinbergConesLinearFormsPartii} first. That the map is well-defined follows from the determinant identity
	\[\sum_{i=0}^{d}(-1)^i\det(v_0,\dots,\widehat{v_i},\dots,v_d)\,v_i = 0\]
that holds for any $v_0,\dots,v_d\in \Q^d$. Next, we prove injectivity. For $A\in\GL_d(\Q)$ we denote $[A]:=[Ae_1,\dots,Ae_d]\in\St_d(\Q)$, where $e_1,\dots,e_d$ is the standard basis of $\Q^d$. By Lemma~\ref{LemmaSteinbergBasis} it suffices to show that the elements $\rho([A])$, $A\in U_d(\Q)$ are linearly independent, where $U_d(\Q)$ is the set of unit upper triangular matrices. Note that for $A\in U_d(\Q)$ we have
    \begin{equation} \label{eq:partialfractioninduction}
    \rho([A]) = \frac{\rho([A_{d-1}])}{z_1a_{1,d}+\dots+z_{d-1}a_{d-1,d}+z_d},
    \end{equation}
where $A_{d-1}$ is the upper left $(d-1)\times (d-1)$-submatrix of $A=(a_{ij})$. Since rational functions $(z_1\alpha_{1}+\dots+z_{d-1}\alpha_{d-1}+z_d)^{-1}$, $\alpha_i\in\Q$ are $\Q(z_1,\dots,z_{d-1})$-linearly independent in $\Q(z_1,\dots,z_{d})$, linear independence of $\rho([A])$, $A\in U_d(\Q)$ trivially follows from~\eqref{eq:partialfractioninduction} by induction on $d$.

For part \ref{PropositionSteinbergConesLinearFormsParti} we first recall from~\cite[Prop. 4.1]{KP92b} that there exists a unique $\Q$-linear mapping $\Ic\colon \Cc(\Q^d)\to\Q(z_1,\dots,z_d)$, such that for any proper cone $C$ and $z$ in the interior of the dual cone $C^{*}$ we have $\Ic([C])(z)=\int_{C}e^{-\langle z,x\rangle} dx$. Taking $C=C(v_1,\dots,v_d)=A\R_{\geq0}^d$, where $A\in\GL_d(\Q)$ satisfies $Ae_i=v_i$, we compute 
    \[\Ic([C(v_1,\dots,v_d)]) = \int_{C}e^{-\langle z,x\rangle}dx = \lvert\det(A)\rvert\int_{\R_{\geq0}^d}e^{-\langle A^Tz,y\rangle}dy=\frac{\lvert\det(v_1,\dots,v_d)\rvert}{\langle v_1,z\rangle\dots \langle v_d,z\rangle}.\]
From this computation we see that the images of $\Ic$ and $\rho$ coincide. This implies that the mapping in the statement of part \ref{PropositionSteinbergConesLinearFormsParti} agrees with $\rho^{-1}\circ \Ic$ on simplicial cones, hence it is well-defined and surjective. Finally, from~\cite[Prop. 4.1 (A) (iii)]{KP92b},~\cite[Thm. 2.1]{KP92b}, and injectivity of Laplace transform (on piecewise constant functions in the region of convergence) it follows that $\Ker(\Ic)=\Cc_{0,\Lc}(\Q^d)$, proving the claim.
\end{proof}

The image of the map $\rho$ from Proposition~\ref{PropositionSteinbergConesLinearForms} consists of rational functions of degree $-d$ that vanish at infinity and can be written as linear combinations of products of reciprocals of linear forms (see~\cite[\S2]{BV99}). By~\cite[Theorem~1]{BV99} the image of $\rho$ generates a free $\Sbb(\Q^d)$-module, where the symmetric algebra acts on $\Q(z_1,\dots,z_d)$ as differential operators with constant coefficients. From this we get the following generalization, where we also apply duality to the element in the Steinberg module (see~\S\ref{SectionCoxeterPairs}).
\begin{corollary} \label{SteinbergPartialFractionsWeightN}
The mapping of $\Q$-vector spaces $\rho\colon\St_{d}(\Q)\otimes\Sbb(\Q^d) \to \Q(z_1,\dots,z_d)$ sending
	\[[v_1,\dots,v_d]\otimes \prod_{i=1}^{d}\frac{v_i^{n_i-1}}{(n_i-1)!} \mapsto \frac{\det(v^1,\dots,v^d)}{\prod_{i=1}^{d}\langle v^i, z\rangle^{n_i}},\]
where $v^1,\dots,v^d$ is a basis dual to $v_1,\dots,v_d$, is well-defined and injective. 
\end{corollary}
From~\cite[Theorem~1]{BV99} it also follows that the image of the above mapping is spanned by all functions of the form 
	\[
	\varphi(z) = \prod_{i=1}^{N}\langle u_i,z\rangle^{-n_i},
	\] 
where $n_i\geq1$ and $u_1,\dots,u_N$ span $\Q^d$. Finally, we remark that in coordinate-free notation, the mapping of Corollary~\ref{SteinbergPartialFractionsWeightN} goes from $\St(V)\otimes \Sbb(V)$ to $\Lambda^d V^{\vee}\otimes \operatorname{Frac}(\Sbb(V^{\vee}))$, and it is easy to see that it is a morphism of $\GL(V)$-modules.

\subsection{Polylogarithms as distributions on real tori}
\label{sec:polylogdistributions}
We assume familiarity with the theory of distributions, but we very briefly recall the basic setup in the special case of distributions on tori (for details see, e.g.,~\cite{khavin1991commutative}).

For a free abelian group $\Lambda$ of rank $d$, let $\TT_{\Lambda}$ denote the real torus $\Lambda_{\R}/\Lambda$, where $\Lambda_{\R}\coloneqq \Lambda\otimes \R$. We will denote $\TT_{\Z^d}=\R^d/\Z^d$ by $\TT^d$. Any distribution $F$ on $\TT_{\Lambda}$ can be formally given as a Fourier series $F(x)=\sum_{\nu\in\Lambda^{\vee}}\widehat{F}(\nu)e^{2\pi i \nu(x)}$ whose coefficients $\widehat{F}(\nu)$ grow at most polynomially as $|\nu|\to\infty$; the value of $F$ (viewed as a linear functional on the space of test functions) on $\varphi\in C^{\infty}(\TT_{\Lambda})$ is then defined as $\sum_{\nu\in\Lambda^{\vee}}\overline{\widehat{F}(\nu)}\widehat{\varphi}(\nu)$. We may therefore think of distributions on $\TT_{\Lambda}$ as slowly growing functions on $\Lambda^{\vee}:=\Hom_{\Z}(\Lambda,\Z)$.

We will consider only homogeneous distributions, which we now define. For $n\in\Z$ we let $\Sc_{n}'(\TT_{\Lambda})$ be the $\Cbb$-vector space of distributions $F$ on $\TT_{\Lambda}$ with $\widehat{F}\colon \Lambda^{\vee}\to\Cbb$ homogeneous of degree $-n$. This space has the following alternative description:
    \[F\in \Sc_{n}'(\TT_{\Lambda})\qquad \Leftrightarrow\qquad 
    \sum_{v\in \Lambda/N\Lambda}F\Big(\frac{x+v}{N}\Big) = N^{d-n}F(x),\quad \mbox{for all }N\in\Z_{>0}.\]
Indeed, letting $D_NF(x) \coloneqq \sum_{v\in \Lambda/N\Lambda}F(\frac{x+v}{N})$, we compute
$\widehat{D_NF}(\nu) = N^{d}\widehat{F}(N\nu)$, which implies the above equivalence. The same computation also easily implies the following result.
\begin{lemma} \label{LemmaDistributionConmeasurable}
	Let $\Lambda\subseteq\Lambda'$ be a pair of free abelian groups of rank $d$. Then for any positive integer $N$ divisible by $|\Lambda'/\Lambda|$ the mapping $F\mapsto N^{n-d}D_NF$ gives an isomorphism between $\Sc_{n}'(\TT_{\Lambda'})$ and $\Sc_{n}'(\TT_{\Lambda})$, and moreover this isomorphism does not depend on the choice of $N$.
\end{lemma}

Let us also note that integer matrices with nonzero determinant act on $\Sc_{n}'(\TT^d)$ via
\begin{equation*} \label{DistributionsAction}
(AF)(x)\coloneqq N^{n-d-1}\sum_{v\in \Z^d/N\Z^d}F\Big(\frac{x+v}{N}A\Big),\qquad N=\lvert\det(A)\rvert,
\end{equation*}
and this action extends to a $\GL_d(\Q)$-action by homogeneity in the same way as in Lemma~\ref{LemmaGroupActionExtension}. 

Multiple polylogarithms $\Li_{n_1,\dots,n_k}(e^{2\pi i x_1},\dots,e^{2\pi i x_k})$ for $k\leq d$ are elements of $\Sc_{n}'(\TT^d)$, $n=n_1+\dots+n_k$. Here we interpret $F(x)=\Li_{n_1,\dots,n_k}(e^{2\pi i x_1},\dots,e^{2\pi i x_k})$ for $x\in\TT^d$ as a limit in the space of distributions of $F_{\varepsilon}(x)=F(x_1+i\varepsilon,\dots,x_d+i\varepsilon)$ as $\varepsilon\to 0+$, where $F_{\varepsilon}\in C^{\infty}(\TT^d)$ for all $\varepsilon>0$. The Fourier coefficients of $F$ are given by
    \[\widehat{F}(m_1,\dots,m_d) = \frac{1}{m_1^{n_1}\cdots m_k^{n_k}},\qquad 0<m_1<\dots<m_k,\; m_{k+1}=\dots=m_d=0,
    \]
and $0$ otherwise. Note that for polylogarithms of depth $k<d$ the support of $\widehat{F}$ is contained in a proper subspace of $\Z^d$. More generally, we consider distributions $F=F_{C,u,n_1,\dots,n_d}\in\Sc'_n(\TT_{\Lambda})$ given by
    \begin{equation} \label{eq:Lidistributions}
    \widehat{F}(\nu) = 
    \begin{cases}
    \displaystyle
    \frac{1}{\prod_{j=1}^{d}\langle u_j,\nu\rangle^{n_j}},
    & \nu\in C\cap \Lambda^{\vee},\; \langle u_1,\nu\rangle,\dots,\langle u_d,\nu\rangle\ne 0, \\[0.5cm]
    0, & \mbox{ otherwise},
    \end{cases}
    \end{equation}
where $C$ is a rational convex cone in $\Lambda^{\vee}_{\R}$, $u_1,\dots,u_d$ is any basis of $\Lambda_{\Q}$, and $n_1,\dots,n_d\geq1$.

\begin{definition} \label{DefinitionPolylogDistributions}
The space $\Sc_{\Lup,n}'(\TT_{\Lambda})$ is the quotient of the $\Q$-subspace of $\Sc_n'(\TT_{\Lambda})$ spanned by all distributions as in~\eqref{eq:Lidistributions} with $n=n_1+\dots+n_d$ by the $\Q$-subspace spanned by distributions $F_{C,u,n_1,\dots,n_d}$ over all cones $C$ with $\dim(C)<d$.
\end{definition}
This space of distributions has the following algebraic description.

\begin{proposition} \label{PropositionDistributionsIsomorphism}
For $n\geq d$ the mapping
    \[[C]\otimes [w_1,\dots,w_d]\otimes\prod_{j=1}^{d}\frac{w_{j}^{n_j-1}}{(n_j-1)!} 
    \longmapsto\sum_{\substack{\nu\in C\cap\Z^d\\ \langle \nu,w_j\rangle \ne 0}}\frac{\det(w^1,\dots,w^d)e^{2\pi i \langle\nu,x\rangle}}{\prod_{j=1}^{d}\langle w^j,\nu\rangle^{n_j}}\] 
gives a well-defined isomorphism 
    \[\big(\Cc(\Q^d)/\Cc_0(\Q^d)\big)\otimes \St_d(\Q)\otimes \Sbb^{n-d}(\Q^d)
    \lra
    \Sc'_{\Lup,n}(\TT^{d}).\]
\end{proposition}
\begin{proof}
Let $R_{d,n}$ be the $\Q$-vector space of functions on $\Q^d$ that are piecewise rational with respect to some complete rational polyhedral fan, and whose rational pieces are homogeneous of degree $-n$ and all belong to the image of the map $\rho$ from Corollary~\ref{SteinbergPartialFractionsWeightN}. At the points where rational functions are not defined, we set the value as $0$ and we consider two elements of $R_{d,n}$ to be equal if they agree outside a union of finitely many hyperplanes. Corollary~\ref{SteinbergPartialFractionsWeightN} implies that the mapping
    \[[C]\otimes [u_1,\dots,u_d]\otimes\prod_{j=1}^{d}\frac{w_{j}^{n_j-1}}{(n_j-1)!} 
	\longmapsto [C](\nu)\cdot \frac{\det(w^1,\dots,w^d)}{\prod_{j=1}^{d}\langle w^j,\nu\rangle^{n_j}}\] 
is an isomorphism of $\big(\Cc(\Q^d)/\Cc_0(\Q^d)\big)\otimes \St_d(\Q)\otimes \Sbb^{n-d}(\Q^d)$ onto $R_{d,n}$. It remains to note that the mapping that sends $\varphi\in R_{d,n}$ to $\sum_{\nu\in\Z^d} \varphi(\nu) e^{2\pi i \langle \nu, x\rangle}$ is an isomorphism onto $\Sc'_{\Lup,n}(\TT^{d})$ since a homogeneous function on $\Q^d$ is uniquely determined by its restriction to any lattice.
\end{proof}

We define a filtration $\Fc$ on $\Sc_{\Lup,n}'(\TT_{\Lambda})$ by setting $\Fc_{k}\Sc_{\Lup,n}'(\TT_{\Lambda})$ to be the $\Q$-linear span of all distributions~\eqref{eq:Lidistributions} for which $C$ contains a $(d-k)$-dimensional linear subspace. From Proposition~\ref{PropositionDistributionsIsomorphism} and part \ref{PropositionSteinbergConesLinearFormsParti} of Proposition~\ref{PropositionSteinbergConesLinearForms} we immediately get the following result.
\begin{corollary}
For $n\geq d$ the mapping
	\[[v_1,\dots,v_d]\otimes [w_1,\dots,w_d]\otimes\prod_{j=1}^{d}\frac{w_j^{n_j-1}}{(n_j-1)!} 
	\longmapsto\frac{\sgn\det(v_1,\dots,v_d)}{\det(w_1,\dots,w_d)}\sum_{\substack{\nu\in C\cap\Z^d\\ \langle w_j,\nu\rangle \ne 0}}\frac{e^{2\pi i \langle\nu,x\rangle}}{\prod_{j=1}^{d}\langle w^j,\nu\rangle^{n_j}},\]
where $C=C(v_1,\dots,v_d)$, and $w^1,\dots,w^d$ is the basis dual to $w_1,\dots,w_d$, gives a well-defined isomorphism 
\[\St_d(\Q)\otimes \St_d(\Q)\otimes \Sbb^{n-d}(\Q^d)
\lra
\gr_d^{\Fc}\Sc'_{\Lup,n}(\TT^{d}).\]
\end{corollary}
This is the analogue of Theorem~\ref{TheoremMain2} that we were after. As a corollary, we see that, as a $\GL_d(\Q)$-module, $\gr_d^{\Fc}\Sc'_{\Lup,n}(\TT^{d})$ is generated by multiple polylogarithms $\Li_{n_1,\dots,n_d}(e^{2\pi i x_1},\dots,e^{2\pi i x_d})$, $n_1+\dots+n_d=n$, and we also recover the formula for lattice Aomoto polylogarithms $\Li_{n_1,\dots,n_d}(v_1,\dots,v_d;w_1,\dots,w_d;x)$ from the introduction. To make the analogy between $\LL_n(\Tup^d)$ and $\Sc'_{\Lup,n}(\TT^{d})$ more direct, it remains to identify $\Fc_{d-1}\Sc'_{\Lup,n}(\TT^{d})$ with a subspace of distributions given by ``products with $\log(z_i)$''. First let us look at the 1-dimensional torus.

\begin{example} \label{ExampleLogsDistributions}
	The space $\Sc'_{\LL,n}(\TT^1)$ has a basis consisting of $\Li_{n}(e^{2\pi i x})$, and $\Li_{n}(e^{-2\pi i x})$, $n\geq1$. The only nontrivial part of the filtration is the space $\Fc_{0}\Sc'_{\LL,n}(\TT^1)$, spanned by
	\[\sum_{\nu\ne0}\frac{e^{2\pi i \nu x}}{\nu^n} = -\frac{(2\pi i)^n}{n!}B_n(\{x\}),\qquad n\geq 1,\]
	where $B_n(t)$ is the $n$-th Bernoulli polynomial and $\{x\}$ denotes the fractional part of $x$.
	Since for $0<x<1$ we have $2\pi i B_1(\{x\})=\log(-e^{2\pi i x})$, we see that any element of $\Fc_{0}\Sc'_{\LL,n}(\TT^1)$ is a polynomial in $\log(-e^{2\pi i x})$ and $2\pi i$ (for $0<x<1$).
\end{example}

 Distributions cannot be multiplied in general, however there is a well-defined external product map $\Sc'_{\Lup,n_1}(\TT_{\Lambda_1})\otimes \Sc'_{\Lup,n_2}(\TT_{\Lambda_2}) \to \Sc'_{\Lup,n_1+n_2}(\TT_{\Lambda_1\oplus \Lambda_2})$ given by $(F\cdot G)(x_1,x_2) = F(x_1)G(x_2)$ for $(x_1,x_2)\in \Lambda_{1,\R}\oplus \Lambda_{2,\R}$.
 On the level of Fourier expansions we have $\widehat{F\cdot G}(\nu_1,\nu_2)=\widehat{F}(\nu_1)\widehat{G}(\nu_2)$, $\nu_1\in\Lambda_1^{\vee}$, $\nu_2\in\Lambda_2^{\vee}$. Similarly, if $\Lambda_1^{\vee}\oplus \Lambda_2^{\vee}\subseteq \Lambda^{\vee}$, then we have a natural product map $\Sc'_{\Lup,n_1}(\TT_{\Lambda_1})\otimes \Sc'_{\Lup,n_2}(\TT_{\Lambda_2}) \to \Sc'_{\Lup,n_1+n_2}(\TT_{\Lambda})$ obtained by composition of the external product with the isomorphism from Lemma~\ref{LemmaDistributionConmeasurable}. Note that by Proposition~\ref{PropositionDistributionsIsomorphism} we have $ \Sc'_{\Lup,n}(\TT_{\Lambda})\cong \Cc(V)/\Cc_0(V)\otimes \St(V)\otimes \Sbb^{n-d}(V)$, where $V=\Lambda_{\Q}^{\vee}$, and the external product is induced by the usual VB-monoid structure for $\St\otimes\Sbb$ (see Definition~\ref{DefinitionTensorS}), and the external product $\Cc(V)\otimes\Cc(W)\to \Cc(V\oplus W)$ sending $[C_1]\otimes [C_2]$ to $[C_1+C_2]$ for the cone algebra. 

\begin{proposition} \label{LogarithmicFiltration1}
	The natural map
	\[\bigoplus_{m=1}^{n-d+1}\bigoplus_{L^{\vee}\oplus \Lambda_1^{\vee} \subseteq\Lambda^{\vee}}\Fc_{0}\Sc'_{\Lup,m}(\TT_{L})\otimes \Sc'_{\Lup,n-m}(\TT_{\Lambda_1})\lra \Fc_{d-1}\Sc'_{\Lup,n}(\TT_{\Lambda}),\] 
	where $L^{\vee}$ runs over all rank $1$ sublattices in $\Lambda^{\vee}$, is surjective. 
\end{proposition}
\begin{proof}
In view of the above discussion, the result would follow from the following claim: the image of $\Cc_\Lc(V)\otimes \St(V)$ in $\Cc(V)\otimes \St(V)$ is a sum of decomposable elements coming from decompositions of $V$ into a sum of $1$-dimensional and $(d-1)$-dimensional subspaces. To prove the claim, let $C\subseteq V$ be a cone containing a line $L$.
By Proposition~\ref{PropositionSteinbergSubspace}
\begin{equation*}
	\bigoplus_{V=L\oplus W} \St(L)\otimes \St(W) \stackrel{\cong}{\lra} \St(V),
\end{equation*}
where $W$ runs over $(d-1)$-dimensional subspaces of $V$ complementary to $L$. But for any such $W$ we have $C=L+C\cap W$, so that $[C]$ lies in the image of $\Cc(L)\otimes \Cc(W)$, and thus the claim is proven.
\end{proof}
Together with the computation from Example~\ref{ExampleLogsDistributions} this proposition shows that $\Fc_{d-1}\Sc'_{\Lup,n}(\TT^{d})$ is indeed spanned by products with $\log(z_i)$ and $2\pi i$.

\begin{remark}
Although we treat elements of $\Sc'_{\Lup,n}(\TT_{\Lambda})$ as abstract distributions, they are in fact well-behaved functions that are integrable on~$\TT^d$ and smooth on a complement of a union of finitely many subtori. For example, for $\Li_{1,\dots,1}(e^{2\pi i x_1},\dots,e^{2\pi i x_d})$ this can be seen from the integral representation
	\[\Li_{1,\dots,1}(z_1,\dots,z_d) = \int_{0<t_1<\dots<t_d<1}\frac{dt_1}{(z_1\cdots z_d)^{-1}-t_1}\wedge\dots \wedge \frac{dt_d}{z_d^{-1}-t_d},\]
using the fact that $\int_{0}^{1}\frac{dt}{|t-z|}$ is bounded by $C\log\frac{1}{|1-z|}$ for $z$ close to $1$.
\end{remark}

\bibliographystyle{halpha2}
\bibliography{bibliography} 

\end{document}